\documentclass[11pt]{article}
\usepackage[a4paper,margin=2.5cm]{geometry}
\usepackage[T1]{fontenc}
\usepackage[utf8]{inputenc}
\usepackage{lmodern}
\usepackage{amssymb,amsmath,amsthm,mathrsfs,dsfont}
\usepackage{bbm}
\usepackage[colorlinks,allcolors=magenta!60!black]{hyperref}
\usepackage{tikz}
\usetikzlibrary{arrows.meta}
\usepackage{hyperref}
\usepackage{empheq}
\usepackage{verbatim}

\newtheorem{thm}{Theorem}[section]
\newtheorem{lem}[thm]{Lemma}
\newtheorem{prop}[thm]{Proposition}
\newtheorem{corr}[thm]{Corollary}

\DeclareMathOperator{\var}{Var}

\let\P\relax\DeclareMathOperator{\P}{\mathbb{P}} 
\DeclareMathOperator{\E}{\mathbb{E}}

\def\R{\mathbb{R}}

\def\N{\mathbb{N}}

\def\diffd{\mathrm{d}}

\newcommand{\indic}[1]{\mathds{1}\raisebox{-.4ex}{$\scriptstyle\{#1\}$}}

\def\uppar#1{^{\scriptscriptstyle(#1)}}

\def\setN{\mathcal N\uppar N}
\def\set+{\mathcal N^{+}}

\def\vl{v^\ell }

\def\B{{\mathcal B}}

\newcommand{\p}[1]{{\mathbb P}\left(#1\right)}
\newcommand{\psub}[2]{{\mathbb P}_{#1}\left(#2\right)}
\newcommand{\Esub}[2]{{\mathbb E}_{#1}\left[#2\right]}

\def\ZZint#1#2#3{{\setbox0=\hbox{$#1{#2#3}{\int}$}
\vcenter{\hbox{$#2#3$}}\kern-.5\wd0}}

\begin{document}

\allowdisplaybreaks

\author{Julien Berestycki\footnote{Department of Statistics, University of Oxford, UK}, Éric Brunet\footnote{Laboratoire de Physique de l’\'Ecole normale sup\'erieure, ENS, Universit\'e PSL, CNRS, Sorbonne Universit\'e, Universit\'e de Paris, F-75005 Paris, France}, James Nolen\footnote{Department of Mathematics, Duke University Box 90320, Durham, NC 27708, USA}, Sarah Penington\footnote{Department of Mathematical Sciences, University of Bath, UK}}

\title{Brownian bees in the infinite swarm limit}

\date{\today}
\maketitle

\begin{abstract}
The {\it Brownian bees} model is a branching particle system with spatial selection. It is a system of $N$ particles which move as independent  Brownian motions in $\R^d$ and independently branch at rate 1, and, crucially, at each branching event, the particle which is the furthest away from the origin is removed to keep the population size constant. In the present work we prove that as $N \to \infty$ the behaviour of the particle system is well approximated by the solution of a free boundary problem (which is the subject of a companion paper \cite{BBNP1}), the  {\it hydrodynamic limit} of the system. We then show that for this model the so-called {\it selection principle} holds, i.e.\@ that as $N \to \infty$ the equilibrium density of the particle system converges to the steady state solution of the free boundary problem.
\end{abstract}


\section{Introduction and main results}

The {\it Brownian bees} model is a particular case of an $N$-particle branching Brownian motion ($N$-BBM for short) which is defined as follows. 
The system consists of $N$ particles with locations in $\R^d$ for some dimension $d$.
Each particle moves independently according to a Brownian motion with diffusivity $\sqrt 2$ and branches independently into two particles at rate one.  
Whenever a particle in the system branches, the particle in the system which is furthest (in Euclidean distance) from the origin is immediately removed from the system, so that there are exactly $N$ particles in the system at all times. Thus, the branching events arrive according to a Poisson process with rate $N$. 
The name {\em Brownian bees}, suggested by Jeremy Quastel, comes from the analogy with bees swarming around a hive; throughout the paper, we will refer to this process simply as $N$-BBM. 

The particles can be labelled in a natural way, which will allow us to write the $N$-BBM as a c\`adl\`ag $(\R^d)^N$-valued process.
Each of the $N$ particles carries a label from the set $\{1,\ldots, N\}$.
Suppose at some time $\tau$ the particle labelled $k$ branches, and the particle with label $\ell$ is the furthest from the origin.
Then at time $\tau$ the particle with label $\ell$ is removed from the system, and a new particle with label $\ell$ appears at the location of the particle with label $k$.
For $k\in \{1,\ldots, N\}$, let $X\uppar N _k(t)$ denote the location of the particle with label $k$ at time $t$.
Then
\[
X\uppar N(t)=\big(X\uppar N_1(t),\ldots , X\uppar N_N(t)\big)
\]
is the vector of particle locations in the $N$-BBM at time $t$.
This labelling is motivated by the following equivalent description of $N$-BBM in terms of jumping rather than branching: at rate $N$, the particle that is furthest from the origin jumps to the location of another particle chosen uniformly at random from all $N$ particles. (The particular choice of ordering of the particles here plays no essential role in our results.)

We shall prove two types of result about this interacting system of particles: results about the spatial distribution of particles at a fixed time $t$ as the number of particles $N\to \infty$, and results about the long-term behaviour of the particle system as time $t\to \infty$ for a fixed large number of particles $N$.  As we show, there is a sense in which these limits commute.  Showing this so-called {\it selection principle} holds is a major motivation of the present work and was originally conjectured by Nathana\"{e}l Berestycki. More precisely, he predicted that as $N\to \infty$, the particles would localise in a ball of finite radius at large times.  

Our first main result is a hydrodynamic limit for the distribution of particle locations at a fixed time $t$ as the number of particles $N\to \infty$. This limit involves the solution of the following free boundary problem: for a probability measure $\mu_0$ on $\R^d$, find $u(x,t)
: \R^d\times (0,\infty) \to [0,\infty)$ and $R_t:(0,\infty) \to [0,\infty]$  such that
\begin{equation}
\begin{cases}
\partial_t u =  \Delta u  +u ,  & \text{for $t>0$ and  $\|x\| < R_t$},   \\
u(x,t)=0, &\text{for $t>0$ and $\|x\|\ge R_t$}, \\
u(x, t) \quad \text{is continuous on $\R^d \times (0,\infty)$}, \\[1ex]
 \displaystyle \int_{\R^d} u(x,t)\, \diffd x =1, & \text{for } t > 0,
    \\[1ex]
u(\cdot,t)\to \mu_0 &\text{weakly as $t \searrow 0$}.
\end{cases}
\label{pbu_probpaper}
\end{equation}
In the companion paper~\cite{BBNP1}, we prove that~\eqref{pbu_probpaper} has a unique solution $(u,R)$, and that the function $R_t$ is finite and continuous for $t>0$. 

For $t\ge 0$, we let
$$
M\uppar N_t=\max_{i\in \{1,\ldots,N\}}\Big\|X_i\uppar N(t)\Big\|
$$
denote the maximum distance of a particle from the origin at time $t$. For
$A\subseteq \R^d$ measurable, we let
$$\mu\uppar N(A,t)=\frac 1N\Big|\Big\{i\in\{1,\ldots,N\}\;:\;
X_i\uppar N(t)\in A\Big\}\Big|$$
denote the proportion of particles which are in the set $A$ at time $t$.
In other words, $\mu\uppar N(\diffd x,t)$ is the empirical measure of the particles at time $t$, i.e.
\[
\mu\uppar N(\diffd x,t) = \frac{1}{N} \sum_{k=1}^N \delta_{X\uppar N_k(t)}(\diffd x).
\]
We can now state our hydrodynamic limit result.
\begin{thm}\label{thm:NBBM}
Suppose that $\mu_0$ is a Borel probability measure on $\R^d$, and that
\begin{itemize}
\item $X\uppar N_1(0),\ldots,X\uppar N_N(0)$ are i.i.d.~with distribution
given by $\mu_0$, and
\item $(u,R)$ is the solution to~\eqref{pbu_probpaper} with initial condition
$\mu_0$.
\end{itemize}
Then, for any $t>0$ and any measurable $A\subseteq \R^d$, almost surely,
$$
\mu\uppar N (A,t) \to \int_A u(x,t)\,\diffd x \quad \text{and} \quad M\uppar N _t \to R_t \quad \text{ as }N \to \infty
$$
(this holds for any coupling of the processes $(X\uppar N)_{N\in \N}$).
\end{thm}
Note that 
Theorem~\ref{thm:NBBM} implies that for $t>0$, almost surely
$\mu\uppar N (\diffd x,t) \to u(x,t)\,\diffd x$ weakly as $N\to \infty$.

Our second set of results concerns the long-term behaviour ($t \to \infty$) of the particle system for large $N$. We can show that for large fixed $N$, the particle system converges in distribution as $t\to \infty$ to an invariant measure.  For $\mathcal X \in (\R^d)^N$, we write $\mathbb P_{\mathcal X}$ to denote the probability measure under which $(X\uppar N (t),t\ge 0)$ is an $N$-BBM process with $X\uppar N (0)=\mathcal X$.

\begin{thm} \label{thm:piNexists}
For $N$ sufficiently large, the process $(X\uppar N(t),t\geq 0)$ has a unique invariant measure $\pi \uppar N$, a probability measure on $(\R^d)^N$.
For any $\mathcal X\in (\R^d)^N$, under $\mathbb P_{\mathcal X}$,
the law of
$X\uppar N(t)$ converges in total variation norm to $\pi\uppar N$ as $t\to \infty$:
\[
\lim_{t \to \infty} \sup_{C} \big|\psub{\mathcal X}{X\uppar N(t) \in C}
- \pi\uppar N (C)\big| = 0,
\]
where the supremum is over all Borel measurable sets $C \subseteq (\R^d)^N$. 
\end{thm}

For each $t \geq 0$, the empirical measure $\mu \uppar N(\cdot, t)$ is a random element of $\mathcal{P}(\R^d)$, the set of Borel probability measures on $\R^d$.  Theorem~\ref{thm:piNexists} implies that as $t \to \infty$  the law of $\mu \uppar N(\cdot, t)$ converges in total variation to the measure  $\pi\uppar N \circ H^{-1}$, where $H:(\R^d)^N \to \mathcal{P}(\R^d)$ is the map defined by $H(x_1,\dots,x_N) = \frac{1}{N} \sum_{i=1}^N \delta_{x_i}$ and $\pi\uppar N \circ H^{-1}$ is the pushforward of $\pi \uppar N$ under the map $H$.  The law of $\mu \uppar N(\cdot, t)$ and the measure $\pi\uppar N \circ H^{-1}$, which are both probability measures on the Polish space $\mathcal{P}(\R^d)$, do not depend on the particular ordering of particles used to define $X \uppar N(t)$ as an $(\R^d)^N$-valued process.


We also obtain more explicit results about the long-term behaviour of the particle system.
We let $U:\R^d \to \R$ denote the principal Dirichlet eigenfunction of ($-\Delta$) in a spherical domain with radius uniquely chosen so that the eigenvalue is $1$. That is, let $(U,R_\infty)$ denote the unique solution to
\begin{equation}
\begin{cases}
- \Delta U(x)  = U(x) ,  & \|x\| <  R_\infty,   \\
U(x) > 0, \quad &  \| x\| < R_\infty, \\
U(x)=0, \quad &  \|x\|\ge R_\infty, \\
\int_{\|x \| \leq R_\infty} U(x) \,\diffd x = 1.
\end{cases} \label{eqU_probpaper}
\end{equation}
Then $(U,R_\infty)$ is a stationary solution to~\eqref{pbu_probpaper}.
In~\cite{BBNP1}, we prove that any solution $(u(\cdot,t),R_t )$ of the free boundary problem~\eqref{pbu_probpaper} converges to the stationary solution $(U,R_\infty)$ as $t\to \infty$, and it turns out that this stationary solution also controls the long-term behaviour of the particle system for large $N$.

We shall use the following notation to denote a reasonable class of initial particle configurations.
For $K > 0$ and $c \geq 0$, let
\begin{equation}
\Gamma(K, c) = \bigg\{  \mathcal X\in (\R^d)^N\;:\;\frac1N \Big|\big\{ i: \| \mathcal X_i  \| < K\big \}\Big| \geq c\bigg\}. 
\label{Gammaxcdef}
 \end{equation}
This is the set of particle configurations which put at least a fraction $c$ of the particles within distance $K$ of the origin.
The following result shows that if $N$ is large, then at a large time~$t$, the particles are approximately distributed according to $U$, and the largest particle distance from the origin is approximately $R_\infty$.
\begin{thm} \label{thm:sspd}
Take $K>0$ and $c\in(0,1]$.
For $\epsilon>0$, there exist $N_\epsilon<\infty$ and $T_\epsilon<\infty$ such 
that for $N\geq N_\epsilon$ and $t\geq T_\epsilon$, for 
an initial condition $\mathcal X\in \Gamma(K,c)$
and $A\subseteq \R^d$ measurable,
\begin{align*}
\P_{\mathcal X}\bigg(\Big|\mu\uppar N (A,t)-\int_A U(x)\,\diffd x\Big|\geq \epsilon
\bigg)
&<\epsilon \\
\text{and }\qquad
\P_{\mathcal X}\bigg(\Big|M\uppar N_t-R_\infty\Big| \ge \epsilon \bigg)
&< \epsilon.
\end{align*}
\end{thm}
As a consequence of Theorems~\ref{thm:piNexists} and~\ref{thm:sspd}, 
for large $N$,
under the invariant distribution $\pi\uppar N$, the proportion of particles in a set $A$ is approximately $\int_A U(x)\, \diffd x$ and the furthest particle distance from the origin is approximately $R_\infty$: 
\begin{thm} \label{cor:ssp}
For $\epsilon>0$ and $A\subseteq \R^d$ measurable,
\begin{align} \label{eq:sspcor3}
\pi \uppar N\left(
\left\{\mathcal X\in (\R^d)^N :\left|  \frac 1N \sum_{i=1}^N
\indic{\mathcal X_i\in A}-\int_A U(x)\,\diffd x\right|\ge\epsilon \right\}
\right)
\to 0
\quad&\text{as }N \to \infty
\\[1ex]
\label{eq:sspcor2}
\text{and }\quad
\pi \uppar N\Big(
\Big\{\mathcal X\in (\R^d)^N :\Big| \max_{i\in \{1,\ldots,N\}}\|\mathcal
X_i\|-R_\infty\Big|\ge\epsilon \Big\}
\Big)
\to 0
\quad&\text{as }N \to \infty.
\end{align}
\end{thm}

The results in this article and in the companion article~\cite{BBNP1} can be summarised in the following informal diagram: 
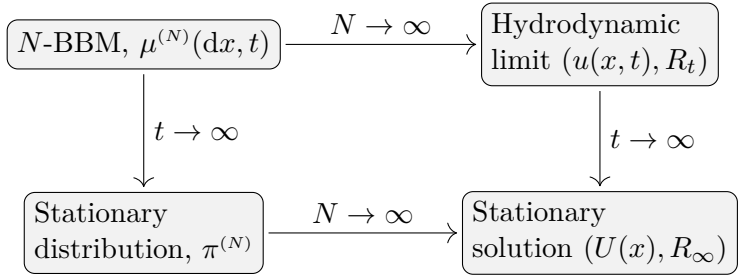
\begin{quote}
\centering
\begin{tikzpicture}[box/.style={draw,rounded corners,align=left,outer
sep=2pt,fill=black!5},arrows={[scale=1.5]}]
\node[box] (a) at (0,0) {$N$-BBM, $\mu\uppar N(\diffd x,t)$};
\node[box] (b) at (6,0) {Hydrodynamic\\limit $(u(x,t),R_t)$};
\node[box] (c) at (0,-2.5) {Stationary\\distribution, $\pi\uppar N$};
\node[box] (d) at (6,-2.5) {Stationary\\solution $(U(x),R_\infty)$};
\draw[->] (a.east) -- node[above] {$N\to\infty$} (b.west);
\draw[->] (a) -- node[right] {$t\to\infty$} (c);
\draw[->] (c.east) -- node[above] {$N\to\infty$} (d.west);
\draw[->] (b) -- node[right] {$t\to\infty$} (d);
\end{tikzpicture}
\end{quote}
In~\cite{BBNP1}, we deal with the
right hand side of the diagram: well-posedness of the free boundary problem~\eqref{pbu_probpaper} and the long-term behaviour of its solutions. In the present article, Theorem~\ref{thm:NBBM} gives rigorous meaning to the top of the diagram, Theorem~\ref{thm:piNexists} covers the left hand side, and Theorem~\ref{cor:ssp} covers the bottom of the diagram.

\subsection{Related works}

The particle system we are considering is a particular case of a more general $N$-particle branching Brownian motion ($N$-BBM) with spatial selection, described as follows: The system consists of $N$ particles  moving in $\R^d$ with locations $\big(X\uppar N_1(t),\ldots , X\uppar N_N(t)\big)$.   Each particle moves independently according to a Brownian motion with diffusivity~$\sqrt 2$ and branches independently into two particles at rate~1. Whenever a particle branches, however, the particle having least ``fitness'' or ``score'' (out of the entire ensemble) is instantly removed (killed), so that there are exactly $N$ particles in the system at all times.  The fitness of a particle is a function $\mathcal{F}(x)$ of its location $x \in \R^d$, and as a result, the elimination of least-fit particles tends to push the ensemble toward regions of higher fitness.  Variants of this stochastic process were first studied in one spatial dimension, beginning with work of Brunet, Derrida, Mueller, and Munier \cite{BDMM06, BDMM07} on discrete-time processes, and the work of Maillard \cite{Maillard16} on the continuous-time model involving Brownian motions.  In these works, the particle removed from the system is always the leftmost particle, which means that they could be described by a monotone fitness function (e.g. $\mathcal{F}(r) = r$, $r \in \R$).
The general multidimensional model which we have just described above was first studied by N.~Berestycki and Zhao \cite{BZ18}; specifically, they studied the particle system with fitness functions $\mathcal{F}(x) = \|x\|$ and $\mathcal{F}(x) = \lambda \cdot x$, both of which have the effect of pushing the ensemble of particles away from the origin.  The Brownian bees model that we consider in this article corresponds to the fitness function $\mathcal{F}(x) = - \|x\|$.

In the setting of one spatial dimension and with monotone fitness function $\mathcal{F}(r) = r,\ r \in \R$, De Masi, Ferrari, Presutti, and Soprano-Loto \cite{DMFPSL} 
determined the hydrodynamic limit of the particle system.
For $t>0$, define the measure
\[
\mu\uppar N(\diffd r,t)=\frac 1 N \sum_{k=1}^N \delta_{X\uppar N_k(t)}(\diffd r).
\]
De Masi et al.~proved that if the initial particle locations $X\uppar N_1(0),\ldots, X\uppar N_N(0)$ are i.i.d., with certain assumptions on the distribution of $X\uppar N_1(0)$,
then the family of empirical measures $\mu\uppar N(\diffd r,t)$ converges, as $N \to \infty$, to a limit which can be identified with a solution $u(r,t)$ to a free boundary problem:
\begin{equation}
\begin{cases}
\partial_t u = \partial_{r}^2 u + u, \quad & r > \gamma_t, \;\; t > 0,  \\
 u(r,t) =0, \quad & r \leq \gamma_t,\;\; t > 0, \\[1ex]
\displaystyle \int_{\gamma_t}^\infty u(r,t) \,\diffd r = 1, \quad & t > 0,
\end{cases} \label{d1fbp}
\end{equation}
where the free boundary at $r = \gamma_t \in \R$ is related to $u$ through the integral constraint.  Global existence of solutions to this free boundary problem was proved by J.~Berestycki, Brunet, and Penington~\cite{BBP}.  
De Masi et al.~also state that for fixed $N$, the particle system (seen from the leftmost particle) converges in distribution as $t\to \infty$ to an invariant measure $\nu_N$, but they did not prove asymptotic results about the shape of the cloud of particles under $\nu_N$ as $N\to \infty$.
As discussed in Section~\ref{subsec:1d} below, a related one-dimensional result plays a fundamental role in our work.  We use some coupling ideas similar to those in the proof of the hydrodynamic limit result in~\cite{DMFPSL}, but we obtain a more quantitative result for our particle system (see Proposition~\ref{prop:NBBM1d} below) which does not require the initial particle locations to be i.i.d.~random variables. This, together with results about the long-term behaviour of the free boundary problem~\eqref{pbu_probpaper} from~\cite{BBNP1}, allows us to control the long-term behaviour of the Brownian bees particle system for large $N$.

Building on the approach of \cite{DMFPSL}, Beckman~\cite{Beck19} derived a similar hydrodynamic limit in the one-dimensional setting with symmetric fitness $\mathcal{F}(r) = -|r|$, which coincides with our case if $d=1$.  Beckman also studied the long-term behaviour of the $N$-BBM in one dimension with a non-monotone fitness function of the form $\mathcal{F}(r) = r + \psi(r)$, $\psi$ being periodic, and proved existence of a stationary distribution in a certain moving reference frame. In earlier work, Durrett and Remenik~\cite{DR11} studied a related branching-selection model in which non-diffusing particles in $\R$ are born at random locations but do not move during their lifetimes. 
They showed that the hydrodynamic limit of this particle system is given by a non-local free boundary problem.

A related model is the Fleming-Viot system studied by Burdzy, Ho\l yst, and March \cite{BHM00}.  In that model, particles diffuse within a bounded domain having {\em fixed} boundary; whenever a particle hits the boundary it is instantly killed and one of the internal particles simultaneously branches, preserving the total mass.  As in our case, the stationary distribution for that system also converges to the principal eigenfunction of the Laplacian (as the number of particles goes to infinity). This eigenfunction is a quasi-stationary distribution for the diffusion conditioned on non-extinction.  See also Collet et al.~\cite{CMSM13}, and references therein, for other related works on quasi-stationary distributions.

In \cite{AFGJ16},  Asselah, Ferrari, Groisman, and Jonckheere considered a slightly different Fleming-Viot particle system. In their work, the $N$  particles live on $\{0,1,2,\ldots\}$,  move  independently as continuous-time sub-critical Galton-Watson processes, and are killed when they hit 0 (each time a particle is killed, one of the remaining $N-1$ particles, chosen at random, branches). Recall that for a single sub-critical Galton-Watson process conditioned on non-extinction, there exists an infinite family of quasi-stationary distributions. (By contrast, observe that a diffusion on a bounded domain conditioned on not exiting the domain has a unique quasi-stationary distribution.) Asselah et al.~showed that for each $N$, the Fleming-Viot particle system has a unique invariant distribution, and that its stationary empirical distribution converges as $N\to \infty$ to the \emph{minimal} quasi-stationary distribution of the Galton-Watson process conditioned on non-extinction (which is the quasi-stationary distribution with the minimal expected time of extinction). This has been called the \emph{selection principle} in the literature. It is reminiscent of the fact that the solution of the Fisher-KPP equation started from a fast decreasing initial condition converges to the minimal-velocity travelling wave (see in particular \cite{GJ18} and the note \cite{GJ13} of Groisman and Jonkheere). This principle is conjectured to hold in quite broad generality. For instance, for the one-dimensional $N$-BBM studied in \cite{DMFPSL}, it is conjectured that the unique invariant distribution of the system seen from the leftmost particle converges, as $N\to \infty$, to the centred minimal-velocity travelling wave solution of~\eqref{d1fbp} (which is given by $\gamma_t =2t,$ $u(2t +r,t) =re^{-r}\indic{r\ge 0}$). 

Finally, we mention the very recent work~\cite{ABLT20} of Addario-Berry, Lin, and Tendron, in which a variant of the Brownian bees model with the following selection rule is considered: each time one of the $N$ particles branches, the particle currently furthest away from the centre of mass of the cloud of particles is removed from the system. Addario-Berry et al.~show that the movement of the centre of mass, appropriately rescaled, converges to a Brownian motion.

\subsection{One-dimensional results and outline of the article} \label{subsec:1d}
The first step in the proofs of Theorems~\ref{thm:NBBM},~\ref{thm:piNexists},~\ref{thm:sspd} and~\ref{cor:ssp}
is to control the proportion of particles within distance $r$ of the origin at a fixed time $t$, when the number of particles $N$ is very large.

For $r>0$, let $\B(r)=\{x \in \R^d : \|x\|<r\}$ be the open ball of radius $r$ 
centred at the origin.
Suppose that $(u,R)$ solves~\eqref{pbu_probpaper} with some initial probability measure $\mu_0$,
and let $v:[0,\infty)\times (0,\infty)\to [0,1]$ denote the mass of $u$ within distance
$r$ of the origin at time $t$:
\begin{equation*}
v(r,t) = \int_{\B(r)} u(x,t) \, \diffd x. 
\end{equation*}
Then $r\mapsto v(r,t)$ is non-decreasing and $v(r,t)=1$ for $r\ge
R_t$. Let $v_0(r)=\mu_0\big(\B(r)\big)$; then,
by Lemma~6.2 of~\cite{BBNP1},
 $v$ satisfies the following parabolic obstacle problem:
\begin{equation}
\begin{cases}
\displaystyle	0 \leq v(r,t) \leq 1,& \text{for $t>0$, $r \geq 0$,}\\[1ex]
\displaystyle \partial_t v = \partial_r^2v -\frac{d-1}r\partial_r v+v, &\text{if $v(r,t) <1$,}   \\[1ex]
v(0,t)=0, &\text{for $t>0$},\\
v(r,t) \text{ is continuous on $[0,\infty) \times (0,\infty)$}, \\
\partial_r v(\cdot ,t) \text{ is continuous on $[0,\infty)$}, & \text{for $t>0$,}\\
v(\cdot,t)\to v_0 &\text{in $L^1_\text{loc}$ as $ t\searrow0$}.
\end{cases}
\label{pbv_probpaper}
\end{equation}
We prove in Theorem~2.1 of~\cite{BBNP1} that for any measurable $v_0:[0,\infty) \to [0,1]$,~\eqref{pbv_probpaper} has a unique solution.

The following result is a hydrodynamic limit result for the distances of particles from the origin, and will be an important step in the proofs of Theorems~\ref{thm:NBBM} and~\ref{thm:sspd}.
Introduce
\begin{equation}\label{def F}
F\uppar N(r,t):=\mu\uppar N\big(\B(r),t\big)
\end{equation}
as the proportion of particles within distance $r$ of the origin at time~$t$. Then, Proposition~\ref{prop:NBBM1d} below says that for any initial configuration of particles, at a fixed time $t$, the proportion $F\uppar N(r,t)$ is close to the solution $v\uppar N(r,t)$ of~\eqref{pbv_probpaper} with initial condition $v_0$ determined by the initial configuration of particles. The bound does not depend on the initial particle configuration; this will be crucial when the result is used in the proof of Theorem~\ref{thm:sspd}.
\begin{prop} \label{prop:NBBM1d}
There exists $c_1\in (0,1)$ such that for $N$ sufficiently large, for $t>0$ and any $\mathcal X \in (\R^d)^N$,
$$
\psub{\mathcal X}{ \sup_{r\geq 0}\left|F\uppar N (r,t)-v\uppar N(r,t) \right|
\geq e^{2t} N^{-c_1} }
\leq e^t N^{-1-c_1},
$$
where $v\uppar N$ is the solution of~\eqref{pbv_probpaper} with $v_0(r)=F\uppar 
N(r,0)$ $\forall r \geq 0$.
\end{prop}

The next result uses Proposition~\ref{prop:NBBM1d} to get an upper bound on the largest particle distance from the origin which holds over a time interval of fixed length. 
This will then allow us to compare the particle system to a system in which particles are killed if they are further than a deterministic distance from the origin, which will enable us to prove the $d$-dimensional hydrodynamic limit in Theorem~\ref{thm:NBBM}.
\begin{prop}\label{prop:bound on M}
There exists $c_2 \in (0,1)$ such that under the assumptions of Theorem~\ref{thm:NBBM}, for any $0<\eta<T$, 
 for $N$ sufficiently large (depending on $\eta$ and $T$),
\begin{equation*}
\P\left(\exists t\in [\eta,T]:M\uppar N_t>R_t+\eta \right)\leq N^{-1-c_2}.
\end{equation*}
\end{prop}
In the case where $\mu_0$ has compact support, the proof of Proposition~\ref{prop:bound on M} can easily be extended to bound the probability that there exists $ t\in (0,T]$ with $M\uppar N_t>R_t+\eta $. However, an upper bound on $M\uppar N_t$ in the time interval $[\eta,T]$ (for an arbitrarily small $\eta$) is enough to allow us to prove Theorem~\ref{thm:NBBM}.

Using Proposition~\ref{prop:NBBM1d} and results about the long-term behaviour of solutions to the obstacle problem~\eqref{pbv_probpaper} from the companion paper~\cite{BBNP1}, we can also prove the following result about the long-term behaviour of particle distances from the origin when $N$ is large.
For $r\ge 0$, let
\begin{equation} \label{eq:Vdef}
V(r) = \int_{\B(r)} U(x) \,\diffd x,
\end{equation}
where $U$ is defined in~\eqref{eqU_probpaper}.
\begin{prop} \label{prop:ssp}
Take $K>0$ and $c\in (0,1]$.
For $\epsilon>0$, there exist $N_\epsilon<\infty$ and $T_\epsilon<\infty$ such 
that for $N\geq N_\epsilon$ and $t\geq T_\epsilon$, for 
an initial condition $\mathcal X\in \Gamma(K,c)$,
\begin{align} 
\P_{\mathcal X}\bigg(\sup_{r\geq 0}\Big|F\uppar N (r,t)-V(r)\Big|\geq \epsilon 
\bigg) &< \epsilon,  \label{eq:sspA}\\
\P_{\mathcal X}\bigg( \Big|M\uppar N_t-R_\infty\,\Big| \ge \epsilon 
\bigg)
&< \epsilon ,
\label{eq:sspB}\\
\text{and }\qquad 
\P_{\mathcal X}\bigg( \sup_{s\in [0,1]}M\uppar N _{t+s}>R_\infty +\epsilon 
\bigg)&<\epsilon. \label{eq:sspC}
\end{align}
\end{prop}
Using~\eqref{eq:sspC}, we can compare the particle system at large times to a system in which particles are killed if they are further than distance $R_\infty+\epsilon$ from the origin. This, together with~\eqref{eq:sspA}, will allow us to prove Theorem~\ref{thm:sspd}.

The rest of the article is laid out as follows.
In Section~\ref{subsec:resultsfromPDE}, we recall results from~\cite{BBNP1} which will be used in this article.
In Section~\ref{sec:notations}, we define notation which will be used throughout the proofs.
Then in Section~\ref{sec:1dhydro}, we  prove Propositions~\ref{prop:NBBM1d} and~\ref{prop:bound on M}, and in Section~\ref{sec:hydrod} we use Proposition~\ref{prop:bound on M} to prove Theorem~\ref{thm:NBBM}.
In Section~\ref{sec:ssp}, we prove Proposition~\ref{prop:ssp} and use this to prove Theorem~\ref{thm:sspd}, and, finally, Theorems~\ref{thm:piNexists} and~\ref{cor:ssp}.

\vspace{0.3in}

{\bf Acknowledgements:}   The work of JN was partially funded through grant DMS-1351653 from the US National Science Foundation.  The authors wish to thank Louigi Addario-Berry, Erin Beckman, Nathana\"{e}l Berestycki and Pascal Maillard for stimulating discussions at various points of this project.

\section{Results from~\texorpdfstring{\cite{BBNP1}}{[BBNP20]}}
\label{subsec:resultsfromPDE}

In this section, we state some results from~\cite{BBNP1} which play a key role in the present work.
The first one is
Theorem~1.1 in~\cite{BBNP1}, which says that the free boundary problem~\eqref{pbu_probpaper} has a unique solution, and that moreover the free boundary radius $R_t$ is continuous.
\begin{thm}[Theorem~1.1 in~\cite{BBNP1}]\label{thm:exists u_probpaper}
Let $\mu_0$ be a Borel probability measure on $\R^d$. Then
there exists a unique classical solution to problem~\eqref{pbu_probpaper}. Furthermore,
\begin{itemize}
\item $t\mapsto R_t$ is continuous (and finite) for $t > 0$. 
\item As $t \searrow 0$, $R_t \to R_0:=\inf\big\{r>0 : \mu_0\big( \B(r)\big)=1\big\} \in [0,\infty]$.
\item For $t>0$ and $\|x\|<R_t$, $u(x,t)>0$.
\end{itemize}
\end{thm}
For a Borel probability measure $\mu_0$ on $\R^d$, let $(u,R)$ denote the solution of~\eqref{pbu_probpaper}.
Let $(B_t)_{t\ge 0}$ denote a $d$-dimensional Brownian motion with diffusivity $\sqrt 2$, and for $x\in \R^d$, write $\mathbb P_x$ for the probability measure under which $B_0=x$.
For $t > 0$, define a family of measures on $\R^d$ according to
\begin{equation*} 
\rho_t(x,A) = \P_{\!x}\big( B_t \in A, \;\; \|B_s\|<R_s \;\forall s\in (0,t)\big) 
\end{equation*}
for all Borel sets $A \subseteq \R^d$.  Then $\rho_t(x,\diffd y)$ is absolutely continuous with respect to the Lebesgue measure, so it has a density. Abusing notation, we denote this density by $\rho_t(x,y)$. 
Then by Proposition~6.1 in~\cite{BBNP1}, 
\begin{align}
u(y,t) & = e^t \int_{\R^d} \mu_0(\diffd x) \rho_t(x,y), \quad y \in \R^d, \; t > 0. \label{udef_probpaper}
\end{align}

Define the
cumulative distribution of the norm process $\|B_t\|$ conditional on $\|B_0\| = y$ as
\begin{equation}
w(y,r,t) := \P\big(\|B_t\| < r \;\big|\; \|B_0\| = y\big) .
\label{wdef_probpaper}
\end{equation}
Then the function $r \mapsto g(y,r,t) := \partial_r w(y,r,t)$ is the
density of $\|B_t\|$ conditional on $\|B_0\| = y$; in other words, $g$
is the transition density of the $d$-dimensional Bessel process with diffusivity $\sqrt2$. The function 
\begin{equation}
G(y,r,t) := -\partial_y w(y,r,t)  \label{Gdef_probpaper}
\end{equation}
is the fundamental solution of the equation
\begin{equation}\label{eqG_probpaper}
\partial_t G =\partial_r^2 G -\frac{d-1}r \partial_r G,\qquad
G(y,0,t)=0,\qquad G(y,r,0)=\delta(r-y).
\end{equation}
(See Section~3 of~\cite{BBNP1} for more details on the properties of $G$.)
For $t,r,y > 0$, the fundamental solution $G$ and transition density $g$ are smooth functions of their arguments, and are related by
\begin{equation} \label{eq:Grgy_probpaper}
\partial_r G = -\partial_y g.
\end{equation}
For a given initial condition $\vl_0\in L^\infty (0,\infty)$, we let
\begin{equation} \label{eq:vlformula_probpaper}
\vl(r,t)=e^t\int_0^\infty\diffd y\, G(y,r,t)\vl_0(y).
\end{equation}
This $\vl$ is a solution to the linear problem
\begin{equation}\label{linear equ_probpaper}
\begin{cases}
\partial_t \vl = \partial_r^2 \vl -\frac{d-1} r \partial_r \vl
+ \vl, \quad &\text{for }t>0, \, r\ge 0,\\
\vl(0,t)=0, \quad &\text{for }t > 0, \\
\vl(\cdot,t)\to\vl_0 \quad &\text{in $L^1_\text{loc}$ as
$t\searrow0$},
\end{cases}
\end{equation}
and it is the unique solution to \eqref{linear equ_probpaper} which is bounded on $[0,\infty) \times [0,T]$ for each $T > 0$.  In the particular case $v_0^\ell(r) = \indic{y \leq r}$, we have $v^\ell(r,t) = e^t w(y,r,t)$.

For $t>0$ and $m\in \R$, we define the operators $G_t$ and $C_m$ by letting
\begin{equation} \label{eq:GCdef_probpaper}
G_tf(r)=\int_0^\infty\diffd y\, G(y,r,t)f(y)
\qquad \text{and}\qquad C_m f(r) = \min\big( f(r),m\big).
\end{equation}
In particular, $\vl=e^tG_t\vl_0$. 
By Lemma~3.1 in~\cite{BBNP1} we have that $\left|
\int_0^\infty G(y,r,t) h(y)\,\diffd y \right| \le \|h\|_{L^\infty}$, and so for $f,g\in L^\infty [0,\infty)$ and $t>0$,
\begin{equation} \label{eq:Gcbounds_probpaper}
 \|G_t f-G_t g\|_{L^\infty} \leq \|f-g\|_{L^\infty}.
\end{equation}
Suppose $v_0: [0,\infty) \to [0,1]$ is non-decreasing, and let $v$ denote the solution of the obstacle problem~\eqref{pbv_probpaper} with initial condition $v_0$.
For $\delta>0$ and $k\in \N_0$, we let
\begin{equation} \label{eq:defv+v-_probpaper}
v^{k,\delta,-}=\big(e^\delta G_\delta C_{e^{-\delta}}\big)^k v_0,
\qquad
v^{k,\delta,+}=\big(C_1e^\delta G_\delta \big)^k v_0.
\end{equation}
Then by Lemmas~4.3 and~4.4 in~\cite{BBNP1}, we have the following result.
\begin{lem}[Lemmas 4.3 and 4.4 in~\cite{BBNP1}]\label{lem:vnapprox}
For any $\delta > 0$ and $k \in \N_0$,
\[
v^{k,\delta,-}(r)\le v(r,k\delta) \le v^{k,\delta,+}(r) \quad \forall \; r 
\geq 0
\quad \text{ and }\quad \big\Vert v^{k,\delta,+}-v^{k,\delta,-}\big\Vert_{L^\infty} \le (e^{k\delta}+1)(e^\delta-1).
\]
\end{lem}
We shall also use the following result which was proved as part of Theorem~2.1 in~\cite{BBNP1}. 
The result says that the solution of~\eqref{pbv_probpaper} is continuous with respect to the initial condition in the following sense.
\begin{lem}[From Theorem 2.1 in~\cite{BBNP1}]\label{lem:cont init cond_probpaper}
Let $v$ and $\tilde v$ be the solutions to \eqref{pbv_probpaper} corresponding to
the initial conditions $v_0$ and $\tilde v_0$.
Then
for $t>0$,
\[
\|v(\cdot,t) - \tilde v(\cdot,t)\|_{L^\infty}  \leq e^t \|v_0 - \tilde v_0 \|_{L^\infty}. 
\]
\end{lem}
Recall the definition of $V$ in~\eqref{eq:Vdef}.
The following result, which follows directly from Theorem~2.2 in~\cite{BBNP1}, gives us control over how quickly the solution $v(\cdot,t)$ of the obstacle problem~\eqref{pbv_probpaper} converges to $V$.
\begin{prop}[From Theorem 2.2 in~\cite{BBNP1}]\label{prop:conv}
For $c\in(0,1]$, $K>0$ and $\epsilon>0$, there exists $t_\epsilon=t_\epsilon(c,K)\in (0,\infty)$ such that the following holds. Suppose
$v_0:[0,\infty) \to [0,1]$ is non-decreasing with 
$v_0(K)\geq c$, and let $v$ solve the obstacle problem~\eqref{pbv_probpaper} with initial condition $v_0$.
For $t> 0$, let $R_t=\inf\big\{r\ge 0:v(r,t)=1\big\}$. Then for $t\ge t_\epsilon$,
\[
\big|v(r,t)-V(r)\big|<\epsilon \,\,\; \forall r\geq 0 \quad \text{ and }\quad |R_{t}-R_\infty |<\epsilon.
\]
\end{prop}
The final result from~\cite{BBNP1} which we need in this article is Proposition~5.10, which says that for large $K$ and small $c$, if an initial condition $v_0$ has mass at least $c$ within distance $K$ of the origin then the solution $v$ of~\eqref{pbv_probpaper} has mass at least $2c$ within distance $K-1$ of the origin during a fixed time interval $[t_0,2t_0]$.
\begin{prop}[Proposition 5.10 in \cite{BBNP1}] \label{prop:movemass}
There exist $t_0>1$ and $c_0\in (0,1/2)$ such that for all $c\in 
(0,c_0]$, all $K\ge 2$, and all $t_1\in [t_0,2t_0]$, for $v_0:[0,\infty)\to[0,1]$ measurable, the condition
\begin{equation*}
v_0(r) \geq c \indic{r \ge K} \;\;\forall r \geq 0
\end{equation*}
implies that 
\begin{equation*}
v(r,t_1) \geq 2c \indic{r\ge K-1} \;\;\forall r \geq 0, \label{vshift1}
\end{equation*}
and
\begin{equation*}
v(r,n t_1) \geq \min(2 c_0 \,, \,2^n c) \indic{r\ge \max(K-n,1)}, \quad \forall r \geq 0, \quad n \in \N, \label{vshiftn}
\end{equation*}
where $v(r,t)$ denotes the solution of~\eqref{pbv_probpaper} with initial condition $v_0$. 
\end{prop}

\section{Notation} \label{sec:notations}

From now on, we let $(B_t)_{t\ge 0}$ denote a $d$-dimensional Brownian motion with diffusivity $\sqrt 2$,
and for $x\in \R^d$ we write $\P_{\!x}$ for the probability measure under which $B_0=x$, and write $\E_x$ for the corresponding expectation.

The locations (positions) of a  collection of $m$ particles in $\R^d$ are written as a vector $\mathcal X\in (\R^d)^m$. The size of the vector (i.e.\@ the number of particles in the collection) is written $|\mathcal X|$ (i.e. $|\mathcal X| = m$ for $\mathcal{X} \in (\R^d)^m$). The individual locations in $\mathcal X$ are written $\mathcal X_k$ for $k\in\{1,\ldots |\mathcal X|\}$:
$$\mathcal X = \big(\mathcal X_1,\ldots,\mathcal X_m\big)\qquad\text{with}\qquad m=|\mathcal X|.$$
We extend some set notation to vectors. Specifically, we write $\mathcal X\subseteq\mathcal Y$ to mean that all the particles in $\mathcal X$ are also in $\mathcal Y$:
\[
\mathcal X\subseteq\mathcal Y\qquad\Leftrightarrow\qquad
\exists j:\{1,\ldots,|\mathcal X|\}\to\{1,\ldots,|\mathcal Y|\} \text{ injective such that }
\mathcal X_k =\mathcal Y_{j(k)}\ \forall k.
\]
We write $\big| A \cap \mathcal X \big|$ for the number of particles in $\mathcal X$ which lie in some set $A \subseteq \R^d$:
\[
\big| A \cap \mathcal X \big| = \Big|\Big\{k\in\{1,\ldots,|\mathcal X|\}\;:\; \mathcal X_k\in A\Big\}\Big|.
\]

If $\mathcal X$ and $\mathcal Y$ are two vectors of particles with locations in $\R^d$, we write $\mathcal X\preceq \mathcal Y$ to mean that the vector $\mathcal X$ contains more particles than the vector $\mathcal Y$ in any ball centred on the origin:
\begin{equation} \label{eq:preceqdef}
\mathcal X\preceq \mathcal Y\qquad\Leftrightarrow\qquad
\big|\mathcal{X}\cap \B(r) \big|\ge \big|\mathcal{Y}\cap \B(r)\big| \quad \text{for all $r>0$},
\end{equation}
where we recall that $\mathcal B(r)$ is the centred open ball of radius $r$:
\[\mathcal B(r) = \Big\{x\in\R^d\;:\; \|x\|<r\Big\}.\]
Notice that $\mathcal X\preceq \mathcal Y$ implies that $|\mathcal X|\ge|\mathcal Y|$.

The order in which the particle locations are written within a vector $\mathcal X$ is irrelevant for the operations $\subseteq$, $\cap$ and $\preceq$ described above.  

It will be useful to compare the $N$-BBM to the standard $d$-dimensional binary branching Brownian motion (BBM) without selection, in which particles move independently in $\R^d$ according to Brownian motions with diffusivity $\sqrt 2$, and branch into two particles at rate $1$.
The BBM may be labelled using the Ulam-Harris scheme (see \cite{Jagers89} and references therein).   Let $\mathcal U= \cup_{n=1}^\infty \N^n$, and for $t\ge 0$,
let $\set+_t\subset \mathcal U$ denote the set of Ulam-Harris labels of the particles in the BBM at time $t$.
(If the BBM has $m$ particles at time 0 then the particles initially have labels $1,\ldots,m$ and for each $t\ge 0$, $\set+_t\subset \cup_{n=1}^\infty \{1,\ldots,m\}\times \{1,2\}^{n-1}$.  For example, the particles labelled $(7,1)$ and $(7,2)$ are the two children of the seventh original particle; the particle $(7,2,1)$ is a grandchild of the seventh original particle and is the first child of $(7,2)$.) 
For $u\in \set+_t$, let $X^+_u(t)$ denote the location of particle $u$ at time $t$,
and for $s\in [0,t)$, let $X^+_u(s)$ denote the location of its ancestor in the BBM at time $s$.
We write $X^+(t)=(X^+_u(t))_{u\in \set+_t}$ for the vector of particle locations in the BBM at time $t$,
with particles ordered lexicographically by their Ulam-Harris labels.

In the proofs, we shall often use the following standard coupling between the BBM and the
$N$-BBM: consider a standard $d$-dimensional binary BBM as described above. In addition to their spatial location, let each particle carry a colour attribute, either red or blue. When a blue particle
branches, the two offspring particles are coloured blue, and simultaneously the blue particle furthest
from the origin turns red. When a red particle branches, the two
offspring particles are coloured red. The system begins with $N$ blue particles at time $0$. The set of
blue particles is a realisation of the $N$-BBM, while the entire collection of particles (blue and red) is a realisation of standard BBM.  Specifically, for $t\ge 0$, let $\setN_t \subseteq \set+ _t$ denote the set of blue particles at time $t$.
Then, $\setN_t$ is always a set of size $N$, and there exists an enumeration $(u_k)_{k=1}^N$ of $\setN_t$ such that
$X\uppar N(t)=\big(X^+_{u_1}(t),\ldots, X^+_{u_N}(t)\big)$. 
Recall that we let $M\uppar N_t =\max_{k\in \{1,\ldots, N\}}\|X\uppar N_k(t)\|$, the
maximum distance of a particle in the $N$-BBM from the origin at time $t$.
Notice that for $t\ge 0$, almost surely
\begin{equation} \label{eq:XNcontains}
\setN_t = \big\{ u \in \set+_t : \|X^+_u(s)\|\le M\uppar N_s \,\; \forall s\in [0,t]\big\}.
\end{equation}

We usually write $\mathcal X$ for the initial configuration of the $N$-BBM
and of the BBM. In some cases where we compare an $N$-BBM and a BBM with
different initial conditions, we write $\mathcal X^+$ for the initial condition of the BBM. Expectations and laws started from an initial condition $\mathcal X$ are written $\E_{\mathcal X}$ and $\P_{\mathcal X}$ respectively.   
We also write $(\mathcal F_t)_{t\ge 0}$ for the natural filtration of $(X\uppar N (t), t\ge 0)$, i.e.~$\mathcal F_t =\sigma((X\uppar N (s),s\le t))$.

\section{One-dimensional hydrodynamic limit results} \label{sec:1dhydro}

In this section, we prove the hydrodynamic limit results about the distances of particles from the origin, Propositions~\ref{prop:NBBM1d} and~\ref{prop:bound on M}.

For $t\ge 0$ and $r > 0$, recall the definition~\eqref{def F} of the cumulative distribution function
\begin{equation}
F\uppar N(r,t)= \mu \uppar N(\B(r),t) = \frac{1}{N} \Big|X\uppar N(t) \cap
\B(r) \Big| = \frac{1}{N} \sum_{i=1}^N \indic{ \| X_i\uppar N(t) \| < r} \label{cumdistrdefn}
\end{equation}
and introduce
\[
F^{+}(r,t)=\frac{1}{N} \Big|X^{+}(t) \cap \B(r) \Big| = \frac{1}{N} \sum_{u\in \set+ _t} \indic{ \| X_u^{+}(t) \| < r}.
\]

\subsection{Proof of Proposition~\ref{prop:NBBM1d}}

\subsubsection{Upper bound for the proof of Proposition~\ref{prop:NBBM1d}} \label{subsec:hydro1upper}

Recall from~\eqref{eq:GCdef_probpaper} that for $m\in \R$ and $f:[0,\infty) \to \R$,
we let $C_m f(r)=\min\big(f(r),m\big)$.
The following proposition will play a crucial role in the proof of Proposition~\ref{prop:NBBM1d}; it says that the random function $F^+$ corresponding to the BBM stochastically dominates the random function $F\uppar N$ corresponding to the $N$-BBM, if both processes start from the same particle configuration.

\begin{prop} \label{prop:FNdom}
Suppose  $\mathcal{X} = ( \mathcal X_1,\dots, \mathcal X_N)\in (\R^d)^N$. There exists a coupling of the $N$-BBM $X\uppar N(t)$ started from
$\mathcal X$ and of the BBM $X^+(t)$ also started from $\mathcal X$ such
that
\[ F\uppar N(\cdot,t) \le C_1 F^+(\cdot,t) \quad \forall \; t \geq 0. \]
(Equivalently, $X^+(t)\preceq X\uppar
N(t)$ for all $t \geq 0$.) In particular, if $f:[0,\infty) \to \R \cup \{ +\infty\}$ is measurable, then for all
$t\ge0$,
\begin{align*} 
\P_{\mathcal X} \bigg( \sup_{r \geq 0} \,\Big(F\uppar N(r,t) - f(r)\Big) \geq
0  \bigg) \leq  \P_{\mathcal X} \bigg( \sup_{r \geq 0} \, \Big(C_1
F^{+}(r,t) - f(r)\Big) \geq 0 \bigg) .
\end{align*}
\end{prop}
\begin{proof}
This is a direct property of the standard coupling between the $N$-BBM $X\uppar N(t)$ and the BBM  $X^+(t)$ as described in Section~\ref{sec:notations}. Observe that with $X\uppar N(0) = \mathcal{X} = X^+(0)$,  under the coupling described in Section~\ref{sec:notations}, for $t\ge 0$ we have $X\uppar N(t)\subseteq X^+(t)$. It follows that $F\uppar N(r,t)\leq F^+(r,t)$ for all $r>0$ and $t\geq 0$. Therefore, since $F\uppar N(r,t) \leq 1$ also holds for all $r > 0$ and $t \geq 0$, we have \[ F\uppar N(r,t) \leq C_1 F^+(r,t), \quad \quad r > 0, \;\; t \geq 0,
\]
which completes the proof.
\end{proof}

Recall the definition of the operator $G_t$ from~\eqref{eq:GCdef_probpaper}, and introduce
\begin{equation}\label{vl}
v^\ell(r,t) = e^t G_t v^\ell_0(r) 
\qquad\qquad\text{with}\qquad v^\ell_0(r) = F^+(r,0),
\end{equation}
the solution of \eqref{linear equ_probpaper} with
initial condition determined by the initial configuration $X^+(0)$ of the BBM.

\begin{lem} \label{lem:Ftilde+}
There exists $N_0 <\infty$ such that for all $N \geq N_0$, all $\mathcal X \in (\R^d)^m$ with $m \leq N$, and all $t >0$, \begin{align*} 
\sup_{r > 0} \, \P_{\mathcal X} \left( \left| F^{+}(r,t)-v^{\ell}(r,t) \right| \geq N^{-1/5}  \right) &\leq 13e^{4 t}N^{-6/5}
\end{align*}
and
\begin{align*}
\P_{\mathcal X} \left(  N^{-1}\big(|X^+(t)| - e^t |\mathcal{X}|\big) \geq N^{-1/5}\right) &\leq 13e^{4 t}N^{-6/5}.
\end{align*}
\end{lem}

\begin{proof}
Recall that $(B_t)_{t\geq 0}$ is a $d$-dimensional Brownian motion with diffusivity $\sqrt 2$.  We claim that for $r> 0$ and $t\geq 0$, 
\begin{align}
v^{\ell}(r,t)= \frac 1N \sum_{i=1}^{| \mathcal{X}|}e^t \psub{\mathcal X_i}{\|B_t\|<
r}.
 \label{Fpmean}
\end{align}
Indeed, 
by the definition of $G_t$ in~\eqref{eq:GCdef_probpaper},
\begin{align*}
G_t F^{+}(r,0)
&=\frac 1N \int_{0}^\infty G(y,r,t)
\sum_{i=1}^{| \mathcal{X}|} \indic{\|\mathcal X_i \|< y}\, \diffd y\\
&=\frac 1N \sum_{i=1}^{| \mathcal{X}|} \int_{\|\mathcal X_i \|}^\infty 
G(y,r,t)\,
\diffd y\\
&=-\frac 1N \sum_{i=1}^{| \mathcal{X}|} \int_{\|\mathcal X_i \|}^\infty
\int_0^r \partial_y g(y,r',t)\,\diffd r'
\diffd y\\
&=\frac 1N \sum_{i=1}^{| \mathcal{X}|} \int_0^r
g(\|\mathcal X_i \|,r',t)\,\diffd r'
\\
&=\frac 1N \sum_{i=1}^{| \mathcal{X}|}\psub{\mathcal X_i}{\|B_t\|<
r}, 
\end{align*}
where the third line follows by~\eqref{eqG_probpaper} and~\eqref{eq:Grgy_probpaper}, and the last two lines follow since $g(y,\cdot,t)$ 
is the density at time $t$ of a $d$-dimensional Bessel process started at $y$. This proves the claim \eqref{Fpmean}.

In the BBM $X^+$ started from $\mathcal X$, for $i\in \{1,\ldots,  |\mathcal X|\}$, denote by $\mathcal N^{+,i}_t$ the family of particles at time $t$ descended from the $i$-th particle
in $\mathcal X$, i.e., recalling that particles are labelled according to
the Ulam-Harris scheme, let
\[
\mathcal N^{+,i}_t = \big\{u \in \set+ _t : u= (i,u_2,\ldots) \big\}.
\]
 Then letting $X^{+,i}(t)=(X^{+}_u(t))_{u\in \mathcal N^{+,i}_t}$, the processes  $X^{+,i}$ form a family of independent BBMs, and for
each $i$ the process $X^{+,i}$ is 
started from a single particle at location $\mathcal X_i$. 
Fix a time $t>0$
and write $n_i=\big|\mathcal N^{+,i}_t\big |$ for the number of particles descended from
$\mathcal X_i$ at time $t$, and 
\[ n_i(r)=\Big|\Big\{u \in \mathcal N^{+,i}_t : X^{+}_u(t)\in\B(r)\Big\}\Big| \]
for the number of particles at time $t$ descended from $\mathcal X_i$ which lie within
distance $r$ of the origin.
 Then $ F^+(r,t)=\frac1N\sum_{i=1}^{|\mathcal X|} n_i(r) $ and, 
 by the many-to-one lemma (see~\cite{HarrisRoberts17}) and~\eqref{Fpmean},
 \[ \frac1N\sum_{i=1}^{|\mathcal X|} \E_{\mathcal X} [n_i(r)] =\frac 1N \sum_{i=1}^{| \mathcal{X}|}e^t \psub{\mathcal X_i}{\|B_t\|<
r} =v^{\ell}(r,t) .
 \]
Therefore
\begin{align*}
&\E_{\mathcal X}  \left[  \left(F^+(r,t)-v^{\ell}(r,t) \right)^4  \right]
\\
&\qquad\qquad= \frac 1 {N^4}
\E_{\mathcal X}\bigg[
	\Big( \sum_{i=1}^{|\mathcal X|} \big( n_i(r) - \E_{\mathcal X}[n_i(r)] \big) \Big)^4
\bigg]\\
&\qquad\qquad= \frac 1 {N^4}
\Bigg(\sum_{i=1}^{|\mathcal X|} \E_{\mathcal X}\bigg[ \Big(n_i(r) - \E_{\mathcal X}[n_i(r)] \Big)^4 \bigg]
+6\sum_{\substack{i,j=1\\i<j}}^{|\mathcal X|}
\var_{\mathcal X}\big(n_i(r)\big)
\var_{\mathcal X}\big(n_j(r)\big)
\Bigg).
\end{align*}
Note that 
\begin{equation} \label{eq:hydroA1}
\E_{\mathcal X}\Big[ \big(n_i(r) - \E_{\mathcal X}[n_i(r)] \big)^4  \Big] \le \E_{\mathcal X}\Big[ n_i(r)^4 +\E_{\mathcal X}[n_i(r)]^4 \Big] \le 2\E_{\mathcal X}\Big[ n_i(r)^4\Big] \le 2\E_{\mathcal X}\Big[ n_i^4\Big] 
\end{equation}
and that 
\begin{equation} \label{eq:hydroA2}
\var_{\mathcal X}\big(n_i(r)\big) \le \E_{\mathcal X}\Big[ n_i(r)^2\Big] \le \E_{\mathcal X}\Big[ n_i^2\Big] .
\end{equation}
Furthermore, since $n_i$ has geometric distribution with parameter $e^{-t}$, 
\begin{align} 
\E_{\mathcal X}\big[ n_i^4 \big] \leq 24e^{4t} \label{eq:hydroA} \quad \text{and }\quad
\E_{\mathcal X}\big[ n_i ^2 \big]\leq 2e^{2t} .
\end{align}
We now have that, since $|\mathcal X|\le N$,
\begin{align*}
\E_{\mathcal X}  \left[  \left(F^+(r,t)-v^{\ell}(r,t) \right)^4  \right]
& \leq  \frac 1 {N^4}\left(  N  48e^{4 t} +3N^2 4e^{4t} \right)
\le (48 N^{-3} +12N^{-2}) e^{4t}.
\end{align*}
By the same argument,
since $|X^+(t)|=\sum_{i=1}^{|\mathcal X|}n_i$ and $\sum_{i=1}^{|\mathcal X|}\E_{\mathcal X}[n_i]=e^t |\mathcal X|$,
we also have
\[
\E_{\mathcal{X}} \left[ \left( N^{-1}(|X^+(t)| - e^t |\mathcal{X}|)\right)^{4}  \right] \leq  (48 N^{-3}+12N^{-2}) e^{4t}.
\]
By Markov's inequality, the stated results follow for $N$ sufficiently large that 
\begin{equation*}
N^{4/5}(48 N^{-3} +12N^{-2} )\leq 13 N^{-6/5}. \tag*{\qedhere} 
\end{equation*}
\end{proof}

Next, we upgrade the estimate in Lemma \ref{lem:Ftilde+} to be uniform in
$r$, at the expense of slightly slower decay in $N$. Recall the definition
of $v^\ell$ in \eqref{vl}.

\begin{lem}\label{lem:F+}
There exists $N_0 <\infty $ such that for all $N \geq N_0$, all $\mathcal X \in (\R^d)^m$ with $m \leq N$, and all  $t >0$ such that $e^t |\mathcal X|N^{-1}\geq 1-N^{-1/10}$,
\begin{align*} 
\P_{ \mathcal{X}}\left( \sup_{r \ge 0}  \left| C_1 F^{+}(r,t)-C_1v^{\ell}(r,t) \right| \geq 3 N^{-1/10} \right) &\leq 13 e^{4t}N^{-11/10}.
\end{align*}
\end{lem}
\begin{proof}
Fix $t>0$ such that $e^t |\mathcal X|N^{-1}\geq 1-N^{-1/10}$. For $k \in \N_0$ with $k 
\leq \lfloor N^{1/10}\rfloor $, let 
$$
r_k = \inf\left\{y\geq 0: v^{\ell}(y,t) \geq \frac{k}{\lfloor N^{1/10} \rfloor} 
\right\}.
$$
Recall the formula for $v^\ell (\cdot,t)$ in~\eqref{Fpmean} in the proof of Lemma~\ref{lem:Ftilde+}.
Note that $F^+(\cdot,t)$ and $v^{\ell}(\cdot,t)$ are non-decreasing, that 
$v^{\ell}(\cdot,t)$ is continuous and that 
$(r_k)$ for $k=0,1,\ldots ,\lfloor N^{1/10}\rfloor$
is an increasing sequence with $r_0=0$.
As $r\to \infty$,
$v^{\ell}(r,t) \to e^t N^{-1}|\mathcal X|\geq 1-N^{-1/10}$, so 
$r_k <\infty$ for $k \leq \lfloor N^{1/10} \rfloor -2$.

Suppose that for every $k \in \{1,\ldots,\lfloor N^{1/10}\rfloor -2\}$,
\begin{equation} \label{eq:hydro(*)}
\left|F^+(r_k,t)-v^{\ell}(r_k,t)\right|\leq N^{-1/5}.
\end{equation}
Then, for $r\geq 0$, if $r\in [r_k,r_{k+1}]$ for some $k \in 
\big\{0,\ldots,\lfloor N^{1/10}\rfloor -3 \big\}$
we have 
$$
\frac{k}{\lfloor N^{1/10}\rfloor}-N^{-1/5}
\leq F^{+}(r_k,t)
\leq  F^{+}(r,t)
\leq  F^{+}(r_{k+1},t)
\leq N^{-1/5}+\frac{k+1}{\lfloor N^{1/10}\rfloor}
$$
and
$$
v^{\ell}(r,t) \in 
\left[\frac{k}{\lfloor N^{1/10}\rfloor},\frac{k +1}{\lfloor 
N^{1/10}\rfloor} \right].
$$
If $N$ is large enough that $\frac{\lfloor N^{1/10}\rfloor -2}{\lfloor 
N^{1/10}\rfloor}+N^{-1/5}\leq 1$, then for all $r \le r_{\lfloor N^{1/10} \rfloor -2}$ we have
$C_1 F^{+}(r,t)= F^{+}(r,t)$ and 
$C_1 v^{\ell}(r,t) = v^{\ell}(r,t)$.
Thus,
\begin{align*}
\left|C_1 F^{+}(r,t)-C_1 v^{\ell}(r,t)\right|
&\leq N^{-1/5}+\big(\lfloor N^{1/10}\rfloor\big)^{-1}\\
&< 3N^{-1/10}
\end{align*}
for $N$ sufficiently large.

If instead $r\geq r_k$ where $k =\lfloor N^{1/10}\rfloor -2$,
then
$$
v^{\ell}(r,t)
\geq \frac k {\lfloor N^{1/10}\rfloor}
\geq 1-2\big(\lfloor N^{1/10}\rfloor\big )^{-1},
$$
and by~\eqref{eq:hydro(*)},
$$
F^{+}(r,t)\geq F^{+}(r_k,t)\geq 1-N^{-1/5}-2\big(\lfloor N^{1/10}\rfloor 
\big)^{-1}.
$$
Hence
\begin{align*}
\big|C_1 F^{+}(r,t)-C_1 v^{\ell}(r,t)\big|
&< 3 N^{-1/10}
\end{align*}
for $N$ sufficiently large.
So for $N$ sufficiently large, if~\eqref{eq:hydro(*)} holds for each $k \in 
\big\{1,\ldots, \lfloor N^{1/10} \rfloor-2\big\}$  then
$$
\sup_{r\geq 0}\big|C_1 F^{+}(r,t)-C_1v^{\ell}(r,t)\big|
< 3N^{-1/10}.
$$
Now by a union bound and Lemma~\ref{lem:Ftilde+}, for $N$ sufficiently large,
\begin{align*}
&\P_{\mathcal X}\bigg(\sup_{r \geq 0}\big|C_1 F^{+}(r,t)-C_1
v^{\ell}(r,t)\big|
\geq 3N^{-1/10} \bigg)\\
&\leq \psub{\mathcal X}{\exists k \in \big\{1,\ldots, \lfloor
N^{1/10}\rfloor-2\big\}:\big|F^{+}(r_k,t)-v^{\ell}(r_k,t)\big|
\geq N^{-1/5}}\\
&\leq N^{1/10}\cdot 13 e^{4t}N^{-6/5}\\
&= 13 e^{4t}N^{-11/10},
\end{align*}
which completes the proof.
\end{proof}

\begin{corr}\label{corr:F+}
There exists $N_0 <\infty$ such that for all $N \geq N_0$, for $\mathcal X \in (\R^d)^N$ and $t >0$, 
\begin{align*} 
\P_{ \mathcal{X} } \left(  \sup_{r \ge 0}  \left( F\uppar N(r,t)-C_1 e^t G_t F\uppar N(r,0)\right)  \geq 3 N^{-1/10} \right) &\leq 13 e^{4t}N^{-11/10}.
\end{align*}
\end{corr}

\begin{proof}
This follows immediately from Lemma \ref{lem:F+} and Proposition \ref{prop:FNdom}, and the definition of $v^\ell$ in~\eqref{vl}.
\end{proof}

As in~\eqref{eq:defv+v-_probpaper}, for $\delta>0$ and $k\in \N_0$, let
\[
v^{k,\delta,+}= (C_1 e^\delta G_\delta)^k  F\uppar N(\cdot,0).
\]
We now apply Corollary~\ref{corr:F+} repeatedly on successive timesteps to prove the following result.

\begin{prop} \label{prop:FN}
There exists $N_0 <\infty$ such that for all $N \geq N_0, \delta>0$, $K\in \N$ and
$\mathcal{X} \in (\R^d)^N$,
$$
\P_{\mathcal X}\left( \sup_{r\geq 0} \left(F\uppar N(r,K\delta)-v^{K,\delta,+}(r)\right) \geq 3K e^{K\delta}N^{-1/10} \right)
\leq 13K e^{4\delta}N^{-11/10}.
$$
\end{prop}
\begin{proof}
For $k\in \{1,\ldots,K\}$ and $r\geq 0$, we can write
\begin{align} \notag
F\uppar N(r,k\delta)-v^{k,\delta,+}(r)
&= F\uppar N(r,k\delta)-C_1 e^\delta G_\delta 
F\uppar N\big(r,(k-1)\delta\big) \\  \label{kdeltadiff}
&\quad + C_1 e^\delta G_\delta 
F\uppar N\big(r,(k-1)\delta\big) \\ \notag
& \quad - C_1 e^\delta G_\delta v^{k-1,\delta,+}(r).
\end{align}
To control \eqref{kdeltadiff}, we define the events
$$
E_k:=\bigg\{ \sup_{r\geq 0} \left( F\uppar N(r,k\delta)-C_1 e^\delta G_\delta 
F\uppar N\big(r,(k-1)\delta\big)\right)
< 3N^{-1/10}\bigg\}
$$
and
\[
E_* = \bigcap_{k=1}^K E_k.
\]
Then on the event $E_*$, for each $k\in \{1,\ldots, K\}$ we have by~\eqref{kdeltadiff} that for $r\ge 0$,
\begin{align*}
F\uppar N(r,k\delta)-v^{k,\delta,+}(r)
& < 3N^{-1/10} + C_1 e^\delta G_\delta 
F\uppar N(r,(k-1)\delta) 
 - C_1 e^\delta G_\delta v^{k-1,\delta,+}(r).
\end{align*}
Note that since $G\ge0$ and $\int_0^\infty  G(y,r,t)\,\diffd y\le1$
by~\eqref{wdef_probpaper} and~\eqref{Gdef_probpaper}, we have that
$G_\delta f-G_\delta g\le \max \big(0,\sup_{r\ge 0}(f(r)-g(r)\big)$
 for any $f,g:[0,\infty)\to \R$.
Moreover,
 $C_1 f(r) -C_1 g(r)\le 
\max \big(0,f(r)-g(r)\big)$ for any $f,g:[0,\infty)\to \R$. 
Therefore for $r\ge 0$,
\[
C_1 e^\delta G_\delta  f(r)-C_1 e^\delta G_\delta  g(r)\le e^\delta \max\Big(0, \, \sup_{r\ge0} \, (f(r)-g(r))\Big).
\]
It follows that
\begin{align*}
F\uppar N(r,k\delta)-v^{k,\delta,+}(r) &< 3N^{-1/10} 
+e^\delta \max\left( 0,\,  \,\sup_{y\geq 
0}\left(F\uppar N(y,(k-1)\delta)-v^{k-1,\delta,+}(y)\right)\right)
\end{align*}
also holds on the event $E_*$. By iterating this argument, it follows that for $k\in 
\{1,\ldots,K\}$,
$$
\sup_{r\geq 0} \left(F\uppar N(r,k\delta)-v^{k,\delta,+}(r)\right)
< 3k e^{k\delta}N^{-1/10}
$$
holds on $E_*$.  

To estimate $\P_{\mathcal X}(E_*^c)$, we use a union bound and Corollary~\ref{corr:F+} with $t = \delta$.  Specifically,
\begin{align*}
\P_{\mathcal X}(E_*^c) =  \psub{\mathcal X}{\,\bigcup_{k=1}^K E_k^c  } \leq \sum_{k=1}^K \P_{\mathcal X}(E_k^c).
\end{align*}
By the Markov property, for $k\in \{1,\ldots, K\}$,
\begin{align*}
\P_{\mathcal X}(E_k^c )  = \E_{\mathcal X} \Big[ \P_{\mathcal X}\big( E_k^c
\;\big| \;\mathcal{F}_{(k-1)\delta}\big)\Big]  =  \E_{\mathcal X} \Big[ H\Big(
X\uppar N\big((k-1)\delta\big)\Big) \Big]
\end{align*}
where $H:(\R^d)^N \to \R$ is defined by
\[
H(\mathcal{X}') =  \P_{ \mathcal{X}'} \left( \sup_{r \geq 0} \left(F\uppar
N(r,\delta) - C_1 e^\delta G_\delta F\uppar N(r,0)\right) \ge 3  N^{-1/10} \right).
\]
Therefore, by Corollary~\ref{corr:F+}, for $N \geq N_0$,
\[
\P_{\mathcal X}(E_*^c) \leq \sum_{k=1}^K 13 e^{4 \delta} N^{-11/10} = 13K e^{4\delta}N^{-11/10}.
\]
The result follows.
\end{proof}

\subsubsection{Lower bound for the proof of Proposition~\ref{prop:NBBM1d}}

We begin by proving that under a suitable coupling, the random function $F\uppar N$ for the $N$-BBM stochastically dominates the random function $F^+$ for the BBM with an initial condition consisting of less than $N$ particles. This result is very similar to the lower bound in Theorem~5.1 of~\cite{DMFPSL_EJP}.

Recall from our definition of the $\preceq$ notation
in~\eqref{eq:preceqdef} in Section~\ref{sec:notations} that
 $$X\uppar N(t)\preceq X^+(t)\qquad\Leftrightarrow\qquad
   F\uppar N(\cdot,t) \ge F^+(\cdot,t).$$
Moreover, the relation $\mathcal{X} \preceq \mathcal{X}^+$ is not affected by the ordering of the points in the vectors $\mathcal{X}$ and $\mathcal{X}^+$.
\begin{prop} \label{prop:lowercouple}
Suppose $\mathcal{X} = \big( \mathcal X_1 ,\dots, \mathcal X_N \big)\in (\R^d)^N$ and $\mathcal{X}^+ = ( \mathcal X_1^+,\dots, \mathcal X_m^+)\in (\R^d)^m$ with $m \leq N$ and such that $\mathcal X\preceq\mathcal X^+$.
There exists a coupling of the $N$-BBM $X\uppar N(t)$ started from $\mathcal X$ and of the
BBM $X^+(t)$ started from $\mathcal X^+$ such that, for $t\ge0$,
\[
X\uppar N(t) \preceq
X^+(t)\qquad\text{if}\quad\big| X^+(t)\big| \le N.
\]
In particular, under that coupling, $F\uppar N(\cdot,t)\ge F^+(\cdot,t)$ if $\big| X^+(t)\big|
\le N$. Then if $f:[0,\infty)\to\R$ is measurable, for $t\ge0$,
\begin{align}
&\P_{ \mathcal{X} } \bigg(  \inf_{r \geq 0} \Big(F\uppar N(r,t) - f(r)\Big) > 0  \bigg) \geq \P_{ \mathcal{X}^+ }  \bigg( \inf_{r \geq 0}\Big( F^{+}(r,t) - f(r)\Big) > 0, \quad |X^+(t)| \leq N \bigg). \label{lowercouple1}
\end{align}
\end{prop}
\begin{proof}
The coupling of the processes $X\uppar N$ and $X^+$ is similar in spirit to
the lower bound in Section~5.4 of~\cite{DMFPSL_EJP}. 

Let $\tau^+_\ell$ for
$\ell\in\N$ denote the successive branch times of the BBM process:
\[
\tau^+_\ell = \inf \big\{ t \geq 0\;:\; |X^+(t)| = m+\ell \big\}.
\]
By induction on $|\mathcal X^+|$, we claim that it is sufficient to find a coupling  of the processes
$X\uppar N$ and $X^+$  on the same probability space with $X\uppar N(0)=\mathcal X$ and $X^+(0)=\mathcal
X^+$ such  that
\begin{equation}
{X}\uppar N(t) \preceq {X}^+(t), \quad  \forall \;\; t \in  \left \{ \begin{array}{ll}  \text{$[0,\tau^+_1]$}, & \quad \text{if} \;\; m < N  \\ \text{$[0,\tau^+_1)$}, & \quad \text{if} \;\; m = N \end{array}  \right.
 \label{compclaim1}
\end{equation}
holds almost surely. 
Indeed, assume that \eqref{compclaim1} holds. If $m=N$,
then the proposition is proved. If $m<N$, then $X\uppar N(\tau_1^+)\preceq
X^+(\tau_1^+)$, and the construction of the coupling can be repeated up to
time $\tau_2^+$ using $X\uppar N(\tau_1^+)$ and $X^+(\tau_1^+)$ as the new
initial particle configurations, by the strong Markov property. By induction, the property $X\uppar N(t)\preceq X^+(t)$ holds for
$t\in[0,\tau^+_{N-m+1})$, where $\tau^+_{N-m+1}$ is the first time at which there are $N+1$ particles in the
BBM $X^+$, and the proposition is proved.

We now show that \eqref{compclaim1} holds.
Set $\tau_0 = 0$, and let $(\tau_i)_{i=1}^\infty$ be the arrival times in
a Poisson process with rate $N$, so that $(\tau_{i+1} - \tau_i)_{i \geq 0}$
is a family of independent $\text{Exp}(N)$ random variables. These $\tau_i$
for $i\ge1$ will define the branch times for the $N$-BBM process $X\uppar N$. The coupling will ensure that $\tau_1^+ =\tau_p $ for some $p\in\N$, where we recall that $\tau_1^+$ is the time of the first branching event in the BBM $X^+$.

We now construct the motion of the particles for $t\in(0,\tau_1)$.  Given $x, x^+ \in \R^d$ with $\|x\| \leq \|x^+\|$, we say that $(B, B^+)$ are a pair of spherically-ordered Brownian motions starting from $(x,x^+)$ if $B$ and $B^+$ are Brownian motions in $\R^d$ (with diffusivity $\sqrt{2}$), starting from $B_0 = x$ and $B ^+_0 = x^+$ and such that, with probability one, $\|B_t\| \leq \|B^+_t\|$ holds for all $t \geq 0$. There are multiple ways to construct such a pair. For example, $B$ and $B^+$ might evolve as independent Brownian motions in $\R^d$ up to the first time $ T$ at which $\|B_T\| = \|B^+_T\|$; after that time they are coupled in such a way that $\|B_t\| = \|B^+_t\|$
for all $t \geq  T$, for example by taking $B_t - B_T = \Theta (B^+_t - B^+_T)$ with $\Theta:\R^d \to \R^d$ being an orthogonal transformation such that $\Theta B^+_T= B_T$.  Alternatively, one could use the skew-product decomposition of Brownian motion (see, for example, the proof of Proposition~2.10 in \cite{BZ18}), driving the radial components of $B$ and $B^+$ by the same Bessel process.

Since the condition $\mathcal{X} \preceq \mathcal{X}^+$ is invariant under
permutation of the indices of points in $\mathcal{X}$ and $\mathcal{X}^+$,
it suffices to assume that the vectors $\mathcal{X}$ and $\mathcal{X}^+$ are ordered in such a way that
\[
\|\mathcal X_k \| \leq \|\mathcal X_k^+\|, \quad \forall \; k \in \{ 1,\dots,m\} ,
\]
where we recall that $m=|\mathcal{X}^+| \leq N$.
(For example, order the vectors $\mathcal X$ and $\mathcal X^+$ so that the points with the lowest indices are the points closest to the origin.)

Then, for $t \in (0,\tau_1)$, we define $(X\uppar N_k(t),X^+_k(t))$ for $k \in \{1,\dots,m\}$ to be $m$ independent pairs of spherically-ordered Brownian motions starting from $(\mathcal X_k,\mathcal X^+_k)$, and for $k \in \{ m+1,\dots,N\}$, the $X\uppar N_k(t)$ are defined to be independent Brownian motions starting from $\mathcal X_k$.
Hence, for $t\in(0,\tau_1)$, 
we have $\big\|X_k\uppar N(t)\big\| \le\big\|X_k^+(t)\big\|$ for $k\in\{1,\ldots,m\}$, which in turn implies that 
$$X\uppar N(t) \preceq X^+(t)\qquad \forall \; t<\tau_1.$$
We now describe the first branching event at time $\tau_1$.

Defining $X\uppar N(\tau_1-)=\big(X\uppar N_1(\tau_1-),\ldots,X\uppar
N_N(\tau_1-)\big)=\lim_{t\nearrow\tau_1} X\uppar N(t)$,
and similarly defining $X^+(\tau_1-)=\lim_{t\nearrow\tau_1} X^+(t)$, we have
\[
\|X\uppar N_k(\tau_1-)\| \leq \|X^{+}_k(\tau_1-)\|, \quad  \forall \; k \in \{ 1,\dots,m\} .
\]
Let $j_1$, a random variable uniformly distributed on $\{1,\dots,N\}$, be the index of the branching particle in the $N$-BBM at time $\tau_1$.
Let $k_1$ denote the index of the particle in $X\uppar N(\tau_1-)$ with maximal distance from the origin.
If $j_1>m$, then at time $\tau_1$, the $N$-BBM branches but the BBM does not. The particle in $X\uppar N$ with index $j_1$ is duplicated; the particle in $X\uppar N$ with index $k_1$ is eliminated. 
More precisely, let
\[X\uppar N_k(\tau_1)=X\uppar N_k(\tau_1-) \text{ for }k\neq k_1, \quad X\uppar
N_{k_1}(\tau_1)=X\uppar N_{j_1}(\tau_1-), \qquad \text{and}\quad 
X^+(\tau_1)=X^+(\tau_1-).
\]
Then for $k\in \{1,\ldots, m\}$,
\[
\|X\uppar N_k(\tau_1)\|\leq  \|X\uppar N_k(\tau_1-)\|\leq \|X^+_k(\tau_1-)\|=\|X^+_k(\tau_1)\|,\]
and, in particular, $X\uppar N(\tau_1) \preceq X^+(\tau_1)$. The construction is then repeated to extend the definition from time $\tau_1$ to $\tau_2$: take new pairs of spherically-ordered Brownian motions to determine the motion of particles up to time $\tau_2$, pick the branching particle $j_2$ in the $N$-BBM uniformly at random, and so on until time $\tau_{i^+}$, where $i^+=\inf\{i\ge 1\;:\; j_i\le m\}$. This time $\tau_{i^+}=\tau_1^+$ is a branching time for both the $N$-BBM and the BBM. Observe that $\tau_1^+\sim \mathop{\mathrm{Exp}}(m)$.

In the construction of the coupling so far we have that for $t<\tau_1^+$,
\[ \|X\uppar N_k(t)\| \le \|X^+_k(t)\| \quad \forall  \; k\in \{1,\ldots, m\},\]
and so, in particular, $X\uppar N(t)\preceq X^+(t)$.
At time $\tau^+_1- = \tau_{i^+}-$, the particles with index $j_{i^+} \in \{
1,\dots,m\}$ from both $X\uppar N$ and $X^+$ branch, and the particle in
$X\uppar N$ of maximal distance from the origin is removed. More precisely,
if $k_{i^+}$ is the index of the particle in $X\uppar N(\tau^+_1-)$ with maximal distance from the origin,
we let
\begin{align*}
X\uppar N_k(\tau_1^+)&= X\uppar N_k(\tau_1^+-) \text{ for }k\neq k_{i^+},
\quad X\uppar N_{k_{i^+}}(\tau_1^+)= X\uppar N_{j_{i^+}}(\tau_1^+-),
\\
\text{and }
\quad X^+(\tau_1^+)&=\Big(X^+_1(\tau_1^+-),\ldots, X_{j_{i^+}}^+(\tau_1^+-),X_{j_{i^+}}^+(\tau_1^+-),\ldots ,X^+_{m}(\tau_1^+-)\Big).
\end{align*}
Suppose $m<N$. Then
$\|X\uppar N_k(\tau^+_1)\|\le \|X^+_k(\tau^+_1-)\|$ $\forall k\in \{1,\ldots,m\}$ with $k\neq k_{i^+}$, and also
$\|X\uppar N_{k_{i^+}}(\tau^+_1)\|\le \|X^+_{j_{i^+}}(\tau^+_1-)\|$.
Moreover, if $k_{i^+}\le m$ then
$\|X\uppar N_{m+1}(\tau^+_1)\|\le \|X\uppar N_{k_{i^+}}(\tau^+_1-)\|\le \|X^+_{k_{i^+}}(\tau^+_1-)\|$.
It follows that ${X}\uppar N(\tau^+_1) \preceq {X}^+(\tau^+_1)$ if $m<N$.  This concludes the coupling construction to achieve \eqref{compclaim1}, and the proof of proposition is now complete.
\end{proof}

We now use Proposition~\ref{prop:lowercouple} to prove a lower bound on $F\uppar N(\cdot,\delta)$.

\begin{lem} \label{lem:F-}
There exists $N_0 <\infty $ such that for all $N \geq N_0$, for all $\mathcal{X}\in (\R^d)^N$ and for all $\delta \in (0,1)$,
\begin{align*}
&\P_{\mathcal{X}}  \bigg(  \inf_{r \ge  0}  \left(F\uppar N(r,\delta)-e^\delta G_\delta  
C_{e^{-\delta} - N^{-1/5}}  F\uppar N(r,0)\right) \leq -4 N^{-1/10}  \bigg) \leq 26e^{4\delta} N^{-11/10}.
\end{align*}
\end{lem}

\begin{proof}
Take $\delta \in (0,1)$ and $\mathcal X \in (\R^d)^N$.
For $r\ge 0$, let $f_{\mathcal X}(r)=N^{-1} |\mathcal X \cap \B(r)|$;
note that if $X\uppar N(0)=\mathcal X$ then $F\uppar N(r,0)=f_{\mathcal X}(r)$.
Let $\mathcal X^+ \subset \mathcal X$ consist of the $\lfloor N(e^{-\delta}-N^{-1/5})\rfloor$ particles in $\mathcal X$ which are closest to the origin.
Let  $(X^+(t),t\ge 0)$ be the BBM started from $X^+(0) = \mathcal{X}^+$.  Observe that $F^+(r,0) =  C_{\lfloor N(e^{-\delta}-N^{-1/5})\rfloor/N} f_{\mathcal X}(r)$, so that 
\begin{equation}
C_{e^{-\delta} - N^{-1/5}} f_{\mathcal X}(r) -N^{-1}\leq F^+(r,0) \leq C_{e^{-\delta} - N^{-1/5}} f_{\mathcal X}(r)\quad\forall r\ge0.
 \label{initialCut}
\end{equation}
Let $R_1$ be the event
\[
R_1 = \left \{ \big|X^+(\delta)\big| \leq N \right \}.
\]
Since  $|\mathcal X^+|\le  N(e^{-\delta}-N^{-1/5})$, one has
$N^{-1}e^\delta |\mathcal X^+|-1\le -N^{-1/5}e^\delta \le -N^{-1/5}$, and
Lemma~\ref{lem:Ftilde+} implies that  for $N$ sufficiently large,
\[
\P_{\mathcal{X}^+} \big( R_1^c  \big)\leq \P_{\mathcal{X}^+}\Big(  N^{-1}\big(|X^+(\delta)| - e^\delta |\mathcal{X}^+|\big) \geq N^{-1/5}\Big) \leq 13e^{4 \delta}N^{-6/5}.
\]
Let $R_2$ be the event
\[
R_2 =  \bigg \{ \sup_{r \ge  0}  \left| C_1 F^{+}(r,\delta)- C_1 v^{\ell}(r,\delta)\right| < 3 N^{-1/10} \bigg\},
\]
where, as in~\eqref{vl} in Section~\ref{subsec:hydro1upper}, we let $v^{\ell}(r,\delta)=e^\delta G_\delta 
F^{+}(r,0) $.
Since $\big|\mathcal X^+\big|>N(e^{-\delta}-N^{-1/5})-1$, one has
$e^\delta N^{-1}|\mathcal X^+|\geq 1-e^\delta N^{-1/5}-e^\delta N^{-1}\geq 1-N^{-1/10}$
if $N$ is sufficiently large that $e (N^{-1/5}+N^{-1})\leq N^{-1/10}$.
Then, by Lemma \ref{lem:F+}, we know that for $N$ sufficiently large,
\begin{align*} 
\P_{ \mathcal{X}^+ } \big(  R_2^c \big) &\leq 13 e^{4\delta}N^{-11/10}.
\end{align*}
Since $e^\delta N^{-1}\big|\mathcal X^+\big|< 1$ we have $v^{\ell}(\cdot,\delta) <1$ (by~\eqref{Fpmean} in the proof of Lemma~\ref{lem:Ftilde+}).
By \eqref{initialCut}, we therefore have that for $r\ge 0$,
\[
C_1 v^{\ell}(r,\delta) = v^{\ell}(r,\delta) 
= e^\delta G_\delta F^+(r,0)\geq e^\delta G_\delta C_{e^{-\delta} - N^{-1/5}} f_{\mathcal X}(r) - e^{\delta}N^{-1}.
\]
On the event $R_1$ we also have $F^+(\cdot,\delta)\le N^{-1}\big|X^+(\delta)\big|\le 1$ and hence $C_1 F^+ (\cdot,\delta)= F^+(\cdot,\delta)$.
This shows that on the event $R_1 \cap R_2$, we have both
$|X^+(\delta)| \leq N$ and
\begin{align*}
\inf_{r \ge  0}  \left(F^{+}(r,\delta)- e^\delta G_\delta 
C_{e^{-\delta} - N^{-1/5}}  f_{\mathcal X}(r)\right) 
&\geq \inf_{r \ge  0}  \left(C_1F^{+}(r,\delta)-C_1 v^{\ell}(r,\delta)\right)-e^\delta N^{-1}
\\&> - 3 N^{-1/10} - e^{\delta} N^{-1}.
\end{align*} 
Note that since $\mathcal X^+\subset \mathcal X$ we have $\mathcal X \preceq \mathcal X^+$.
Therefore, by Proposition \ref{prop:lowercouple}, and taking $N$ sufficiently large that $eN^{-1}\leq N^{-1/10}$,
\begin{align*}
& \P_{\mathcal X} \bigg(  \inf_{r \ge 0} \Big( F\uppar N(r,\delta)-e^\delta G_\delta  
C_{e^{-\delta} - N^{-1/5}}  f_{\mathcal X}(r)\Big) > -4 N^{-1/10} \bigg) \\
&\quad \geq \P_{\mathcal X^+} \bigg(  \inf_{r \ge 0} \Big( F^+(r,\delta)-e^\delta G_\delta  
C_{e^{-\delta} - N^{-1/5}}  f_{\mathcal X}(r)\Big) > -4 N^{-1/10} ,\, \big|X^+(\delta)\big|\leq N\bigg)\\
& \quad  \geq \P_{\mathcal{X}^+}\big(R_1\cap R_2\big)\\
& \quad \geq 1 - \P_{\mathcal{X}^+} \big(  R_1^c  \big) - \P_{\mathcal{X}^+} \big(  R_2^c  \big) \\
& \quad  \geq 1 - 13e^{4 \delta}N^{-6/5} - 13 e^{4\delta}N^{-11/10},
\end{align*}
which completes the proof, since $f_{\mathcal X}(\cdot)=F\uppar N(\cdot,0)$ if $X\uppar N(0)=\mathcal X$.
\end{proof}

As in~\eqref{eq:defv+v-_probpaper}, for $k\in \N_0$ and $\delta>0$, let
\[
v^{k,\delta,-}= (e^\delta G_\delta 
C_{e^{-\delta}})^k F\uppar N(\cdot ,0).
\]
By applying Lemma~\ref{lem:F-} repeatedly, we can prove the following lower bound, which, together with the upper bound in Proposition~\ref{prop:FN}, will allow us to prove Proposition~\ref{prop:NBBM1d}.

\begin{prop}\label{prop:F-}
There exists $N_0 <\infty$ such for all $N \geq N_0$, for all $\mathcal X \in (\R^d)^N$, $\delta\in (0,1)$ and  $K\in \N$,
\begin{align*}
\psub{\mathcal X}{ \inf_{r\geq 0} \left( F\uppar N(r,K\delta)-v^{K,\delta,-}(r)\right)
\leq - 5K e^{K\delta}N^{-1/10} }
&\leq 26K e^{4\delta}N^{-11/10}.
\end{align*}
\end{prop} 

\begin{proof}
For $k \in \{1,\dots,K\}$ and $r \geq 0$, we can write
\begin{align} \label{eq:FNlowersplit}
F\uppar N(r,k\delta)-v^{k,\delta,-}(r) 
&=  F\uppar N(r,k\delta)
-e^\delta G_\delta C_{e^{-\delta}-N^{-1/5}}  F\uppar N\big(r,(k-1)\delta\big)
\notag \\ &\quad
 +e^\delta G_\delta C_{e^{-\delta}-N^{-1/5}} F\uppar N\big(r,(k-1)\delta\big)
-e^\delta G_\delta C_{e^{-\delta}} F\uppar N\big(r,(k-1)\delta\big)
\quad\notag \\ &\quad
+ e^\delta G_\delta C_{e^{-\delta}}  F\uppar N\big(r,(k-1)\delta\big)
-e^\delta G_\delta C_{e^{-\delta}}v^{k-1,\delta,-}(r).
\end{align}
For the second line on the right hand side of~\eqref{eq:FNlowersplit}, note
that $\big\|C_{e^{-\delta}-N^{-1/5}}f-C_{e^{-\delta}}f\big\|_{L^\infty}\leq N^{-1/5}$ for any $f:[0,\infty)\to \R$, and so, using~\eqref{eq:Gcbounds_probpaper},
\begin{equation} \label{eq:FNline2}
\Big\|e^\delta G_\delta C_{e^{-\delta}-N^{-1/5}}  
F\uppar N(\cdot,(k-1)\delta)-e^\delta G_\delta C_{e^{-\delta}}  
F\uppar N(\cdot,(k-1)\delta)\Big\|_{L^\infty}
\leq e^\delta N^{-1/5}.
\end{equation}
For the third line on the right hand side of~\eqref{eq:FNlowersplit}, 
observe that $C_m f(r) -C_m g(r)\ge
\min\big(0,f(r)-g(r)\big)$ for any $m\in(0,1)$ and any $f,g:[0,\infty)\to \R$. Then, since $\min\big(0,f(r)-g(r)\big)\le0$, we have that for any $\delta>0$, $m\in (0,1)$ and $r\ge 0$,
$$G_\delta C_m f(r)-G_\delta C_m g(r)\ge G_\delta \min\Big(0,f(r)-g(r)\Big)
\ge \inf_{y\ge0} \min\Big(0,f(y)-g(y)\Big),$$
where we used from \eqref{wdef_probpaper} and \eqref{Gdef_probpaper} that 
$G\ge0$ and $\int_0^\infty  G(y,r,t)\,\diffd y\le1$.
It follows that
\begin{align}
&\inf_{r\ge0}\Big(
	 G_\delta C_{e^{-\delta}}  F\uppar N\big(r,(k-1)\delta\big) 
	-G_\delta C_{e^{-\delta}}v^{k-1,\delta,-}(r)
\Big) 
\notag\\&\hspace{5cm}\ge \min\bigg(0,\inf_{y\ge0}\Big(F\uppar N\big(y,(k-1)\delta\big)
-v^{k-1,\delta,-}(y)\Big)\bigg).
\label{eq:FNline3}
 \end{align}
To control the first line of the right hand side of~\eqref{eq:FNlowersplit},
for $k\in \N$, define the event
\begin{align*}
E_k&:=\left\{ \inf_{r\geq 0} \left( F\uppar N(r,k\delta) - e^\delta 
G_\delta C_{e^{-\delta}-N^{-1/5}} F\uppar N\big(r,(k-1)\delta\big) \right)
> - 4N^{-1/10}\right\}
\end{align*}
and let
\[
E_* = \bigcap_{k=1}^K E_k.
\]
Then on the event $E_*$, using~\eqref{eq:FNline2} and~\eqref{eq:FNline3}, we have for each $k\in \{1,\ldots,K\}$, for all $r \geq 0$,
\begin{align*}
F\uppar N(r,k\delta)-v^{k,\delta,-}(r)
 &> -4N^{-1/10} -e^\delta N^{-1/5}\\
& \quad + e^\delta \min\left(0, \,  \inf_{y\geq 0} \left(   
F\uppar N(y,(k-1)\delta)- v^{k-1,\delta,-}(y) \right) \right).
\end{align*}
By iterating this bound, it follows that on the event $E_*$, for $k\in 
\{1,\ldots,K\}$,
\begin{equation} \label{eq:E*lower}
\inf_{r\geq 0} \left(F\uppar N(r,k\delta)-v^{k,\delta,-}(r)\right)
> - (4N^{-1/10}+e^\delta N^{-1/5}) k e^{k\delta}.
\end{equation}

To estimate $\P_{\mathcal X}(E_*^c)$, we use a union bound and Lemma~\ref{lem:F-}.  Specifically,
\begin{align*}
\P_{\mathcal X}(E_*^c) =  \psub{\mathcal X}{\bigcup_{k=1}^K E_k^c } \leq \sum_{k=1}^K \P_{\mathcal X}(E_k^c ).
\end{align*}
By the Markov property, for $k\in \{1,\ldots, K\}$,
\begin{align*}
\P_{\mathcal X}(E_k^c)  = \E_{\mathcal X} \Big[ \P_{\mathcal X}\big( E_k^c \;\big|\; \mathcal{F}_{(k-1)\delta}\big)\Big]
 =  \E_{\mathcal X} \Big[H\Big( X\uppar N\big((k-1)\delta\big)\Big) \Big],
\end{align*}
where $H:(\R^d)^N \to \R$ is defined by
\[
H(\mathcal{X}') =  \P_{\mathcal X'} \left( \inf_{r\geq 0} \left(F\uppar N(r,\delta) - e^\delta 
G_\delta C_{e^{-\delta}-N^{-1/5}} F\uppar N(r,0)\right) 
\le - 4N^{-1/10}\right).
\]
Therefore, by Lemma~\ref{lem:F-}, for $N \geq N_0$,
\[
\P_{\mathcal X}(E_*^c) \leq \sum_{k=1}^K 26 e^{4 \delta} N^{-11/10} = 26 K e^{4\delta}N^{-11/10}.
\]
Taking $N$ sufficiently large that $eN^{-1/5}\leq N^{-1/10},$
the result follows by~\eqref{eq:E*lower}.
\end{proof}

\subsubsection{Combining the upper and lower bounds for the proof of Proposition~\ref{prop:NBBM1d}}

We can now complete the proof of Proposition~\ref{prop:NBBM1d}.
Let $v\uppar N$ denote the solution of~\eqref{pbv_probpaper} with initial condition $v_0(r)=F \uppar 
N(r,0)$ for $r\geq 0$.
By Lemma~\ref{lem:vnapprox} we have that
for $\delta>0$, $k\in \N_0$ and $r\geq 0$,
$$
v^{k,\delta,-}(r)=(e^\delta G_\delta C_{e^{-\delta}})^k F\uppar N(r,0)\leq v\uppar N 
(r,k\delta)\leq (C_1 e^\delta G_\delta)^k F\uppar N(r,0)=v^{k,\delta,+}(r)
$$
and
$$
\big\|v^{k,\delta,+}-v^{k,\delta,-}\big\|_{L^\infty}\leq (e^{k\delta}+1)(e^\delta -1).
$$
Therefore, for $N$ sufficiently large, for $\mathcal X \in (\R^d)^N$, $\delta \in (0,1)$ and $K\in \N$, by 
Proposition \ref{prop:FN},
$$
\psub{\mathcal X}{\sup_{r\geq 0}
\Big(F\uppar N(r,K\delta)-v \uppar N(r,K\delta)\Big)
\geq 3K e^{K\delta}N^{-1/10} +(e^{K\delta}+1)(e^\delta-1) }
\leq 13K e^{4\delta}N^{-11/10}
$$
and by Proposition~\ref{prop:F-}, 
\begin{align*}
\psub{\mathcal X}{ \inf_{r\geq 0} \Big( F\uppar N(r,K\delta)-v\uppar N(r,K\delta)\Big)
\leq - 5 K e^{K\delta}N^{-1/10} -(e^{K\delta}+1)(e^\delta-1)}
\leq 26K e^{4\delta}N^{-11/10}.
\end{align*}
It follows that for $N$ sufficiently 
large, for $\mathcal X \in (\R^d)^N$, $\delta \in (0,1)$ and $K\in \N$, 
\begin{align*}
\psub{\mathcal X}{ \sup_{r\geq 0}\Big|F\uppar N(r,K\delta)-v\uppar N(r,K\delta)\Big|
\geq 5K e^{K\delta}N^{-1/10} +(e^{K\delta}+1)(e^\delta-1)}
&\leq 39 K e^{4\delta}N^{-11/10}.
\end{align*}
Take $t>0$, and let $K=\big\lceil N^{1/20}t\big \rceil$ and $\delta=t/K$.
Then $\delta \leq N^{-1/20}$ and so for $N$ sufficiently large (not depending on $t$),
\begin{align*}
5K e^{K\delta}N^{-1/10} +(e^{K\delta}+1)(e^\delta-1)
&\leq 5(N^{1/20}t +1)e^t N^{-1/10}+(e^t+1)(e^{N^{-1/20}}-1)\\
&\leq 5(t+1)e^t N^{-1/20}+4e^t N^{-1/20}\\
&\leq 9e^{2t} N^{-1/20}.
\end{align*}
Also, for $N$ sufficiently large (still not depending on $t$), $39 Ke^{4\delta}N^{-11/10}\leq 
40(t+1)N^{-21/20}$. 
Therefore for $N$ sufficiently large, for $\mathcal X \in (\R^d)^N$ and $t>0$,
\begin{align*}
\psub{\mathcal X}{\sup_{r\geq 0}\Big|F\uppar N(r,t)-v\uppar N(r,t)\Big|
\geq 9e^{2t} N^{-1/20}}
&\leq 40e^t N^{-21/20}.
\end{align*}
This completes the proof of Proposition~\ref{prop:NBBM1d}.

\subsection{Proof of Proposition~\ref{prop:bound on M}} \label{subsec:bound on M}

We begin with the following lemma, which is a consequence of Proposition~\ref{prop:NBBM1d} and a concentration estimate for $F \uppar N(\cdot, 0)$  in the case where $X_1\uppar N  (0),\ldots, X_N \uppar N (0)$ are i.i.d.~with some fixed distribution $\mu_0$. This lemma will be used later to argue that at a fixed time $t$, $F\uppar N(R_t,t)$ is close to 1, where $(u,R)$ is the solution of the free boundary problem~\eqref{pbu_probpaper} with initial condition $\mu_0$.

\begin{lem} \label{cor:iidhydro}
There exists a constant $c_3 \in (0,1)$ such that the following holds.
Suppose $X_1\uppar N (0),$  $ $\ldots,$ $ $X_N \uppar N (0)$ are i.i.d.~with distribution given 
by $\mu_0$.
Let $v$ denote the solution of~\eqref{pbv_probpaper} with initial condition $v_0(r)=\mu_0(\B(r))$.
Then for $N$ sufficiently large, for $t\geq 0$, 
\begin{align*}
\p{  \big\|F\uppar N(\cdot,t)-v(\cdot,t)\big\|_{L^\infty}\geq e^{2t}N^{-c_3}}
\le e^t N^{-1-c_3}.
\end{align*} 
\end{lem}

The difference between Lemma \ref{cor:iidhydro} and Proposition~\ref{prop:NBBM1d} is that the initial condition for $v(\cdot,t)$ in Lemma~\ref{cor:iidhydro} is given by $v_0(r)=\mu_0(\B(r))$, whereas the initial condition for $v\uppar N(\cdot,t)$ in Proposition~\ref{prop:NBBM1d} is given by $v_0(r)=F\uppar N(r,0)$.
 
\begin{proof}
Recall from Proposition~\ref{prop:NBBM1d} that there is a constant $c_1>0$ such that for
$N$ large enough, for $t\ge 0$,
$$
\p{ \left\|F\uppar N (\cdot,t)-v\uppar N(\cdot,t) \right\|_{L^\infty}
\geq e^{2t} N^{-c_1} }
\leq e^t N^{-1-c_1},
$$
where  $v \uppar N$ is the solution of~\eqref{pbv_probpaper} with initial
condition   $v\uppar N_0(r) = F\uppar N(r,0)$. By Lemma~\ref{lem:cont init
cond_probpaper},
\begin{equation} \label{eq:iidhydro1}
\big\|v\uppar N(\cdot,t)-v(\cdot, t)\big\|_{L^\infty} \le e^t
  \big\|F \uppar N(\cdot,0)-\mu_0(\B(\cdot))\big\|_{L^\infty}.
\end{equation}
The function $v_0(r) = \mu_0(\B\big(r)\big)$ is the cumulative distribution function for each of the real-valued random variables $\|X_i \uppar N(0)\|$, $i = 1,\dots,N$, which are independent.  Therefore, it follows immediately from Corollary~1 and Comment~2(iii) of~\cite{Mass90} (which is a sharp, quantitative version of the Glivenko-Cantelli theorem), that
\begin{equation} \label{eq:iidhydro2}
\P\left( \Big\|F \uppar N(\cdot,0)-\mu_0(\B\big(\cdot)\big)\Big\|_{L^\infty}> \epsilon \right)
\le 2e^{-2N \epsilon^2}
\end{equation}
holds for all $\epsilon > 0$ and $N \geq 1$.  

For $\epsilon > 0$ to be chosen, let $E$ be the event
\begin{equation*}
E=\Big\{\big\|F \uppar
N(\cdot,0)-\mu_0\big(\B(\cdot)\big)\big\|_{L^\infty}\le
\epsilon \Big\}.
\end{equation*}
Then, for $c_3>0$  to be determined, $N$ large enough for
Proposition~\ref{prop:NBBM1d} to hold, and $t\ge 0$,
\begin{align*}
&\p{  \big\|F\uppar N(\cdot,t)-v(\cdot,t)\big\|_{L^\infty}\geq e^{2t}N^{-c_3}}
\\&\qquad\qquad\le 
\p{  \big\|F\uppar N(\cdot,t)-v\uppar N(\cdot,t)\big\|_{L^\infty}
   + \big\|v\uppar N(\cdot,t)-v        (\cdot,t)\big\|_{L^\infty}\geq e^{2t}N^{-c_3}}
\\&\qquad\qquad\le 
\p{  \big\|F\uppar N(\cdot,t)-v\uppar N(\cdot,t)\big\|_{L^\infty}
   + \big\|v\uppar N(\cdot,t)-v        (\cdot,t)\big\|_{L^\infty}\geq e^{2t}N^{-c_3}
\;,\; E}+\P(E^c)
\\&\qquad\qquad\le 
\p{  \big\|F\uppar N(\cdot,t)-v\uppar N(\cdot,t)\big\|_{L^\infty}
   + e^t \epsilon \geq e^{2t}N^{-c_3} 
\;,\; E}+ 2e^{-2N \epsilon^2},
\end{align*}
by~\eqref{eq:iidhydro1} and~\eqref{eq:iidhydro2}.
Choose $c_3 \in (0,\min(c_1,1/2))$ and $\epsilon = \frac{1}{2}N^{-c_3}$. Then, for $N$ large enough, one has
$e^{2t}N^{-c_3}-e^t \epsilon \ge e^{2t}N^{-c_1}$ for all $t\ge0$ (it is
sufficient for the inequality to hold at $t=0$).  Therefore
\begin{align*}
\p{  \big\|F\uppar N(\cdot,t)-v(\cdot,t)\big\|_{L^\infty}\geq e^{2t}N^{-c_3}}
\le e^t N^{-1-c_1} + 2e^{-\frac{1}{2} N^{1 - 2 c_3} }.  
\end{align*}
Since $1 - 2c_3 > 0$, we may take $N$ larger if necessary so that $e^t N^{-1-c_1} + 2e^{-\frac{1}{2} N^{1 - 2 c_3} } \le e^t
N^{-1-c_3}$ for all $t\ge0$ (it is
sufficient for the inequality to hold at $t=0$), and the proof is complete.
\end{proof}

Now consider an $N$-BBM $X\uppar N$ started from an initial condition $\mathcal
X\in(\R^d)^N$. Recall the coupling in
Section~\ref{sec:notations} between $X\uppar N$ and $X^+$, 
where $X^+(t)=(X^+_u(t))_{u\in \set+_t}$ is the vector of particle locations at time $t$ in a standard BBM started from the same initial condition $\mathcal
X$, such that under the coupling, for all $t\ge0$,
$X\uppar N (t)\subseteq X^+(t).$
Recall also from~\eqref{eq:XNcontains} that under the coupling, for $t\ge 0$, almost surely
\begin{align} \label{eq:XNcontainsbis}
\{X\uppar N_k (t)\}_{k=1}^N & = \{X^+_u(t)\}_{u\in \setN _t}, \quad \text{where } \quad \setN _t=  \Big\{ u\in \set+_t :
\big\|X^+_{u}(s)\big\| \le M\uppar N_{s}\ \forall s\in [0,t] \Big\},
\end{align}
and
$M\uppar N_s=\max_{k\in \{1,\ldots,N\}} \|X_k\uppar N(s)\|$ is the
maximum distance of a particle in the $N$-BBM from the origin at time $s$.

For Lemmas \ref{ZbigBsmall} to \ref{lemAsup}, we also introduce two quantities. For any particle $u\in \set+ _t$ in the BBM at time $t$, let $B_{u}(s)$ be the displacement 
of that particle at time $s\in[0,t]$ from its location at time $0$:
$$B_{u}(s) = X_{u}^+(s) - X_{u}^+(0).$$
For $\epsilon>0$ and $r>0$, let 
$Z_\epsilon(r)$ denote the number of particles in the BBM at time $\epsilon$ which started (at time 0) from a particle in $\B(r)$:
$$Z_\epsilon(r) = \Big|\Big\{u\in \set+_\epsilon :
X^{+}_{u}(0)\in\B(r)\Big\}\Big|.$$

\begin{lem}\label{ZbigBsmall}
For $\epsilon, r>0$, if the following two conditions hold:
\begin{itemize}
\item $Z_\epsilon(r)> N$,
\item $\displaystyle
\max_{u\in \set+_{2\epsilon} }\sup_{s\in[0,2\epsilon]}\big\|B_{u}(s)\big\| \le
\tfrac13 \epsilon^{1/3}$,
\end{itemize}
then
$$
M\uppar N_{s}\le r +\epsilon^{1/3}\quad \forall
s\in[\epsilon,2\epsilon].$$
\end{lem}
\begin{proof}
Assume the hypotheses hold. Then we must have
 $$\exists s^*\in[0,\epsilon] :\quad M\uppar N_{s^*}\le
r+\tfrac13\epsilon^{1/3}.$$
Indeed, if we had $M\uppar N_{s}>
r+\tfrac13\epsilon^{1/3}$\ \ $\forall s\in [0,\epsilon]$, then by~\eqref{eq:XNcontainsbis} and using that
$\|B_{u}(s)\| =\| X_{u}^+(s) - X_{u}^+(0)\|\le \frac13\epsilon^{1/3}$,
\begin{align*}
\setN_\epsilon &\supseteq \Big\{  u\in \set+ _\epsilon :\,
\big\|X^+_{u}(s)\big\|\le  r +\tfrac13\epsilon^{1/3}\ \forall s\in[0,\epsilon]\Big\}\\
 &\supseteq \Big\{ u\in \set+ _\epsilon :\,
\big\|X^+_{u}(0)\big\|< r \Big\} ,
\end{align*}
which is impossible as the set on the right hand side has size $Z_\epsilon(r)$ and we assumed that $Z_\epsilon(r)>N$.

Then, for any $s\in[s^*,2\epsilon]$, 
for any particle $u\in \setN_s$, its ancestor at time $s^*$
must have been within distance $M\uppar N _{s^*}$ of the origin, i.e.
$\|X^+_{u}(s^*)\|\le M\uppar N _{s^*}$.
Hence
$$M_{s}\uppar N\le M_{s^*}\uppar N + \max_{u \in \set+_{s}} \big\|B_{u}
(s)-B_{u}(s^*)\big\|\le \Big(r+\tfrac13\epsilon^{1/3}\Big)+
2\cdot \tfrac13\epsilon^{1/3},$$
which completes the proof.
\end{proof}

Recall from the definition of $\Gamma$ in \eqref{Gammaxcdef} that for $r>0$ and $\delta \in [0,1)$, $\mathcal
X\in\Gamma(r,1-\delta)$ means that at least a fraction $1-\delta$ of the $N$
particles of the vector $\mathcal X$ are in $\B(r)$, and, in particular, for $t\ge 0$,
\begin{equation}\label{recallGamma}
X\uppar N(t)\in\Gamma(r,1-\delta)\qquad\Leftrightarrow\qquad
F\uppar N(r,t)\ge1-\delta.
\end{equation}

\begin{lem}\label{Zbig}
There exists a constant $c_4>0$ such that for $r>0$ and $\epsilon \in (0,1)$, if
 $\mathcal X\in\Gamma(r,1-\tfrac 14 \epsilon )$
then
$$\P_{\mathcal X}\big(Z_\epsilon(r)\le N\big)\le e^{-c_4 \epsilon N} .$$ 
\end{lem}
\begin{proof}
Note that since each particle in the BBM branches independently at rate
1, and since we assumed that the number $|\mathcal X \cap \B(r)|$ of initial particles in $\B(r)$ is at least $N(1-\frac 14 \epsilon)$, we have
$$Z_\epsilon (r) \ge N(1-\tfrac 14 \epsilon) + \xi$$
where $\xi$ is a Poisson random variable with mean $N(1-\frac 14 \epsilon)\epsilon$.
Hence, for any $c>0$, by Markov's inequality,
\begin{align*}
\P_{\mathcal X}\big(Z_\epsilon (r) \leq N\big) & \le \P(\xi \le \tfrac 14 \epsilon N)
\\& = \P \big( e^{-c\xi} \ge e^{-\frac 14 c\epsilon N}\big)
\\& \le e^{\frac 14 c\epsilon N}\E\big[ e^{-c\xi} \big]
\\& = e^{\frac 14 c\epsilon N +N(1-\frac 14 \epsilon)\epsilon(e^{-c}-1)}
\\&\le e^{N(1-\frac 14 \epsilon)\epsilon(\frac 12 c +e^{-c}-1)},
\end{align*}
where we used  $\frac 14 \epsilon \leq \frac 12 \epsilon(1-\frac 14 \epsilon)$ in the last line.
Fixing $c>0$ sufficiently small that $\tfrac 12 c+e^{-c}-1=-c'<0$, it follows that
\begin{align*}
\P_{\mathcal X}\big(Z_\epsilon (r) \leq N\big)
&\leq e^{-c'\epsilon N(1-\frac 14 \epsilon)}\leq e^{-\frac 12 c' \epsilon N}. \qedhere
\end{align*}
\end{proof}

\begin{lem}\label{lemAsup}
Let $\epsilon=N^{-b}$ for some $b\in (0,1/2)$.
Then for $N$ sufficiently large, for $r>0$,
if $\mathcal X\in \Gamma(r,1-\frac 14 \epsilon)$ then 
\[ \P_{\mathcal X}\bigg(\sup_{s \in [ \epsilon,2\epsilon ]} M\uppar N_{s}>r
+\epsilon^{1/3}\bigg)\le   e^{-\epsilon^{-1/4}}.
\]
\end{lem}

\begin{proof}
Take $\mathcal X\in \Gamma(r,1-\frac 14 \epsilon)$.
By Lemma~\ref{ZbigBsmall}, observe that
\begin{equation}\label{PPP}
\begin{aligned}
\P_{\mathcal X}\bigg( \sup_{s \in [ \epsilon,2\epsilon ]} M\uppar N_{s}>r
+\epsilon^{1/3}\bigg)
\le \quad& \P_{\mathcal X}\big( Z_\epsilon(r) \leq N \big)
\\+&
\P_{\mathcal X}\bigg( \exists u\in \set+_{2\epsilon} :
\sup_{s\in [0,2\epsilon]}\|B_{u}(s)\|> \tfrac 13
\epsilon^{1/3}\bigg).
\end{aligned}
\end{equation}
By
Lemma~\ref{Zbig}, 
the first term on the right hand side is bounded by $e^{-c_4 \epsilon N}$. We focus on the second term. By the many-to-one lemma, recalling that we let $(B_s)_{s\geq 0}$ denote a $d$-dimensional Brownian motion with diffusivity $\sqrt 2$,
$$\P_{\mathcal X}\bigg( \exists u\in \set+_{2\epsilon} :
\sup_{s\in [0,2\epsilon]}\|B_{u}(s)\|> \tfrac 13
\epsilon^{1/3}\bigg) \le Ne^{2\epsilon}\,\P_0 \bigg(\sup_{s\in 
[0,2\epsilon]}\|B_s\|> \tfrac 13
\epsilon^{1/3}\bigg).$$
Letting $\xi_{1,s},\ldots,\xi_{d,s}$ denote the $d$ coordinates of $B_s$, which are themselves independent one-dimensional Brownian motions,
\begin{equation}
\begin{aligned}
\P_0 \left(\sup_{s\in [0,2\epsilon]}\|B_s\|> \tfrac 13
\epsilon^{1/3}\right)
&=\P_0 \left(\sup_{s\in [0,2\epsilon]}\big(\xi_{1,s}^2+\cdots+\xi_{d,s}^2\big)>
\tfrac19\epsilon^{2/3}\right) \label{bsplit1}
\\&\le\P_0 \left(\sup_{s\in [0,2\epsilon]}\xi_{1,s}^2> \tfrac1{9d}\epsilon^{2/3}
\text{ or }\ldots\text{ or }\sup_{s\in [0,2\epsilon]}\xi_{d,s}^2 >
\tfrac1{9d}\epsilon^{2/3}\right)
\\&\le d\, \P_0 \left(\sup_{s\in [0,2\epsilon]}\big|\xi_{1,s}\big|
>\tfrac1{3\sqrt d}\epsilon^{1/3}\right)
\\&\le 4d\,\P_0 \left(\xi_{1,2\epsilon}
>\tfrac1{3\sqrt d}\epsilon^{1/3}\right) 
\\&
\le 4d\exp\left({-\frac{\epsilon^{-1/3}}{72d}}\right),
\end{aligned}
\end{equation}
where the fourth line follows by the reflection principle, and the last line by a Gaussian tail bound.

By~\eqref{PPP} we now have that 
\[
\P_{\mathcal X}\bigg( \sup_{s \in [ \epsilon,2\epsilon ]} M\uppar N_{s}>r
+\epsilon^{1/3}\bigg)
\le 
e^{-c_4 \epsilon N}+N e^{2\epsilon} \cdot 4d \exp\left({-\frac{\epsilon^{-1/3}}{72d}}\right)
\le e^{-\epsilon^{-1/4}}
\]
for $N$ sufficiently large, since $\epsilon = N^{-b}$ and $b\in (0,1/2)$.
\end{proof}
We can now complete the proof of Proposition~\ref{prop:bound on M}.
Recall that we assume that the $N$-BBM is started from $N$ i.i.d.\@ particles
with distribution given by some $\mu_0$, and that
$(u,R)$ denotes the solution to the free boundary
problem~\eqref{pbu_probpaper} with initial condition $\mu_0$.
As in Lemma~\ref{cor:iidhydro}, let $v$ denote the solution of~\eqref{pbv_probpaper} with initial condition $v_0(r)=\mu_0(\B(r))$, and recall from Section~\ref{subsec:1d} that $v(r,t)=\int_{\B(r)}u(x,t)\, \diffd x$ and so, in particular, $v(R_t,t)=1$ for $t>0$.

Take $c_3\in (0,1)$ as in Lemma~\ref{cor:iidhydro}, and let $\epsilon=N^{-c_3/2}$.
Then for $T>0$, by a union bound,
\begin{align*}
&\P\left( \exists t\in [2\epsilon,T]:M\uppar N_t>R_{\epsilon(\lfloor
t /\epsilon
\rfloor -1)}+\epsilon^{1/3}\right)\\
&\quad \le \sum_{k=1}^{\lfloor T/\epsilon\rfloor-1}
\P\left( \sup_{s\in[\epsilon,2\epsilon]} M\uppar N_{\epsilon k +s}>R_{\epsilon
k}+\epsilon^{1/3}\right)
\\
&\quad \le \sum_{k=1}^{\lfloor T/\epsilon\rfloor-1}
\bigg(
\P\big(E_{\epsilon k}^c\big)+ \P\bigg( E_{\epsilon k};\sup_{s\in[\epsilon,2\epsilon]} M\uppar N_{\epsilon k 
+s}>R_{\epsilon
k}+\epsilon^{1/3}\bigg) \bigg).
\end{align*}
The above is of course valid for any choice of events $E_t$ for each $t>0$ but,
in order to use Lemma~\ref{lemAsup}, we
let $E_t=\big\{X\uppar N(t) \in\Gamma(R_t,1-\frac 14 \epsilon)\big\}$.
Then for $N$ sufficiently large that $\frac 14 \epsilon > e^{2T}N^{-c_3}$,
recalling the meaning of $\Gamma$ in~\eqref{recallGamma}
and since $v(R_t,t)=1$, for $t\in (0,T]$,
\begin{align*}
\P\big(E_t^c\big)=\P\big(X\uppar N(t) \not\in\Gamma(R_t,1-\tfrac 14 \epsilon)
\big)
&=
\P\big(F\uppar N(R_t, t) < 1-\tfrac 14 \epsilon \big)\\
&\le \p{\sup_{r\ge 0}|F\uppar N(r,t)-v(r,t)| \ge e^{2T}N^{-c_3}}\\
&\le e^T N^{-1-c_3}
\end{align*}
 by Lemma~\ref{cor:iidhydro}.
For $N$ sufficiently large, 
for $k\in \N$,
by
Lemma~\ref{lemAsup} and the Markov property at time $\epsilon k$ we have
\[
\P\bigg( E_{\epsilon k};\sup_{s\in[\epsilon,2\epsilon]} M\uppar N_{\epsilon k 
+s}>R_{\epsilon
k}+\epsilon^{1/3}\bigg) \le e^{-\epsilon^{-1/4}}.
\]
Therefore
$$
\P\left( \exists t\in [2\epsilon,T]:M\uppar N_t>R_{\epsilon(\lfloor
t /\epsilon
\rfloor -1)}+\epsilon^{1/3}\right)
\le \lfloor T/\epsilon\rfloor\big(e^T N^{-1-c_3}+e^{-\epsilon^{-1/4}}\big)\le N^{-1-\frac 13 c_3}$$
for $N$ sufficiently large. For $\eta \in (0,T)$,
since $(R_t)_{t\in [\eta,T]}$ is continuous (by Theorem~\ref{thm:exists u_probpaper}), for $N$ sufficiently large,
\begin{equation}\label{continuityofR}
R_{\epsilon(\lfloor
t /\epsilon
\rfloor -1)}+\epsilon^{1/3}
\leq R_t +\eta
\quad \forall t\in [\eta,T].
\end{equation}
Therefore for $N$ sufficiently large, 
$$
\p{\exists t\in [\eta,T]:M\uppar N_t>R_t+\eta }\leq N^{-1-\frac 13 c_3},
$$
which completes the proof of Proposition~\ref{prop:bound on M}.

\section{Proof of Theorem \ref{thm:NBBM}: Hydrodynamic limit result for
\texorpdfstring{$u$}{u}}
\label{sec:hydrod}

First we notice that it is sufficient to prove that
there exists $c_5>0$ such that for any $t>0$, $A\subseteq
\R^d$ measurable and $\delta>0$, for $N$ sufficiently large,
\begin{equation} \label{eq:FAclaim}
\P\left( \mu\uppar N(A,t) - \int_A u(x,t)\,\diffd
x\ge\delta\right) \leq N^{-1-c_5}.
\end{equation}
Indeed, since
$\mu\uppar N(A,t) + \mu\uppar N\big({\R^d\backslash A},t\big)=1$ and
$\int_A u(x,t)\,\diffd x +
\int_{\R^d\backslash A}u(x,t)\,\diffd x=\int_{\R^d}u(x,t)\,\diffd x=1$, 
\begin{equation} \label{eq:FAcomplement}
\P\left( \mu\uppar N(A,t) - \int_A u(x,t)\,\diffd
x\le-\delta\right)  =
\P\left( \mu\uppar N\big({\R^d\backslash A},t\big) - \int_{\R^d\backslash A} u(x,t)\,\diffd
x\ge\delta\right).
\end{equation}
Hence it follows from~\eqref{eq:FAclaim} that for $t>0$, $A\subseteq \R^d$ measurable and $\delta>0$, for $N$ sufficiently large,
\[
\P\left( \left| \mu\uppar N(A,t) - \int_A u(x,t)\,\diffd
x\right| \ge\delta\right) \leq 2N^{-1-c_5},
\]
and so by Borel-Cantelli, $\mu\uppar N (A,t)\to \int_A u(x,t)\,\diffd x$ almost surely as $N\to \infty$.
Moreover, for $t>0$ and $\delta>0$, let
$\delta'=1-\int_{\B(R_t-\delta)}u(x,t)\,\diffd x>0$ by Theorem~\ref{thm:exists u_probpaper}.
Then 
\begin{align*}
\p{M\uppar N_t<R_t -\delta}
&=\P\Big(\mu\uppar N\big({\B(R_t-\delta)},t\big)=1\Big)
\\&=\P\bigg(\mu\uppar
N\big({\B(R_t-\delta)},t\big)-\int_{\B(R_t-\delta)}u(x,t)\,\diffd x\ge
\delta'\bigg)
\le N^{-1-c_5}
\end{align*}
for $N$ sufficiently large, by~\eqref{eq:FAclaim}.
Also, by Proposition~\ref{prop:bound on M} with $\eta=\min(\delta,t)$, for $N$ sufficiently large,
\[
\p{M\uppar N_t>R_t +\delta}\le N^{-1-c_2}.
\]
Therefore, by Borel-Cantelli, $M\uppar N_t \to R_t$ almost surely as $N\to \infty$.

It now remains to prove~\eqref{eq:FAclaim}. 
Let $(X^+(t),t\ge 0)$ be a BBM with the same initial particle distribution as the $N$-BBM,
i.e.~such that $X^+(0)=(X^+_i(0))_{i=1}^N$, where $(X^+_i(0))_{i=1}^N$ are i.i.d.~with distribution given by $\mu_0$.
Recall the coupling described in Section \ref{sec:notations} between the $N$-BBM $X\uppar
N$ and the BBM $X^+$ such that under the coupling, for all $t\ge0$,
\[ X\uppar N(t) \subseteq X^+(t). \]

Take $t>0$ and $\eta\in (0,t)$.
We let $\mathcal C_{\eta,t}$ denote the set of locations of particles in
the BBM (without killing) at time $t$ whose ancestors at times
$s\in[\eta,t]$
were always within distance $R_s+\eta$ of the origin:
$$
\mathcal C_{\eta,t}
=\Big\{ X_u^+(t) : u \in \set+_t,\ \|X_{u}^+(s)\|
\le R_s +\eta\ \forall s\in[\eta,t]\Big\}.
$$
Notice that 
if $M\uppar N_s\leq R_s+\eta\ \forall s\in [\eta,t]$, then, by~\eqref{eq:XNcontains}, almost surely
$$
X\uppar N(t)
\subseteq \mathcal C_{\eta,t}.
$$
Therefore, for $A\subseteq \R^d$ measurable and $\delta>0$,
\begin{align} \label{eq:Fudelta}
\P\left( \mu\uppar N(A,t) - \int_{A} u(x,t)\,\diffd
x \ge \delta\right)
&\le\P\left( \exists s\in[\eta,t]: M\uppar N_s> R_s+\eta\right) \notag 
\\&\quad +\ \P \left( \frac1N \Big|\mathcal C_{\eta,t}\cap A\Big| -\int_A 
u(x,t)\,\diffd
x \ge \delta\right) \notag \\
&\le N^{-1-c_2}+\P \left(\frac1N \Big|\mathcal C_{\eta,t}\cap A\Big| -\int_A 
u(x,t)\,\diffd
x \ge \delta\right)
\end{align}
for $N$ sufficiently large (depending on $\eta$ and $t$)
by
Proposition~\ref{prop:bound on M}. We now focus on the second term on the right hand side.

\begin{lem}\label{lem:C to u}For any $t>0$ and $A\subseteq \R^d$ measurable,
$$\lim_{\eta\searrow0}\E\bigg[\frac1N\big| \mathcal C_{\eta,t}\cap
A\big|\bigg]= \int_A u(x,t)\,\diffd x\qquad\text{uniformly in $N$}.$$
\end{lem}
\begin{proof}
We claim that for $y\in \R^d$,
\begin{equation} \label{eq:touchdontcross}
\P_y(B_t \in A,  \|B_s\|\le R_s\,\,\forall s\in(0,t]\big) = \P_y(B_t \in A, \|B_s\|< R_s\,\,\forall s\in(0,t)\big).
\end{equation}
(In words: the probability that the Brownian motion touches  the moving
boundary $R$ at a positive time without crossing it is zero.)
We shall begin by showing that the lemma follows from~\eqref{eq:touchdontcross}, and then prove the claim~\eqref{eq:touchdontcross}.

For $\eta \in (0,t)$, by the many-to-one lemma,
and since at time $0$ the BBM consists of $N$ particles with locations which are random variables with distribution $\mu_0$,
$$\begin{aligned}
\E\Big[\big| \mathcal C_{\eta,t}\cap
A\big|\Big]
&=N e^t\int_{\R^d} \mu_0(\diffd y)\, \P_y\big( B_t\in A,\ \|B_s\|\le
R_s+\eta \  \forall s\in [\eta,t]\big).
\end{aligned}$$
By dominated convergence, and uniformly in $N$,
$$\begin{aligned}
\lim_{\eta\searrow0}\E\bigg[\frac1N\big| \mathcal C_{\eta,t}\cap
A\big|\bigg]
&= e^t\int_{\R^d} \mu_0(\diffd y) \P_y\big(B_t\in A,\, \|B_s\|\le
R_s \  \forall s\in (0,t]\big)
\\
&=e^t\int_{\R^d} \mu_0(\diffd y) \P_y\big( B_t\in A,\, \|B_s\|<
R_s \  \forall s\in (0,t)\big)\\
&=\int_A u(x,t)\, \diffd x,
\end{aligned}$$
where the second equality holds by~\eqref{eq:touchdontcross} and the last equality follows from~\eqref{udef_probpaper}.

It remains to prove~\eqref{eq:touchdontcross}; we shall use the following claim.
Suppose $(\xi_s)_{s\in [0,t]}$ is a continuous path in $\R^d$ with $\xi_0=y$ and $\xi_t=0$,
and suppose $\mathbf e \in \R^d$ is a unit vector.
We claim that there are at most two values of $r\in \R$ such that
\begin{equation} \label{eq:hitslant}
\min_{s\in (0,t)} \big(R_s-\|\xi_s +\tfrac s t r \mathbf e\|\big)=0,
\end{equation}
which we write to mean that the min exists in $(0,t)$ and is equal to 0; in
other words,
$R_s\ge\|\xi_s +\tfrac s t r \mathbf e\|\ \forall
s\in(0,t)$ and $\exists s\in(0,t)$ such that $R_s=\|\xi_s +\tfrac
s t r \mathbf e\|$.
Indeed, we have that
\[
\big\|\xi_s +\tfrac s t r \mathbf e\big\|^2
=\big\|\xi_s -(\xi_s\cdot \mathbf e )\mathbf e\big\|^2 +\big(\xi_s\cdot
\mathbf e +\tfrac st r\big)^2,
\]
and so if~\eqref{eq:hitslant} holds,  then the inequality $R_s\ge \|\xi_s +\tfrac
s t r \mathbf e\|$ implies that, for each $s\in (0,t)$,
\begin{equation} \label{eq:rineqs}
 -\tfrac ts \Big((R_s^2 -\|\xi_s-(\xi_s\cdot \mathbf e ) \mathbf e\|^2 )^{1/2}+\xi_s \cdot \mathbf e \Big)\le 
r\le  \tfrac ts \Big((R_s^2 -\|\xi_s-(\xi_s\cdot \mathbf e ) \mathbf e\|^2 )^{1/2}-\xi_s \cdot \mathbf e \Big).
\end{equation}
Moreover, for any value of $s\in(0,t)$ such that $R_s=\|\xi_s +\tfrac
s t r \mathbf e\|$, 
one of the two inequalities in~\eqref{eq:rineqs} must be an equality.
Therefore
\begin{align*}
r &\in \Big\{ \inf_{s\in (0,t)} \left( \tfrac ts \Big(\big(R_s^2
-\|\xi_s-(\xi_s\cdot \mathbf e ) \mathbf e\|^2 \big)^{1/2}-\xi_s \cdot
\mathbf e \Big) \right),\\
&\qquad \qquad \qquad \sup_{s\in (0,t)} \left( -\tfrac ts \Big(\big(R_s^2
-\|\xi_s-(\xi_s\cdot \mathbf e ) \mathbf e\|^2 \big)^{1/2}+\xi_s \cdot
\mathbf e\Big ) \right) \Big\},
\end{align*}
which establishes the claim that~\eqref{eq:hitslant} holds for at most two
values of $r$.

Now for $y\in \R^d$, under the probability measure $\P_y$, let $(\xi_s)_{s\in [0,t]}$ denote a $d$-dimensional Brownian bridge with diffusivity $\sqrt 2$ from $y$ to $0$ in time $t$. 
Then
\begin{align*}
&\P_y\Big(
\big\{ \|B_s\|\le R_s\,\,\forall s\in(0,t]\big\}
\cap
\big\{\exists s\in(0,t):\|B_s\|=R_s\big\}
\Big)\\
&\qquad\qquad=\Esub{y}{\P_y\Big(
\big\{ \|B_s\|\le R_s\,\,\forall s\in(0,t]\big\}
\cap
\big\{\exists s\in(0,t):\|B_s\|=R_s\big\} \, \Big| B_t
\Big)}\\
&\qquad\qquad= \int_{\R^d} \diffd z \, \Phi_t(y-z) \P_y \bigg(\min_{s\in (0,t)} (R_s-\|\xi_s +\tfrac s t z\|)=0\bigg)\\
&\qquad\qquad= \E_y \left[\int_{\R^d} \diffd z \, \Phi_t(y-z) \indic{\min_{s\in (0,t)} (R_s-\|\xi_s +\frac s t z\|)=0}\right]
\end{align*}
by Fubini's theorem,
and where $\Phi_t(x)=(4\pi t)^{-d/2}e^{-\|x\|^2/(4t)}$ is the heat kernel.
By~\eqref{eq:hitslant} we have that $\min_{s\in (0,t)} (R_s-\|\xi_s +\tfrac s t z\|)\neq 0$ for almost every $z$, and~\eqref{eq:touchdontcross} follows.
\end{proof}

\begin{lem} \label{lem:dhydro2}
For $N$ sufficiently large, for any $A\subseteq \R^d$ measurable and any $0<\eta <t$,
$$
\E\Bigg[\bigg(\frac1N \big|\mathcal C_{\eta,t}\cap A\big|-\E\left[\frac1N \big|\mathcal C_{\eta,t}\cap A\big|\right]\bigg)^4\Bigg]\le13 e^{4t}N^{-2}.
$$ 
\end{lem}
\begin{proof}
Recall that we let $X^+(0)=(X^+_i(0))_{i=1}^N$, where $(X^+_i(0))_{i=1}^N$ are i.i.d.~with distribution given by $\mu_0$.
As in the proof of Lemma~\ref{lem:Ftilde+}, denote by $X^{+,i}$ the family
of particles descended from the $i$-th particle in the initial configuration $X^+(0)$. The $X^{+,i}$ form a family of independent BBMs, and for
each $i$ the process $X^{+,i}$ is started from a single particle at location
$X^+_i(0)$.
Fix $0<\eta <t$,  write $n_i=|X^{+,i}(t)|$ for the number of particles descended from
$X^+_i(0)$ at time $t$ and introduce $n_{i,A}$ as the number of
particles in $\mathcal C_{\eta,t}\cap A$ which are descendants of particle $X^+_i(0)$:
$$ n_{i,A} = \Big| \mathcal C_{\eta,t} \cap A\cap X^{+,i}(t)\Big| .$$
Then $|\mathcal C_{\eta,t}\cap A|= \sum_{i=1}^N n_{i,A}$ and $(n_{i,A})_{i=1}^N$ are i.i.d., so
\begin{align*}
&\E\bigg[\Big(\big|\mathcal C_{\eta,t}\cap A\big|-\E\left[ \big|\mathcal C_{\eta,t}\cap A\big|\right]\Big)^4\bigg]\\
&\quad =\E\Bigg[\bigg(\sum_{i=1}^N \big(n_{i,A}-\E[ n_{i,A}]\big)\bigg)^4\Bigg]\\
&\quad =\sum_{i=1}^N \E\Big[\big(n_{i,A}-\E[ n_{i,A}]\big)^4\Big]
+6\sum_{\substack{i,j=1\\i<j}}^N
\var\big(n_{i,A}\big)\var\big(n_{j,A}\big).
\end{align*}
By the same argument as in~\eqref{eq:hydroA1},~\eqref{eq:hydroA2} and~\eqref{eq:hydroA} in the proof of Lemma~\ref{lem:Ftilde+},
\[
\E\Big[\big(n_{i,A}-\E[ n_{i,A}]\big)^4\Big]\le
2\E\big[n_i^4\big]\le 48 e^{4t}
\qquad \text{and}\qquad
\var\big(n_{i,A} \big)\le \E\big[n_i^2\big]\le 2e^{2t}.
\]
Therefore
\begin{align*}
\E\bigg[\bigg(\big|\mathcal C_{\eta,t}\cap A\big|-\E\left[ \big|\mathcal C_{\eta,t}\cap A\big|\right]\bigg)^4\bigg]
&\leq N\cdot 48 e^{4t}+3N(N-1)\big(2e^{2t}\big)^2\\
&\leq 13 e^{4t}N^2
\end{align*}
for $N$ sufficiently large.
\end{proof}

We can now conclude; for fixed $t>0$, $A\subseteq \R^d$ measurable and $\delta>0$, let $\eta>0$ be sufficiently small
that, by Lemma~\ref{lem:C to u},
$$\bigg|\frac1N \E\Big[\big|\mathcal C_{\eta,t}\cap A\big|\Big]-\int_A
u(x,t)\,\diffd x\bigg|<\frac\delta 2\quad\forall N\in \N.$$
Then for $N$ sufficiently large,
$$\begin{aligned}
\P \left( \frac1N \big|\mathcal C_{\eta,t}\cap A\big| -\int_A u(x,t)\,\diffd
x \ge \delta\right)
&\le
\P \left( \frac1N \big|\mathcal C_{\eta,t}\cap A\big| -\frac1N\E\Big[\big|
\mathcal C_{\eta,t}\cap A\big|\Big]>\frac\delta2
\right)
\\ &\le \frac{16}{\delta^4}\cdot 13 e^{4t}N^{-2},
\end{aligned}
$$
by Lemma~\ref{lem:dhydro2} and Markov's inequality.
By~\eqref{eq:Fudelta}, it follows that for $N$ sufficiently large,
\begin{align*}
\p{\mu\uppar N (A,t)-\int_Au(x,t)\,\diffd x \geq \delta}
&\leq N^{-1-c_2}+16 \delta^{-4}\cdot 13 e^{4t}N^{-2}\\
&\leq N^{-1-\frac 12 c_2},
\end{align*}
for $N$ large enough, which establishes~\eqref{eq:FAclaim} and completes the proof of Theorem~\ref{thm:NBBM}.

\section{Proof of Proposition~\ref{prop:ssp} and of Theorems~\ref{thm:piNexists},~\ref{thm:sspd} and~\ref{cor:ssp}} \label{sec:ssp}

We begin by proving Proposition~\ref{prop:ssp}.
We shall first describe heuristically how the proof works.  Recall from \eqref{Gammaxcdef} that for $K>0$ and $c\in (0,1]$,
\[
\Gamma(K, c) = \left\{  \mathcal X\in (\R^d)^N\,:\, \frac1N \left|\{ i: \| \mathcal X_i \| < K \}\right|\geq c\right\}
\]
is the set of ``good'' initial particle configurations that have at least a fraction $c$ of particles within distance $K$ of the origin (the dependence on $N$ is implicit). Recall from \eqref{cumdistrdefn} that $F\uppar N(\cdot ,t)$ is the empirical cumulative distribution of the particles at time $t$. Let us explain how we go about proving that for $\epsilon>0$ there exist $N_\epsilon,$ $T_\epsilon$ such that
$$
\P_{\mathcal X} \Big(  \sup_{r\ge 0} |F\uppar N(r,t) -V(r)|\ge \epsilon  \Big) < \epsilon
$$ 
holds for any $N\ge N_\epsilon, t\ge T_\epsilon$ and $\mathcal X\in \Gamma(K,c)$.  

Let $K_0=3$, and take $c_0$ as in Proposition~\ref{prop:movemass} and $t_\epsilon=t_\epsilon(c_0,K_0)$ as in Proposition~\ref{prop:conv}. We are going to show that there exists a time $t_a=t_a(c,K)$  independent of $N$ such that, for $t$ large enough,
\begin{align*}
\mathcal X \in\Gamma(K,c) &\implies X\uppar N(t_a)\in \Gamma(K_0,c_0)\ \text{with high probability}\\
&\implies  X\uppar N(t-t_\epsilon)\in \Gamma(K_0,c_0)\ \text{with high probability}\\
&\implies  F\uppar N(\cdot ,t)\text{ is close to }V\ \text{with high probability}.
\end{align*}
For the first step, we use Proposition~\ref{prop:movemass} to choose $t_a$ such that $v\uppar N(r,t_a) \ge 2 c_0 \indic{r\ge K_0}\ \forall r\ge0$, where
$v\uppar N$ denotes the solution of the obstacle problem~\eqref{pbv_probpaper} with initial condition $v_0(r)= F\uppar N(r,0) = N^{-1} |\mathcal X \cap \B(r)|$. Then Proposition~\ref{prop:NBBM1d} tells us that $X\uppar N(t_a)\in \Gamma(K_0,c_0)$ with high probability for $N$ large enough.

The second step will be provided by Proposition~\ref{prop:muNlargetime} below, with Proposition~\ref{prop:muNcjump} as an intermediate result. 

For the third step, let $\tilde v$ denote the solution of the obstacle problem~\eqref{pbv_probpaper} with initial condition $\tilde v_0(r)= F\uppar N(r,t-t_\epsilon)$. From the second step,  $X\uppar N(t-t_\epsilon)\in \Gamma(K_0,c_0)$ with high probability, which is equivalent to $\tilde v_0(K_0)\ge c_0$. Proposition~\ref{prop:conv} then gives us that $\tilde v(\cdot,t_\epsilon)$ is close to $V(\cdot)$. Furthermore, Proposition~\ref{prop:NBBM1d}, together with the Markov property at time $t-t_\epsilon$, implies that $F\uppar N(\cdot,t)$ is close to $\tilde v(\cdot,t_\epsilon)$, which in turn is close to $V(\cdot)$.

We start by proving the two propositions needed for the second step, and then we prove Proposition~\ref{prop:ssp}. The first proposition implies that 
if $X\uppar N(0) \in \Gamma(m,c)$ for some $m$ and $c$,
then, for $j$ large and $t$ not too large, with high probability $X\uppar N(t) \in \Gamma(m+j,c)$.

\begin{prop} \label{prop:muNcjump}
For any $c\in(0,1)$, any $N\in\N$, any $t\ge0$ and any $m>0$, if $\mathcal X\in\Gamma(m,c)$, then
\[
\P_{\mathcal X}\big(  X\uppar N(t) \not \in \Gamma(m+j,c) \big)
\leq 4d Ne^{t} e^{-j^2 /(36 d t)} \quad \text{for all}\; j>0.
\]
\end{prop}
\begin{proof}
The argument, in particular the final Gaussian tail estimate, is very similar to the one used in the proof of Lemma~\ref{lemAsup}. Recall the coupling in Section \ref{sec:notations} between the
$N$-BBM $X\uppar N$ and the BBM $X^+$ such that under the coupling, $X\uppar N(0)=X^+(0)$ and for $t\ge0$, almost surely
\begin{equation}
\{X\uppar N_k (t)\}_{k=1}^N = \{X^+_u(t)\}_{u\in \setN _t}, \quad \text{where } \quad \setN _t=  \Big\{ u\in \set+_t :
\big\|X^+_{u}(s)\big\| \le M\uppar N_{s}\ \forall s\in [0,t] \Big\}.
\label{eq:couplingXN}
\end{equation}

Fix $t\ge 0$ and $j>0$, and define the event
\[
E=\bigg\{
\max_{u\in \set+_t }\sup_{s\in [0,t]}
\big\|X_{u}^+(s)-X_{u}^+(0)\big\| < \tfrac13j
\bigg\}.
\]
We claim that on the event $E$, if $X\uppar N(0)\in \Gamma(m,c)$ for some
$c\in(0,1)$ and $m>0$, then $X\uppar N(t)\in \Gamma(m+j,c)$. The proof of the claim is split into two cases, depending on the stopping time 
\[
\tau=\inf\big\{s\geq 0:M\uppar N_s\leq m+\tfrac 13 j\big\}.
\]
We first consider the case $\tau>t$, meaning that $M\uppar N _s> m+\frac 13 j$ $\forall s\le t$.
By~\eqref{eq:couplingXN}, on the event $E\cap \{\tau>t\}$ we have
\begin{align*}
X\uppar N(t) 
&\supseteq
\Big\{ X_u^+(t):u \in \set+_t,\
\big\|X^+_{u}(s)\big\|<m+\tfrac13j\ \forall s\in [0,t] \Big\}
\\
&\supseteq
\Big\{ X_u^+(t):u \in \set+_t,\
\big\|X^+_{u}(0)\big\|<m \Big\},
\end{align*}
where we used   $\big\|X_{u}^+(s)-X^+_{u}(0)\big\| < \frac13j$ in the last
line. Now using that $\big\|X_{u}^+(t)-X^+_{u}(0)\big\| <  j$ on the event $E$, it follows that if $X\uppar N (0)\in \Gamma(m,c)$ then
$X\uppar N(t)\in \Gamma(m+j,c)$.

We now consider the second case, $\tau\leq t$. Take $u\in \setN_t$. On the event $E\cap \{\tau \leq t\}$, $\|X_{u}^+(\tau)\|\le M_\tau\uppar N \le m+\frac13j$, and by the triangle inequality,
\begin{align*}
\big\|X^+_u(t)\big\| 
\le \big\|X_{u}^+(t)-X_{u}^+(0)\big\|
+\big\|X_{u}^+(\tau)-X_{u}^+(0)\big\|
+ \big\|X_{u}^+(\tau)\big\|
< m+j
\end{align*}
by the definition of the event $E$.
By~\eqref{eq:couplingXN}, this shows that $\|X\uppar N_k(t)\| < m+j$ $\forall k\in \{1,\ldots, N\}$, and, in particular, 
$X\uppar N(t)\in \Gamma(m+j,c)$.
This completes the proof of the claim.

Hence for $\mathcal X \in \Gamma(m,c)$,
\begin{align*}
\P_{\mathcal X}\left(  X\uppar N(t) \not \in \Gamma(m+j,c) \right) \leq \P_{\mathcal X}(E^c)
&\le N e^{t}\P_0\bigg(\sup_{s\in [0,t]}\|B_s\|\ge \tfrac 13 j \bigg)
\end{align*}
by the many-to-one lemma. Letting $\xi_{1,s},\ldots,\xi_{d,s}$ denote the coordinates of $B_s$, we have
\begin{align*}
\P_{\mathcal X}\left(  X\uppar N(t) \not \in \Gamma(m+j,c) \right)
&\le N e^{t}\P_0 \left(\sup_{s\in [0,t]}\xi_{1,s}^2\ge \tfrac 1 {9 d}j^2
\text{ or }\ldots\text{ or }\sup_{s\in [0,t]}\xi_{d,s}^2 \ge 
\tfrac 1 {9 d}j^2 \right)
\\
&\le N e^{t}d\, \P_0 \left(\sup_{s\in [0,t]}\big|\xi_{1,s}\big|
\ge \tfrac 1 {3 \sqrt d}j\right)
\\
&\le N e^{t} \cdot 4d\,\P_0 \left(\xi_{1,t} \ge \tfrac 1 {3 \sqrt d}j\right)
\\&
\le 4d N e^{t}e^{-j^2/(36 dt)},
\end{align*}
where the third inequality follows by the reflection principle and the fourth by a Gaussian tail estimate.
\end{proof}
We now show that 
for $c_0 \in (0,1)$ and $K_0>0$ appropriately chosen,
if $N$ and $t$ are sufficiently large, 
if $X\uppar N(0)\in \Gamma(K_0,c_0)$ then
$X\uppar N (t)\in \Gamma(K_0,c_0)$ with high probability.
\begin{prop} \label{prop:muNlargetime}
Take $t_0>1$ and $c_0>0$ as in Proposition~\ref{prop:movemass}, and fix $K_0=3$.
For $\epsilon>0$, there exist $N'_\epsilon<\infty$ and $t'_\epsilon<\infty$ such 
that for $N\geq N'_\epsilon$, the following holds. For $t\ge t'_\epsilon$, 
\begin{align} 
&\inf_{\mathcal X\in \Gamma(K_0,c_0)}\psub{\mathcal X}{X\uppar N(t)\in \Gamma(K_0,c_0)}
\ge 1- \epsilon.
\label{eq:MC1}
\intertext{Furthermore, for $t_1\in [t_0,2t_0]$ and any $K>0$,}
 &\sup_{\mathcal X\in \Gamma(K,c_0)}\E_{\mathcal X}\left[\inf\{n\geq 1: X\uppar N(nt_1)\in \Gamma(K_0,c_0)\} \right]<\infty.
 \label{eq:MC3}
\end{align}
\end{prop}
Note that for any initial condition $\mathcal X \in
(\R^d)^N$, we have that $\mathcal X \in \Gamma(K,c_0)$ where $K=\max_{i\le N} \|\mathcal X_i\|+1$, and so it follows immediately from~\eqref{eq:MC3} that
for $N$ sufficiently large, for $\mathcal X \in (\R^d)^N$ and $t_1\in [t_0,2t_0]$,
\begin{equation} \label{eq:MC2}
\psub{\mathcal X}{ \inf\{n\geq 1: X\uppar N(nt_1)\in \Gamma(K_0,c_0)\}<\infty}=1.
\end{equation} 
\begin{proof}
The proof uses Propositions~\ref{prop:NBBM1d},~\ref{prop:muNcjump} 
and~\ref{prop:movemass} to establish a coupling with a Markov chain. Take $\delta\in (0,1/16)$ sufficiently small that $\frac{1}{1+15\delta}\geq 1- 
\frac 12 \epsilon$.
Suppose $N'_\epsilon$ is sufficiently large that for $N\geq 
N'_\epsilon$ we have 
\begin{equation} \label{eq:NassumptionsGamma}
4dNe^{2t_0}  
e^{-(\log N)^{4/3}/(144 dt_0)}<\tfrac 12 \delta
\quad \text{and} \quad
 \frac{(\log N)^{2/3}}{144dt_0}>\log 2.
\end{equation}
Also, recalling the definition of $c_1$ in Proposition~\ref{prop:NBBM1d}, suppose that $N'_\epsilon$ is sufficiently  large that for $N\ge N'_\epsilon$, Proposition~\ref{prop:NBBM1d} holds,
\begin{equation}
 \label{asump1}
 e^{4t_0}N^{-c_1}\le c_0 \quad \text{and} \quad e^{2t_0}N^{-1-c_1}\le \delta 2^{-(\log N)^{2/3}-1}.
\end{equation}
Take $N\geq N'_\epsilon$ and $t_1\in[t_0,2t_0]$. 
We first show that
\begin{equation}
\label{forallj}
\mathcal X\in\Gamma(m,c_0)\implies\psub{\mathcal X}{X\uppar N(t_1) \notin \Gamma(m+j-1,c_0)}
\leq \delta 2^{-j}\quad\forall j,m\in\N_0\text{ with }m\ge K_0 .
\end{equation}
Take $m\in \N$ with $m\geq K_0$ and suppose  $\mathcal X\in
\Gamma(m,c_0)$. 
We first assume that $j\in \N$ with $j \geq (\log N)^{2/3}+1$. Then by Proposition~\ref{prop:muNcjump}, the fact that $t_1\in[t_0,2t_0]$ and then by~\eqref{eq:NassumptionsGamma}
we have
$$
\psub{\mathcal X}{X\uppar N (t_1) \notin \Gamma(m+j-1,c_0)}
\leq 4dNe^{2t_0}  e^{-(j-1)^2 /(72 dt_0)} \leq \tfrac 12 \delta e^{-(j-1)^2
/(144 d t_0)}\leq \delta 
2^{-j}.
$$
We now consider the case $j \leq (\log N)^{2/3}+1$.
By Proposition~\ref{prop:movemass}, for $v\uppar N$ the solution 
of~\eqref{pbv_probpaper} with $v_0(r)= N^{-1} |\mathcal X \cap \B(r)|$, we have that $v\uppar N(m-1,t_1)\geq 2c_0$.
Therefore by Proposition~\ref{prop:NBBM1d}, 
since, by~\eqref{asump1}, $e^{2t_1}N^{-c_1}\le c_0$,
\begin{align*}
\psub{\mathcal X}{X\uppar N (t_1) \notin \Gamma(m-1,c_0)}
&\leq \psub{\mathcal X}{ \big|F\uppar N(m-1,t_1)-v\uppar N(m-1,t_1)\big|\geq c_0 }\\
&\leq e^{t_1}N^{-1-c_1}.
\end{align*}
In particular, this and the condition \eqref{asump1} imply that for $j\in \N_0$ with $j\leq (\log N)^{2/3}+1$,
\begin{align*}
\psub{\mathcal X}{X\uppar N(t_1) \notin \Gamma(m+j-1,c_0)}
\le \psub{\mathcal X}{X\uppar N(t_1) \notin \Gamma(m-1,c_0)}
\le e^{t_1}N^{-1-c_1}\leq \delta 2^{-j},
\end{align*}
and \eqref{forallj} is proved.

Let us define the sequence of random variables $(\theta_n)_{n=0}^\infty$ by
\[
\theta_n = \min\Big\{i\in \N_0 \; :\;X\uppar N (nt_1)\in \Gamma(K_0 + i,
c_0)\Big\} = \min\Big\{i\in \N_0 \;:\; F\uppar N(K_0 + i, nt_1)\ge c_0\Big\}.
\]
Although $\theta_n$ itself is not a Markovian process, \eqref{forallj} and the Markov property applied to $X\uppar N$ implies that for all $\mathcal X \in (\R^d)^N$ and $n,i,j \in \N_0$, 
\begin{align} \label{eq:probjumpMC}
\theta_n\le i \implies \psub{\mathcal X}{ \theta_{n+1} \geq i+j \;|\; \mathcal F_{nt_1} }\leq \delta 2^{-j}.
\end{align}
Define a Markov chain $(Y_n)_{n=0}^\infty$ on $\N_0$ as follows.
For $n\in \N_0$ and $i,j\in \N_0$, let $$\P\big(Y_{n+1}=j \;\big|\; Y_n=i
\big) 
=p_{i,j},$$ where
$$ p_{0,j}=\begin{cases}
1-\delta \quad &\text{if }j=0\\
\delta 2^{-j} \quad &\text{if }j\geq 1
\end{cases}
\qquad\qquad\text{and, for $i\ge1$, }\quad
 p_{i,i+j}=\begin{cases}
1-2\delta \quad &\text{if }j=-1\\
\delta 2^{-j} \quad &\text{if }j\geq 0.
\end{cases}
$$
Suppose for $K>0$ that $\mathcal X\in \Gamma(K,c_0)$.
Then by~\eqref{eq:probjumpMC},
conditional on $X\uppar N(0)=\mathcal X$ and $Y_0=\max(0,\lceil
K-K_0\rceil)$, 
we can couple $(X\uppar N (nt_1))_{n=0}^\infty $ and $(Y_n)_{n=0}^\infty$ in such a way that almost surely, $\theta_n \leq Y_n$ holds for each $n \in \N_0$, which means that
\[
X\uppar N (nt_1)\in \Gamma(K_0+Y_n,c_0).
\]

For $j\in \N_0$,
introduce $m_j\ge1$ as the expected number of steps needed for $Y_n$ to
reach zero starting from $Y_0=j$:
$$m_j:= \E\left[\left.\inf\{n\geq 1 :Y_n=0\}\;\right|\; Y_0=j\right].$$
Then for  $n\in \N_0$ and $\mathcal X\in \Gamma(K,c_0)$, the coupling
implies
\begin{equation} \label{eq:muNY}
\psub{\mathcal X}{X\uppar N (nt_1)\notin \Gamma (K_0,c_0)} = \psub{\mathcal X}{\theta_n > 0}
\leq \P\big(Y_n \neq 0\;\big|\; Y_0=\lceil K-K_0\rceil\vee0 \big)
\end{equation}
and
\begin{equation} \label{eq:(A)MCcomp}
\E_{\mathcal X}\left[\inf\{n\geq 1: X\uppar N (nt_1)\in \Gamma (K_0,c_0)\} \right]
\leq m_{\lceil K-K_0\rceil\vee0 }.
\end{equation}
Note also that
\begin{align}\label{eq:m0}
m_0
&=1-\delta +\sum_{j=1}^\infty \delta 2^{-j}
m_j.
\end{align}
We now bound $m_j$ for $j\in \N$.
Let $(A_i)_{i=1}^\infty$ be i.i.d.~with $A_i\sim \text{Bernoulli}(2\delta)$, and 
let $(G_i)_{i=1}^\infty$ be i.i.d.~geometric random variables with
$\P(G_1=k)=2^{-{k-1}}$ for $k\in\N_0$.
Then for $j\in \N$,
\begin{align} \label{eq:EhitGA}
m_j
&=\E\bigg[\inf\bigg\{n\geq 1 :j + \sum_{i=1}^n\big(A_iG_i-(1-A_i)\big)\leq
0\bigg\}\bigg] \notag \\
&=\E\bigg[\inf\bigg\{n\geq 1 :\sum_{i=1}^n A_i(G_i+1)\leq n-j\bigg\}\bigg] \notag \\
&\leq 1+\sum_{k=1}^\infty \p{\sum_{i=1}^k A_i (G_i+1)>k-j},
\end{align}
since for a random variable $Z$ taking values in $\N_0$, we have that
$\E[Z]=\sum_{k=0}^\infty \p{Z> k}$, and since for $k\ge 1$,
\[
\P\bigg({\inf\bigg\{n\geq 1 :\sum_{i=1}^n A_i(G_i+1)\leq
n-j\bigg\}>k}\bigg)
\le \P\bigg(\sum_{i=1}^k A_i (G_i+1)>k-j\bigg).
\]
For $k\in \N$ and $\lambda>0$, by Markov's inequality,
\begin{align*}
\p{\sum_{i=1}^k A_i(G_i+1)> k-j}
&\leq e^{-\lambda (k-j)} \E\left[e^{\lambda \sum_{i=1}^k A_i(G_i+1)}\right]\\
&=e^{-\lambda (k-j)} \E\left[e^{\lambda A_1(G_1+1)}\right]^k.
\end{align*}
For $\lambda \in (0,\log 2)$,
\begin{align*}
\E\left[e^{\lambda A_1(G_1+1)}\right]&=2\delta \frac{\frac 12 e^\lambda}{1-\frac 12 
e^\lambda}+1-2\delta .
\end{align*}
Hence letting $\lambda  = \log (3/2)$,
\begin{align*}
\E\left[e^{\log (3/2) A_1(G_1+1)}\right]&=1+4\delta <\tfrac 54
\end{align*}
since we chose $\delta<\frac{1}{16}.$
It follows that for $k\in \N$,
\begin{align*}
\p{\sum_{i=1}^k A_i(G_i+1)> k-j}
&\leq \left( \tfrac 32 \right)^j \left( \tfrac 56 \right)^k.
\end{align*}
Hence by~\eqref{eq:EhitGA},
\begin{align} \label{eq:Ehitsum}
m_j
&\leq 1+\left( \tfrac 32 \right)^j \sum_{k=1}^\infty\left( \tfrac 56 \right)^k 
=1+5 \left( \tfrac 32 \right)^j.
\end{align}
Therefore by~\eqref{eq:m0},
\begin{align*}
m_0 &\leq 1-\delta +\sum_{j=1}^\infty \delta 2^{-j}
\left(1+5 \left( \tfrac 32 \right)^j\right)
=1+15\delta.
\end{align*}
It follows that $(Y_n)_{n=0}^\infty$ is positive recurrent, and since it is also irreducible and
aperiodic, by convergence to equilibrium for Markov chains we have that as $n\to \infty$,
$$
\p{Y_n=0\;|\;Y_0=0}\to \frac{1}{m_0}\geq \frac{1}{1+15\delta}.
$$
Since $\frac{1}{1+15\delta}\geq 1-\frac 12 \epsilon$, there exists 
$n_\epsilon<\infty$ such that for $n\geq n_\epsilon$,
$$
\p{Y_n\neq 0\;|\;Y_0=0}<\epsilon,
$$
and so by~\eqref{eq:muNY}, for $\mathcal X \in \Gamma(K_0,c_0)$ and $n\geq n_\epsilon$,
\begin{equation} \label{eq:nearlyMC1}
\psub{\mathcal X}{X\uppar N (nt_1)\notin \Gamma (K_0,c_0)}
\leq \epsilon.
\end{equation}
Let $t'_\epsilon = \max(n_\epsilon t_0,2t_0)$. Then for $t\ge t'_\epsilon$, we have that 
$\lfloor t/t_0 \rfloor \ge n_\epsilon$ and $ t / \lfloor t/t_0 \rfloor\in [t_0,2t_0]$.
Therefore~\eqref{eq:MC1} follows from~\eqref{eq:nearlyMC1} with $t_1=t / \lfloor t/t_0 \rfloor$ and $n=\lfloor t/t_0 \rfloor$.

Finally, note that by~\eqref{eq:(A)MCcomp} and~\eqref{eq:Ehitsum}, for $K>0$,
if $\mathcal X \in \Gamma(K,c_0)$ then
\begin{align*}
\E_{\mathcal X}\left[\inf\{n\geq 1: X\uppar N(nt_1)\in \Gamma(K_0,c_0)\}\right]
&\leq 1+5\left(\tfrac 3 2 \right)^{\lceil K -K_0\rceil \vee 0},
\end{align*}
which establishes~\eqref{eq:MC3} and completes the proof.
\end{proof}

\begin{proof}[Proof of Proposition~\ref{prop:ssp}]
Take $t_0$ and $c_0$ as in Proposition~\ref{prop:movemass}, and fix $K_0=3$. Take $\epsilon >0$ and take $N'_\epsilon$ and $t'_\epsilon $ as in Proposition~\ref{prop:muNlargetime}. 
Take $t_\epsilon=t_\epsilon(c_0,K_0)$ as defined in Proposition~\ref{prop:conv}.
Take $c\in (0,c_0]$ and $K\ge K_0$, and let
\begin{equation} \label{eq:Ldef}
L = \lceil K-K_0 \rceil +1+\left\lceil \frac{\log (c_0/c)}{\log 2} \right\rceil.
\end{equation}
Recall the definition of $c_1$ in Proposition~\ref{prop:NBBM1d}, and suppose $N\ge N'_\epsilon$ is sufficiently large that  Proposition~\ref{prop:NBBM1d} holds and that
\[
e^{2Lt_0}N^{-c_1}\le c_0,\qquad
e^{Lt_0}N^{-1-c_1}\le \epsilon,\qquad
e^{2t_\epsilon} N^{-c_1}<\epsilon\qquad
\text{ and }\qquad
e^{t_\epsilon}N^{-1-c_1}<\epsilon.
\]

Take $\mathcal X\in \Gamma(K,c)$, and let $v\uppar N$ denote the solution of~\eqref{pbv_probpaper} with initial condition $v_0(r)=N^{-1} \big|\mathcal X\cap \B(r)\big|$, which satisfies $v_0(r) \geq c \indic{r \geq K }$. By 
Proposition~\ref{prop:movemass}, recalling the definition of $L$ in~\eqref{eq:Ldef}, we have the lower bound
\begin{equation}\label{vshiftapp}
v\uppar N(r,Lt_0) \geq 2 c_0 \indic{r \geq K_0}, \quad \forall \;\; r \geq 0.
\end{equation}
(The time $Lt_0$ is the same as the time $t_a$ mentioned in the outline at the start of Section~\ref{sec:ssp}.)

Now we compare $F\uppar N(\cdot ,L t_0)$ with $v\uppar N(\cdot ,L t_0)$, to show that $F\uppar N(K_0,Lt_0) \geq c_0$ with high probability. By \eqref{vshiftapp} and by Proposition~\ref{prop:NBBM1d}, since $N$ is sufficiently large 
that $e^{2Lt_0}N^{-c_1}\le c_0$ and $e^{Lt_0}N^{-1-c_1}\le \epsilon$,
we now have that
\begin{align} \label{eq:FN(2)}
\psub{\mathcal X}{F\uppar N(K_0,Lt_0)\leq c_0 }&\leq 
\P_{\mathcal X}\bigg(\sup_{r\geq 0}\big|F\uppar N (r,Lt_0)-v\uppar N (r,Lt_0)\big|\geq c_0\bigg) \notag \\
&\leq \epsilon .
\end{align}
Take $t \geq t'_\epsilon 
+Lt_0$.
Then
\begin{align} \label{eq:FNK_0big}
&\psub{\mathcal X}{F\uppar N(K_0,t)< c_0 } \notag \\
&\quad \leq \psub{\mathcal X}{F\uppar N(K_0,Lt_0)\leq c_0 } +
 \psub{\mathcal X}{F\uppar N(K_0,Lt_0)\geq c_0,\; F\uppar N(K_0,t) <
c_0} \notag \\
&\quad \leq \epsilon +\E_{\mathcal X}\left[ \psub{X \uppar N(Lt_0)}{X\uppar N(t-Lt_0)\notin \Gamma (K_0,c_0)}\indic{X\uppar N (Lt_0)\in \Gamma (K_0,c_0)}\right] \notag \\
&\quad \leq 2\epsilon, 
\end{align}
where the second inequality
follows by~\eqref{eq:FN(2)} and the last inequality
 follows by~\eqref{eq:MC1} in
Proposition~\ref{prop:muNlargetime}.

Now note that for any configuration $\tilde{\mathcal X} \in \Gamma(K_0,c_0)$, letting $\tilde v$ denote the solution of~\eqref{pbv_probpaper} with initial condition $v_0(r)= N^{-1} \big|\tilde{\mathcal X}\cap\B(r)\big|$,  we have by Proposition~\ref{prop:conv} that $\sup_{r\geq 0}|\tilde v(r,t_\epsilon)-V(r)|<\epsilon$. Hence
\begin{align} \label{eq:FNVclose}
 \P_{\tilde{\mathcal X}}\bigg(\sup_{r\geq 0}\Big|F\uppar N(r,t_\epsilon)-V(r)\Big|\geq 2\epsilon\bigg) 
&\leq  \P_{\tilde{\mathcal X}}\bigg( \sup_{r\geq 0}\Big|F\uppar N(r,t_\epsilon)-\tilde v(r,t_\epsilon) \Big|
\geq \epsilon\bigg) 
\le \epsilon,
\end{align}
by Proposition~\ref{prop:NBBM1d}, since 
$N$ is large enough that 
$e^{2t_\epsilon}N^{-c_1}\le \epsilon$ and $e^{t_\epsilon}N^{-1-c_1}\le \epsilon$.

Hence for $t\geq t'_\epsilon+Lt_0+t_\epsilon$ and $\mathcal X \in \Gamma(K,c)$,
\begin{align} \label{eq:(dagger)ssp}
\P_{\mathcal X}\bigg(\sup_{r\geq 0}\Big|F\uppar N(r,t)-V(r)\Big|\geq 2\epsilon    \bigg) 
&\le\P_{\mathcal X}\bigg(\sup_{r\geq 0}|F\uppar N(r,t)-V(r)|\geq 2\epsilon \;,\;\;F\uppar N(K_0,t-t_\epsilon)\geq c_0  \bigg) \notag \\
&\qquad+ \P_{\mathcal X}\Big(F\uppar N(K_0,t-t_\epsilon)< c_0 \Big) \notag \\
&\le \epsilon+2\epsilon,
\end{align}
using \eqref{eq:FNK_0big} for the second term and  \eqref{eq:FNVclose} with the Markov property and $\tilde{\mathcal X}:=X\uppar N(t-t_\epsilon)$ for the first term. (Indeed,  $F\uppar N(K_0,t-t_\epsilon)\ge c_0$ is equivalent to $X\uppar N(t-t_\epsilon)\in\Gamma(K_0,c_0)$.) This concludes the proof of \eqref{eq:sspA}, the first statement of Proposition~\ref{prop:ssp}. We now turn to proving \eqref{eq:sspC}, the third statement.

Assume that $t\geq t'_\epsilon+Lt_0 +t_\epsilon+1 $ and $\mathcal X \in \Gamma(K,c)$, and
introduce $\lambda=N^{-c_1/3}$. For any family $(E_k)_{k=0}^\infty$ of events, we have using~\eqref{eq:FNK_0big} that
\begin{align} \label{eq:(*)ssp}
\P_{\mathcal X}\bigg( \sup_{s\in [0,1]} M\uppar N_{t+s}>R_\infty +2\epsilon\bigg) 
&\le 2\epsilon+\sum_{k=0}^{\lfloor1/\lambda\rfloor}
\P_{\mathcal X} \bigg(E_k\;,\;\sup_{s\in[\lambda,2\lambda]}M\uppar N_{t+(k-1)\lambda+s}>R_\infty+2\epsilon\bigg)\notag\\
&\qquad+\sum_{k=0}^{\lfloor 1/\lambda\rfloor}
\P_{\mathcal X}\Big(E_k^c\;,\;F\uppar N(K_0,t-t_\epsilon-\lambda)\ge c_0\Big) .
\end{align}
We use this expression with the events
$$E_k=  \big\{F\uppar N\big(R_\infty +\epsilon,t+(k-1)\lambda\big )\ge 1-N^{-c_1/2}\big\}.$$
For $\tilde{\mathcal X} \in \Gamma(K_0,c_0)$, let $\tilde v$ denote the solution of~\eqref{pbv_probpaper} with initial condition 
\[
v_0(r)=N^{-1}\big| \tilde{\mathcal X} \cap \B(r)\big|
\]
and let $\tilde R_t=\inf\{r\ge 0:\tilde v(r,t)=1\}$ for $t>0$. Then by Proposition~\ref{prop:conv}, for $t\ge t_\epsilon$, $|\tilde R_t-R_\infty|<\epsilon$ and so $\tilde v(R_\infty +\epsilon,t)=1$. 
Hence by Proposition~\ref{prop:NBBM1d}, for $k\in \{0,\ldots, \lfloor 1/\lambda\rfloor\}$,
\begin{align*}
\psub{\tilde{\mathcal X}}{F\uppar N (R_\infty +\epsilon,t_\epsilon+k \lambda)\leq 1-e^{2(t_\epsilon+1)}N^{-c_1}}
\leq e^{t_\epsilon +1}N^{-1-c_1}.
\end{align*}
For $N$ sufficiently large that $e^{2(t_\epsilon+1)}N^{-c_1}<N^{-c_1/2}$, by the Markov property at time $t-t_\epsilon-\lambda$ this implies that
$$\P_{\mathcal X}\Big(E_k^c\;,\;F\uppar N(K_0,t-t_\epsilon-\lambda)\ge c_0\Big)\le e^{t_\epsilon+1}N^{-1-c_1} .$$
By Lemma~\ref{lemAsup} with $b=c_1/3$, for $N$ sufficiently large, for $t' \ge 0$,
\begin{align*}
\P_{\mathcal X}\bigg(F\uppar N (R_\infty +\epsilon,t')\ge 1-\tfrac 14 \lambda,\,
\sup_{s\in [\lambda,2\lambda]}M\uppar N_{t'+s}>R_\infty +\epsilon+\lambda^{1/3}\bigg)
\le e^{-\lambda^{-1/4}}.
\end{align*}
For $N$ sufficiently large that
$\frac 14 \lambda \ge N^{-c_1/2}$ and
 $\lambda^{1/3}=N^{-c_1/9}<\epsilon$, choosing $t'=t+(k-1)\lambda$, this implies that
$$\P_{\mathcal X}\bigg(E_k\;,\; \sup_{s\in [\lambda,2\lambda]}M\uppar N_{t+(k-1)\lambda+s}>R_\infty +2\epsilon\bigg)\le e^{-\lambda^{-1/4}}=e^{-N^{c_1/12}}.
$$
By~\eqref{eq:(*)ssp},
it follows that  for $N$ sufficiently large,
\begin{align*} 
\P_{\mathcal X}\bigg( \sup_{s\in [0,1]} M\uppar N_{t+s}>R_\infty +2\epsilon \bigg)
&\leq 2\epsilon
 + ( N^{c_1/3}+1) (e^{-N^{c_1/12}}+e^{t_\epsilon +1}N^{-1-c_1})
 \le 3\epsilon,
\end{align*}
which concludes the proof of \eqref{eq:sspC} in Proposition~\ref{prop:ssp}. It remains to show \eqref{eq:sspB}.

Recall from~\eqref{eq:Vdef} that $V$ is strictly increasing and continuous on $[0,R_\infty]$, with $V(0)=0$ and $V(R_\infty)=1$.
If $\sup_{r\geq 0}|F\uppar N(r,t)-V(r)|< 2\epsilon $ then 
$M\uppar N_t> V^{-1}(1-2\epsilon)$, and so for $t\geq 
t'_\epsilon +Lt_0+t_\epsilon $ and $\mathcal X \in \Gamma(K,c)$, by~\eqref{eq:(dagger)ssp},
\begin{equation} \label{eq:MNlower}
\psub{\mathcal X}{M\uppar N_t \le V^{-1}(1-2\epsilon) }
\leq 3\epsilon.
\end{equation}
The statement~\eqref{eq:sspB} now follows from~\eqref{eq:MNlower} and~\eqref{eq:sspC}, which completes the proof.
\end{proof}
The following lemma is the main remaining step in the proof of Theorem~\ref{thm:sspd}.
\begin{lem} \label{lem:sspdproof}
For $\epsilon>0$, let
\[
\mathcal{C}_{\epsilon}
=\Big\{X^+_u(\epsilon^{-1/2}):u\in \set+_{\epsilon^{-1/2}}\;,\;
\big\|X^+_{u}(s)\big\|\leq R_\infty +\epsilon \,\,\forall s\in
[\epsilon,\epsilon^{-1/2}]\Big\}.
\]
Take $\delta>0$. Then for $\epsilon>0$ sufficiently small, for $N$ sufficiently large, the following holds: if the initial particle configuration $\mathcal X \in (\R^d)^N$ satisfies
$$\sup_{r\ge 0} \Big|N^{-1} |\mathcal X \cap \B(r)| - V(r)\Big| \le \epsilon, $$
then for any $A\subseteq \R^d$ measurable,
\[
\psub{\mathcal X}{\frac 1 N |\mathcal{C}_{\epsilon}\cap A|-\int_A U(x)\,\diffd x \geq \delta}<\tfrac 12 \delta.
\]
\end{lem}
\begin{proof}
First, we estimate the mean of $\frac 1 N |\mathcal C_{\epsilon} \cap A |$. By the many-to-one lemma,
\begin{align} \label{eq:sspdB}
\E_{\mathcal X}\left[ \frac 1 N |\mathcal C_{\epsilon} \cap A | \right]
&=\frac1N e^{\epsilon^{-1/2}}\sum_{k=1}^N \psub{\mathcal X_k}{B_{\epsilon^{-1/2}}\in A, \, \|B_s\| \leq R_\infty +\epsilon \, \,\, \forall s\in [\epsilon,\epsilon^{-1/2}]} \notag \\
&=\frac1N e^{\epsilon^{-1/2}}\sum_{k=1}^N \Esub{\mathcal
X_k}{\psub{B_\epsilon}{B_{\epsilon^{-1/2}-\epsilon}\in A, \, \|B_s\| \leq R_\infty +\epsilon \, \,\, \forall s\in [0,\epsilon^{-1/2}-\epsilon]}} \notag \\
&= e^{\epsilon^{-1/2}}\int_A w(y,\epsilon^{-1/2}-\epsilon)\,\diffd y,
\end{align}
where $w(y,s)$ is the unique solution to 
\begin{equation}
\begin{cases} \label{weqneps}
\partial _s w=\Delta w  & s>0,\ \|y\|<R_\infty +\epsilon,\\
w(y,s)=0  & s\ge0,\ \|y\|\geq R_\infty +\epsilon, \\
w(y,0) = (\Phi_\epsilon* \mu_{\mathcal X})(y) \quad & \|y\|<R_\infty +\epsilon,
\end{cases}
\end{equation}
where $\Phi_\epsilon(y)$ is the heat kernel and $\mu_{\mathcal X}(\diffd y)$ is the empirical measure on $\R^d$ determined by the points $\{\mathcal X_k\}_{k=1}^N$:
\[
\mu_{\mathcal X}(\diffd y) = \frac{1}{N} \sum_{k=1}^N \delta_{\mathcal X_k}, \qquad \Phi_\epsilon(y) = (4 \pi \epsilon)^{-d/2} e^{- \frac{\|y\|^2}{4 \epsilon}}.
\]

The function $w(x,s)$ can be expanded as a series in terms of the Dirichlet eigenfunctions of the Laplacian on $\mathcal{B}(R_\infty+\epsilon)$ (see~Theorem~7.1.3 of~\cite{Evans2010}):
\begin{equation} \label{eq:sppdC}
w(x,s)=\sum_{k=1}^\infty a_k e^{-s\lambda^\epsilon_k}U^\epsilon_k (x),
\quad s\geq 0,\;\|x\|\leq R_\infty +\epsilon,
\end{equation}
where the partial sums converge weakly in $L^2([0,T];H^1_0)$.
Here $\big\{\big(U_k^\epsilon(x),\lambda^\epsilon_k\big) \big\}_{k \geq 1}$ denote the Dirichlet eigenfunctions and eigenvalues for $-\Delta$ on the ball $\big\{\|x\| \leq R_\infty + \epsilon \big\}$:
\begin{alignat}{2}
 - \Delta U^{\epsilon}_k &= \lambda_k^\epsilon U_k^{\epsilon}, \quad && \text{for }\|x\| < R_\infty + \epsilon, \nonumber \\
U_k^{\epsilon}(x)& = 0, \quad &&\text{for }\|x\| \ge R_\infty + \epsilon. \nonumber 
\end{alignat}
By scaling we have
\begin{equation*}
\lambda^\epsilon_k  = \bigg(\frac{R_\infty }{R_\infty + \epsilon}\bigg)^2 \lambda^0_k\qquad\text{and}\qquad U^{\epsilon}_k(x) = \left(\frac{R_\infty}{R_\infty + \epsilon}\right)^{d/2} U^0_k\bigg(x\frac{R_{\infty}}{R_\infty + \epsilon}\bigg).
\end{equation*}
The eigenvalues satisfy  $\lambda^0_1=1<\lambda^0_2\leq \ldots $ and the eigenfunctions
$\{U_k^\epsilon\}_{k=1}^\infty$ form an orthonormal basis in $L^2(\mathcal{B}(R_\infty + \epsilon))$. Furthermore, $U^0_1(x) =\|U\|^{-1}_{L^2} U(x)$, see \eqref{eqU_probpaper}.

Define $\tilde w(y,s) = \sum_{k= 2}^\infty a_k e^{-s\lambda^\epsilon_k}U^\epsilon_k (y)$. Thus,
\begin{equation} \label{eq:wtildew}
w(y,s) = a_1 e^{-s\lambda^\epsilon_1 }U^\epsilon_1 (y) + \tilde w(y,s),
\end{equation}
and $\tilde w(\cdot,s)$ is orthogonal to $U^\epsilon_1$ in $L^2$ for all $s \geq 0$. Observe that
\begin{align*}
\big\|w(\cdot ,0)\big\|^2_{L^2} & \le  \int_{\R^d} \left ( \int_{\R^d} \Phi_\epsilon(y-x)\,  \mu_{\mathcal X}(\diffd x)  \right)^2 \,\diffd y  \\
& \leq \int_{\R^d} \int_{\R^d} \Phi_\epsilon(y-x)^2\,  \mu_{\mathcal X}(\diffd x)   \,\diffd y \\
& = \int_{\R^d} \left( \int_{\R^d} \Phi_\epsilon(y-x)^2 \,\diffd y \right)\,\mu_{\mathcal X}(\diffd x) = (8\pi \epsilon)^{-d/2},
\end{align*}
where the second line follows by Jensen's inequality.
Therefore, for $s\ge 0$,
\begin{equation} \label{eq:sspdE}
\big\| \tilde w(\cdot,s) \big\|^2_{L^2} \leq e^{-2\lambda_2^\epsilon s} \big\| \tilde w(\cdot,0) \big\|^2_{L^2} \leq e^{-2\lambda_2^\epsilon s} \big\| w(\cdot,0) \big\|^2_{L^2} \leq e^{-2\lambda_2^\epsilon s} (8\pi \epsilon)^{-d/2}.
\end{equation}
Now we estimate $a_1$. Since $\|U_1^\epsilon\|_{L^2} = \|U_1^0\|_{L^2} = 1$ and since $U_1^\epsilon$ is spherically symmetric, writing $\mathbf{e}$ for a unit vector, we have 
\begin{align}
a_1 = \int_{\B(R_\infty +\epsilon)}
U^\epsilon_1 (x)w(x,0)\,\diffd x & =\int_0^{R_\infty +\epsilon}U^\epsilon_1(r\mathbf{e})\left( \int_{\partial\B(r)} (\Phi_\epsilon*\mu_{\mathcal X})(y)\, \diffd S(y) \right) \diffd r \notag \\
& =-\int_0^{R_\infty +\epsilon}\partial _rU^\epsilon_1(r\mathbf{e})\left( \int_{\B(r)} (\Phi_\epsilon*\mu_{\mathcal X})(y)\, \diffd y \right) \diffd r   \label{a1form}
\end{align}
by integration by parts.  Also, for $r\ge 0$,
\begin{align} \label{eq:dU}
-\partial_r U^\epsilon_1(r\mathbf{e})
=-\left(\frac{R_\infty}{R_\infty +\epsilon}\right)^{d/2+1}\big\|U\big\|_{L^2}^{-1}\,\partial_r U\left(\frac{rR_\infty \mathbf{e}}{R_\infty +\epsilon} \right).
\end{align}
In particular, since $U(r\mathbf{e})$ is non-increasing in $r$ for $r\ge 0$, $-\partial_r U^\epsilon_1(r\mathbf{e})\geq 0$.

Now recall that we assume $\big|N^{-1}|\mathcal X \cap \B(r)|-V(r)\big|<\epsilon$ for all $r\geq 0$, and so
\begin{equation}
\sup_{r \geq 0} \big|\mu_{\mathcal X}(\mathcal{B}(r))-V(r)\big|<\epsilon. \label{goody1yN}
\end{equation}
Then for $r\ge 0$,
$$\int_{\mathcal{B}(r)} \int_{\|x \| < r +  \epsilon^{1/3}} \Phi_\epsilon(y - x)\, \mu_{\mathcal X}(\diffd x) \,\diffd y \le \mu_{\mathcal X}\big( \mathcal{B}\big(r + \epsilon^{1/3}\big)\big )
\le  V\big(r+\epsilon^{1/3}\big)+\epsilon,$$
where the first inequality follows since $\int_{\R^d}\Phi_\epsilon (y-x)\,\diffd y=1$.
Furthermore, writing $y=x+z$,
\begin{align*}\int_{\mathcal{B}(r)}\int_{\|x\| \ge r + \epsilon^{1/3}} \Phi_\epsilon(y - x)\,\mu_{\mathcal X}(\diffd x) \,\diffd y
&= \int_{\|x\| \ge r+\epsilon^{1/3}} \int_{\R^d}\indic{\|x+z\|<r} \Phi_\epsilon(z) \,\diffd z\, \mu_{\mathcal X}(\diffd x) 
\\&\le\int_{\|x\| \ge r+\epsilon^{1/3}}\int_{\|z\| > \epsilon^{1/3}} \Phi_\epsilon(z)\,\diffd z\,\mu_{\mathcal X}(\diffd x)
\\&\le \int_{\|z\| > \epsilon^{1/3}} \Phi_\epsilon(z)\,  \diffd z\\
&\le 2d e^{-\epsilon^{-1/3}/(4d)},
\end{align*}
where we used $\mu_{\mathcal X}(\R^d)=1$ in the third line and  a Gaussian tail bound in the last line. This implies that
$$\int_{\mathcal{B}(r)} (\Phi_\epsilon* \mu_{\mathcal X})(y)\, \diffd y \le V(r+\epsilon^{1/3})+\epsilon+2d e^{-\epsilon^{-1/3}/(4d)}.$$
Therefore, by \eqref{a1form} and~\eqref{eq:dU},
\begin{align*}
a_1  &\leq -\int_0^{R_\infty +\epsilon}\left(\frac{R_\infty}{R_\infty+\epsilon}\right)^{d/2+1}\|U\|_{L^2}^{-1}\partial_r U\left( \frac{rR_\infty \mathbf{e}}{R_\infty +\epsilon}\right)\left(V(r+\epsilon^{1/3})+\epsilon+2de^{-\epsilon^{-1/3}/(4d)}\right)\diffd r\\
&\leq \|U\|_{L^2}^{-1}\int_0^{R_\infty +\epsilon}\left(-\partial_r U(r\mathbf{e})+\epsilon  \|\partial_{r}^2 U\|_\infty\right)\left(V(r)+\epsilon^{1/3}\|V'\|_\infty+\epsilon+2de^{-\epsilon^{-1/3}/(4d)}\right)\diffd r\\
&\leq -\|U\|^{-1}_{L^2}\int_0^{R_\infty} \partial _r U(r\mathbf{e})V(r)\,\diffd r +\mathcal O(\epsilon^{1/3}).
\end{align*}
Note that by integration by parts,
$$
-\int_0^{R_\infty}\partial _r U(r\mathbf{e})V(r)\,\diffd r
=\int_0^{R_\infty} U(r\mathbf{e}) V'(r)\,\diffd r
=\|U\|^2_{L^2},
$$
since $V(r)=\int_{\B(r)} U(y)\,\diffd y$.
Hence for $\delta>0$, for $\epsilon$ sufficiently small, 
we have that
\begin{equation} \label{eq:sspdD}
a_1 \leq \|U\|_{L^2} \big(1+\tfrac 14 \delta\big) .
\end{equation}

By~\eqref{eq:wtildew},
we have now shown that for $\epsilon$ sufficiently small, for $s\ge 0$ and $\|x\|\le R_\infty +\epsilon$,
\begin{equation}
w(x,s)\leq \big(1+\tfrac 14 \delta\big)e^{-s\big(\frac{R_\infty}{R_\infty+\epsilon}\big)^2}\left( \frac{R_\infty}{R_\infty+\epsilon}\right)^{d/2} U\left( \frac{xR_\infty}{R_\infty+\epsilon}\right)+\tilde w(x,s), \label{eq:sspdE2}
\end{equation}
where $\tilde w$ satisfies \eqref{eq:sspdE}. Then
\begin{align}\int_A w(x,s)\,\diffd x &\le \big(1+\tfrac14\delta+\mathcal O(\epsilon)\big)e^{-s\big(\frac{R_\infty}{R_\infty+\epsilon}\big)^2}\int_A U(x)\,\diffd x
+ \int_{\mathcal B(R_\infty +\epsilon)} \Big|\tilde w (x,s)\Big|\,\diffd x.
\label{intwA}
\end{align}
By Jensen's inequality and then by~\eqref{eq:sspdE},
$$\int_{\mathcal B(R_\infty +\epsilon)} \big|\tilde w (x,s)\big|\,\diffd x
\le \big|\B(R_\infty+\epsilon)\big|^{1/2} \|\tilde w (\cdot,s)\|_{L^2}
\le \big|\B(R_\infty+\epsilon)\big|^{1/2} e^{-\lambda_2^\epsilon s}(8\pi\epsilon)^{-d/4}.$$
Take $c\in (0,\lambda^0_2-1)$, and
suppose $\epsilon$ is sufficiently small that $\lambda_2^\epsilon=(\frac{R_\infty}{R_\infty +\epsilon})^2 \lambda^0_2>1+c$.
Then
$$\int_{\mathcal B(R_\infty +\epsilon)} \big|\tilde w (x,s)\big|\,\diffd x
\le \big|\B(R_\infty+\epsilon)\big|^{1/2} e^{-(1+c) s}(8\pi\epsilon)^{-d/4}.
$$
By~\eqref{intwA} and \eqref{eq:sspdB}, it follows that
\begin{align}
\E_{\mathcal X}\left[\frac 1 N \big|\mathcal C_{\epsilon} \cap A\big | \right] 
&\le \big(1+\tfrac 14 \delta+\mathcal O(\epsilon)\big) e^{\epsilon^{-1/2}\big[1-\big(\frac{R_\infty}{R_\infty+\epsilon}\big)^2\big]}\int_AU(x)\,\diffd x
\notag\\&\qquad+\big|\B(R_\infty+\epsilon)\big|^{1/2} e^{-c\epsilon^{-1/2}+(1+c)\epsilon}(8\pi\epsilon)^{-d/4}
\notag\\&\le \big(1+\tfrac 14 \delta+\mathcal O(\epsilon^{1/2})\big)\int_AU(x)\,\diffd x
 +\mathcal O(\epsilon)
.
\end{align}
Therefore,
for $\delta>0$, for $\epsilon>0$ sufficiently small, 
if $\mathcal X \in (\R^d)^N$ satisfies~\eqref{goody1yN} then
for $A\subseteq \R^d$ measurable,
\begin{align*}
\E_{\mathcal X}\left[ \frac 1 N \big|\mathcal C_{\epsilon}\cap A\big| \right]
\leq \int_A U(x)\,\diffd x +\tfrac 12 \delta.
\end{align*}
By the same argument as for
Lemma~\ref{lem:dhydro2}, for $N$ sufficiently large, for $A\subseteq \R^d$ measurable and $\epsilon>0$,
\begin{align} \label{eq:sspdF}
\E_{\mathcal X}\left[\left(\frac 1 N \big|\mathcal C_{\epsilon} \cap A\big|-\E_{\mathcal X}\left[\frac 1 N \big|\mathcal C_{\epsilon} \cap A\big| \right]\right)^4 \right]
\leq 13 e^{4\epsilon^{-1/2}}N^{-2}.
\end{align}
So for $\delta>0$, for $\epsilon>0$ sufficiently small and $N$ sufficiently large,
if $\mathcal X \in (\R^d)^N$ satisfies~\eqref{goody1yN} then
for
$A\subseteq \R^d$ measurable,
\begin{align*}
\psub{\mathcal X}{\frac 1 N \big|\mathcal C_{\epsilon} \cap A\big|-\int_A U(x)\,\diffd x \ge \delta }
&\leq \psub{\mathcal X}{ \left|\frac 1 N \big|\mathcal C_{\epsilon} \cap A\big|-\E_{\mathcal X} \left[\frac 1 N |\mathcal C_{\epsilon} \cap A|\right]\right| \ge \tfrac 12 \delta }\\
&\leq 16 \delta^{-4} \cdot 13 e^{4\epsilon^{-1/2}}N^{-2},
\end{align*}
by~\eqref{eq:sspdF} and Markov's inequality.
The result follows by taking $N$ sufficiently large.
\end{proof}
The following result is an immediate consequence of Lemma~\ref{lem:sspdproof} and the coupling between the BBM $X^+(t)$ and the $N$-BBM $X\uppar N(t)$ described in Section~\ref{sec:notations}.
\begin{corr} \label{cor:sspdproof}
Take $\delta>0$. Then for $\epsilon>0$ sufficiently small, for $N$ sufficiently large, the following holds: if the initial particle configuration $\mathcal X \in (\R^d)^N$ satisfies
$$\sup_{r\ge 0} \big|N^{-1} |\mathcal X \cap \B(r)| - V(r)\big| \le \epsilon, $$
then for any  $A\subseteq \R^d$ measurable,
\[
\psub{\mathcal X}{ M\uppar N_s \le R_\infty +\epsilon \  \forall s\in
[0,\epsilon^{-1/2}]\;,\; \mu\uppar N\big(A,\epsilon^{-1/2}\big)
-\int_A U(x)\,\diffd x \geq \delta}<\tfrac 12 \delta.
\]
\end{corr}
\begin{proof}
Under the coupling between the BBM $X^+$ and the $N$-BBM $X\uppar N$ described in Section~\ref{sec:notations}, we have that for $\epsilon>0$, if $M\uppar N_s \le R_\infty +\epsilon$ $\forall s\in [0,\epsilon^{-1/2}]$ then
by~\eqref{eq:XNcontains}, almost surely
\begin{align*}
\Big\{X\uppar N _k (\epsilon^{-1/2}):k\in \{1,\ldots, N\}\Big\}
&\subseteq \Big\{X^+_u(\epsilon^{-1/2}):u\in \set+_{\epsilon^{-1/2}},\;
\big\|X^+_{u}(s)\big\|\leq R_\infty +\epsilon \,\,\forall s\in
[0,\epsilon^{-1/2}]\Big\}\\
&\subseteq \mathcal{C}_{\epsilon },
\end{align*}
where $\mathcal C_\epsilon$ is defined in the statement of Lemma~\ref{lem:sspdproof}.
Hence for any $A\subseteq \R^d$ measurable, $\mu\uppar N(A,\epsilon^{-1/2})\le \frac 1N |\mathcal C_\epsilon \cap A|$.
The result follows by Lemma~\ref{lem:sspdproof}.
\end{proof}
We can now use Proposition~\ref{prop:ssp} and Corollary~\ref{cor:sspdproof} to prove Theorem~\ref{thm:sspd}.
\begin{proof}[Proof of Theorem~\ref{thm:sspd}]
Take $K>0$, $c\in(0,1]$ and $\delta>0$. As the second statement of the theorem was already proved in
Proposition~\ref{prop:ssp}, it remains to prove that for $N$
and $t$ large enough, for $\mathcal X\in\Gamma(K,c)$ and $A\subseteq \R^d$
measurable,
$$
\P_{\mathcal X}\bigg(\Big|\mu\uppar N (A,t)-\int_A U(x)\,\diffd x\Big|\geq
\delta
\bigg)
<\delta .
$$
Take $\epsilon>0$ sufficiently small that Corollary~\ref{cor:sspdproof} holds and $\lceil \epsilon^{-1/2} \rceil \epsilon +\epsilon<\frac 12 \delta$,
and take $N_\epsilon$, $T_\epsilon$ as defined in Proposition~\ref{prop:ssp}.
Take
 $N\geq N_\epsilon$ 
sufficiently large that Corollary~\ref{cor:sspdproof} holds, and take
$t\geq T_\epsilon+\epsilon^{-1/2}$. Let $t_0=t-\epsilon^{-1/2}$. 

For an initial particle configuration $\mathcal X \in \Gamma(K,c)$ and
$A\subseteq \R^d$ measurable, we have by a union bound that
\begin{align*}
&\psub{\mathcal X}{\mu\uppar N (A ,t)-\int_A U(x)\,\diffd x \geq \delta } \notag \\
&\ \ \le 
\P_{\mathcal X}\bigg(M\uppar N_s \le R_\infty +\epsilon \, \forall s\in [t_0,t]\,, \,
\sup_{r\geq 0}\big|F\uppar N(r,t_0)-V(r)\big|\leq \epsilon\,,\,
  \mu\uppar N (A ,t)-\int_A U(x)\,\diffd x \geq \delta \bigg) \notag \\
&\qquad\quad+ \P_{\mathcal X}\bigg(\sup_{s\in [t_0,t]}M\uppar N _s >R_\infty +\epsilon\bigg)
+\P_{\mathcal X}\bigg(\sup_{r\geq 0}\big|F\uppar N(r,t_0)-V(r)\big|\geq \epsilon \bigg) \notag \\
&\ \  \leq \tfrac 12 \delta+ \lceil t-t_0 \rceil \epsilon +\epsilon,
\end{align*}
by~\eqref{eq:sspC} and~\eqref{eq:sspA} in Proposition~\ref{prop:ssp} for the last two terms (since $N\geq N_\epsilon$ and $t_0 = t-\epsilon^{-1/2} \geq T_\epsilon$), and by Corollary~\ref{cor:sspdproof} and the Markov property at time $t_0$ for the first term.  
Therefore, since $\lceil t-t_0 \rceil \epsilon +\epsilon = \lceil \epsilon^{-1/2} \rceil \epsilon +\epsilon<\frac 12 \delta$, it follows that for any $A\subseteq \R^d$ measurable,
\begin{align*}
\psub{\mathcal X}{\mu\uppar N (A,t)-\int_A U(x)\,\diffd x \geq \delta }
&\leq  \delta.
\end{align*}
As in~\eqref{eq:FAcomplement},
since $\mu\uppar N (A,t)+\mu\uppar N \big(\R^d\setminus A,t\big)=1$ and $\int _A U(x)\,\diffd x+ \int _{\R^d\setminus A}U(x)\,\diffd x=1$, we have that
\begin{align*}
&\psub{\mathcal X}{\mu\uppar N (A,t)- \int_A U(x)\,\diffd x \leq -\delta }
= \psub{\mathcal X}{\mu\uppar N \big({\R^d \setminus A},t\big)- \int_{\R^d \setminus A} U(x)\,\diffd x \geq \delta }\le\delta,
\end{align*}
which completes the proof.
\end{proof}

It remains to prove Theorems~\ref{thm:piNexists} and~\ref{cor:ssp}, which will follow easily from the following proposition.
\begin{prop} \label{prop:inv}
Take $t_0>1$ as in Proposition~\ref{prop:muNlargetime}.
For $N$ sufficiently large, for any $t_1\in (0,t_0]$, the Markov chain
$(X\uppar N (t_1 n) )_{n=0}^\infty$ is a positive recurrent strongly
aperiodic Harris chain.
\end{prop}
\noindent\textit{Remark:} This proposition will be used in the
proof of Theorem~\ref{thm:piNexists} in combination with Theorems~6.1
and~4.1 of~\cite{AN78}, which
say that a positive recurrent strongly aperiodic Harris chain admits
a unique invariant probability measure, and that the distribution of the
state of the Harris chain after $n$ steps converges to that invariant
probability measure as $n\to\infty$.
\begin{proof}
For $n\in \N_0$, let $Y_n=X\uppar N (t_1  n)$. 
We use a similar strategy to the proof of Proposition~3.1 in \cite{DR11}.
By~\cite{AN78}, to show that ($Y_n)_{n=0}^\infty$ is a recurrent
strongly aperiodic Harris chain, it suffices to show that there exists
a set $\Lambda\subseteq (\R^d)^N$ such that
\begin{enumerate}
\item $\psub{\mathcal X}{\tau_\Lambda<\infty}=1$ $\forall \mathcal X\in (\R^d)^N$, where 
$\tau_\Lambda =\inf\{n\geq 1:Y_n \in \Lambda\}$.
\item There exist $\epsilon>0$ and a probability measure $q$ on $\Lambda$ such that
$\psub{\mathcal X}{Y_1 \in C}\geq \epsilon q(C)$ for any $\mathcal X\in
\Lambda$ and $C\subseteq \Lambda$.
\end{enumerate}
To furthermore prove that the Harris chain is positive recurrent, we also need
to show that
\begin{enumerate}\setcounter{enumi}{2}
\item $\sup_{\mathcal X\in \Lambda}\Esub{\mathcal X}{\tau_\Lambda}<\infty$.
\end{enumerate}
We prove the proposition with the set
$$\Lambda=\big(\mathcal B(R_\infty +1)\big)^{ \otimes N}.$$
We start with the third point: showing that $\sup_{\mathcal X\in
\Lambda}\Esub{\mathcal X}{\tau_\Lambda}<\infty$.

Take $K_0=3$ and $c_0>0$ as in Proposition~\ref{prop:muNlargetime};
let
$\Lambda'=\Gamma(K_0\vee (R_\infty +1),c_0)\supset \Lambda$.
By~\eqref{eq:sspB} in Proposition~\ref{prop:ssp}, there exist $n_1,N_1<\infty$ such that for $N\geq N_1$, for
$\mathcal X \in \Lambda'$,
$$
\psub{\mathcal X}{M\uppar N_{n_1t_1}\ge R_\infty+1  }\leq \frac 12.
$$
Hence
\begin{equation} \label{eq:inv}
\sup_{\mathcal X\in \Lambda'} \psub{\mathcal X}{Y_{n_1}\notin \Lambda}\leq \frac 12.
\end{equation}
Let $\tau_{\Lambda'}=\inf\{n\geq 1:Y_n\in \Lambda'\}.$
Note that letting $t_2=\lceil t_0/t_1\rceil t_1$, we have that
 \[
 \tau_{\Lambda'}\leq \left\lceil \frac  {t_0}{t_1}\right\rceil \inf\big\{n\geq 1:X\uppar N(nt_2)\in \Gamma (K_0 \vee (R_\infty +1),c_0)\big\}.
 \]
Hence, since $t_2\in [t_0,2t_0]$, if $N$ is sufficiently large then by~\eqref{eq:MC2} and~\eqref{eq:MC3} in Proposition~\ref{prop:muNlargetime},
\begin{equation} \label{eq:corsspA}
\psub{\mathcal X}{\tau_{\Lambda'}<\infty }=1 \quad \forall \mathcal X\in
(\R^d)^N\qquad \text{and} \qquad \sup_{\mathcal X\in
\Lambda'}\Esub{\mathcal X}{\tau_{\Lambda'}}<\infty. 
\end{equation}

Let $\tau(0)=0$ and for $k\in \N$, let
$$
\tau(k)=\inf\{n \geq 1+\tau(k-1):Y_n \in \Lambda'\}.
$$
Then by~\eqref{eq:corsspA}, $(\tau(k))_{k=1}^\infty $  form  an increasing sequence of
almost surely finite times such that $Y_{\tau(k)}\in \Lambda'$. Notice that
$\tau(1)=\tau_{\Lambda'}$ and that
$$\tau_\Lambda\le \tau(k^*)+n_1\qquad\text{where}\quad k^*=\inf\big\{k\ge1:
Y_{\tau(k)+n_1}\in \Lambda\big\}.$$
It is therefore sufficient to show that $\sup_{\mathcal X \in \Lambda}\E_{\mathcal X}[\tau(k^*)]<\infty$
to establish that $\sup_{\mathcal X \in \Lambda} \E_{\mathcal
X}[\tau_\Lambda]<\infty$.

Write $(\mathcal F_n)_{n= 0}^\infty$ for the natural filtration of the Markov chain $(Y_n)_{n=0}^\infty$.
Notice that for $k\ge n_1$, the event $\{k^*> k-n_1\}=\{Y_{\tau(j)+n_1}\not\in \Lambda\ \forall 1\le
j\le k-n_1\}$ is measurable in $\mathcal F_{\tau(k)}$ since $\tau(k-n_1)+n_1
\le \tau(k)$.
Therefore by the strong Markov property, for $k\ge n_1$ and $\mathcal X\in
\Lambda'$, 
\begin{align}
\psub{\mathcal X}{k^*>k}=\psub{\mathcal X}{Y_{\tau(j)+n_1}\notin \Lambda \  \forall 1\le j\le k}
&\le \psub{\mathcal X}{k^*>k-n_1\,,\,
                       Y_{\tau(k)+n_1}\notin \Lambda}
\notag\\
&=\Esub{\mathcal X}{\indic{k^*>k-n_1} \ \psub{Y_{\tau(k)}}{Y_{n_1}\notin
\Lambda}}
\notag \\
&\le \frac12 \psub{\mathcal X}{k^*>k-n_1},
\notag
\end{align}
where we used~\eqref{eq:inv} in the last step.
Then, by an induction argument, for $k\in\N$ and  $\mathcal X\in \Lambda'$,
\begin{equation}
\psub{\mathcal X}{k^*>k}\le 2^{-\lfloor k/n_1\rfloor}. \label{eq:YnotA}
\end{equation}
In particular, $k^*$ is almost surely finite.
For $\mathcal X\in \Lambda'$, 
\begin{align}
\Esub{\mathcal X}{\tau(k^*)}
&=\sum_{k=1}^\infty \Esub{\mathcal X}{ \tau(k)\indic{k^*=k}}
=\sum_{k=1}^\infty \sum_{\ell=1}^k\Esub{\mathcal X}{
\big(\tau(\ell)-\tau(\ell-1)\big)\indic{k^*=k}}\notag\\
&=\sum_{\ell=1}^\infty \Esub{\mathcal X}{
\big(\tau(\ell)-\tau(\ell-1)\big)\indic{k^*\ge \ell}}.
 \label{eq:EtauAboundbis}\end{align}
Then, for $\mathcal X\in \Lambda'$ and $\ell\geq 1$, by~\eqref{eq:YnotA} and the strong Markov property,
\begin{align}
    \Esub{\mathcal X}{ \big(\tau(\ell)-\tau(\ell-1)\big)\indic{k^*\ge \ell}}
&\le \Esub{\mathcal X}{ \big(\tau(\ell)-\tau(\ell-1)\big)\indic{k^*>
\ell-1-n_1}}
\notag\\
& = \Esub{\mathcal X}{ \E_{\mathcal X}\big[\tau(\ell)-\tau(\ell-1) \;\big|\;
\mathcal F_{\tau(\ell-1)}\big]\indic{k^*> \ell-1-n_1}}
\notag\\
&\le \sup_{\mathcal Y\in \Lambda'}\E_{\mathcal Y}[\tau(1)]\ \P_{\mathcal X}(k^*>
\ell -1-n_1)
\notag\\
&\le 2^{-\lfloor (\ell-1-n_1)/n_1\rfloor}\sup_{\mathcal Y\in \Lambda'}\E_{\mathcal Y}[\tau(1)] . \notag
\end{align}
By substituting into \eqref{eq:EtauAboundbis}, we get that $\sup_{\mathcal X \in \Lambda'}\Esub{\mathcal
X}{\tau(k^*)}<\infty$, since we have that 
$\tau(1)=\tau_{\Lambda'}$ and, from~\eqref{eq:corsspA}, that
$\sup_{\mathcal Y\in \Lambda'}\E_{\mathcal    Y}[\tau_{\Lambda'}]<\infty$.
This implies that $\sup_{\mathcal X \in \Lambda'}\E_{\mathcal X}[\tau_\Lambda]<\infty$ and, in particular, $\sup_{\mathcal X \in \Lambda}\E_{\mathcal X}[\tau_\Lambda]<\infty$.

Proving the first of the three points at the beginning of the proof is now straightforward:
for $\mathcal X\in (\R^d)^N$, we have $\tau_{\Lambda'}<\infty$ a.s.\@ from \eqref{eq:corsspA} and, by the strong Markov property at time $\tau_{\Lambda'}$ and since $\psub{\mathcal X '}{\tau_\Lambda<\infty}=1$ for $\mathcal X' \in \Lambda'$,
$$
\psub{\mathcal X}{\tau_\Lambda<\infty}= 1.
$$

Finally, it remains to prove the second of the three points at the beginning of the proof. Note that conditional on the event that none of the particles in the $N$-BBM branch in the time interval $[0,t_1]$, the $N$ particles move according to independent Brownian motions.
Therefore, for $\mathcal X=(\mathcal X_1,\ldots, \mathcal X_N)\in \Lambda$ and $C\subseteq \Lambda$,
\begin{align*}
\psub{\mathcal X}{Y_1\in C}
&\geq e^{-t_1 N} \int_C \prod_{i=1}^N \left( \frac 1 {(4\pi t_1)^{d/2}} e^{-\frac 1 {4t_1}\|\mathcal X_i-y_i\|^2}\right)\diffd y_1\ldots \diffd y_N\\
&\geq e^{-t_1 N} (4\pi t_1)^{-Nd/2} e^{-N(R_\infty+1)^2/t_1}\text{Leb}(C),
\end{align*}
where $\text{Leb}(\cdot)$ is the Lebesgue measure on $(\R^d)^N$. Indeed,
for $y\in \Lambda$, $\|\mathcal X_i-y_i\|^2 \leq 4(R_\infty +1)^2$ $\forall i\in \{1,\ldots, N\}$.
Let
$$\epsilon 
=e^{-t_1 N} (4\pi t_1)^{-Nd/2} e^{-N(R_\infty+1)^2/t_1}\text{Leb}(\Lambda),
$$
and define a probability measure $q$ on $\Lambda$ by letting
$
q(C)=\text{Leb}(C)/\text{Leb}(\Lambda)
$
for $C\subseteq \Lambda$.
Then for $\mathcal X\in \Lambda$ and $C\subseteq \Lambda$,
$\psub{\mathcal X}{Y_1\in C}\geq \epsilon q(C)$.
The result follows.
\end{proof}
\begin{proof}[Proof of Theorem~\ref{thm:piNexists}]
By Proposition~\ref{prop:inv}, and by Theorems~6.1
and~4.1 in~\cite{AN78}, for $N$ sufficiently large, for $t_1\in(0,t_0]$, $(X\uppar N(nt_1))_{n=0}^\infty$
has a unique invariant measure $\pi_{t_1} \uppar N$ which is a probability
measure on $(\R^d)^N$, and for any $\mathcal X \in (\R^d)^N$, the law of $X\uppar N(nt_1)$ under
$\P_{\mathcal X}$ converges as $n\to\infty$ to $\pi_{t_1} \uppar N$ in total variation norm. In particular,
if $C\subseteq (\R^d)^N$ is measurable,
\begin{equation} \label{eq:corsspB}
\psub{\mathcal X}{X\uppar N (t_1 n)\in C } \to \pi_{t_1} \uppar N(C)
\quad \text{as }n\to \infty.
\end{equation}
Fix $N$ large enough for Proposition~\ref{prop:inv} to hold.
We begin by showing that
\begin{equation}\label{allmeaseq}
\pi_{t_1} \uppar N=\pi_{t_0} \uppar N=:\pi\uppar N\quad \forall t_1\in(0,t_0].
\end{equation}
Take $\mathcal X\in (\R^d)^N$,  and $C\subseteq (\R^d)^N$ a closed set.
Take $\delta>0$.
For $\epsilon>0$, let
\[
C^\epsilon = \big\{\mathcal Y\in (\R^d)^N : \inf_{\mathcal Z\in
C}\|\mathcal Y-\mathcal Z\|<\epsilon\big\}.
\]
Here $\|\mathcal Y -\mathcal Z \|$ denotes the Euclidean norm of $\mathcal Y -\mathcal Z $ regarded as a vector in $\R^{dN}$. 
Then $\pi\uppar N_{t_0}(C^\epsilon)\to \pi\uppar N_{t_0}(C)$ as $\epsilon\to 0$.
Take $\epsilon>0$ sufficiently small that 
\[ \pi\uppar N _{t_0}(C^\epsilon)<\pi\uppar N_{t_0} (C)+\tfrac 13 \delta. \]
It is easy to see that if $t_1 > 0$ is small enough, then
\begin{equation}
\P_{\tilde{\mathcal X}}\big(X\uppar N(s)\not\in C^\epsilon\big)<\tfrac 1 3 \delta 
\qquad\forall\ \tilde{\mathcal X}\in C,\ s\in[0,t_1].
\label{eq:dontmovemuch}
\end{equation}
Indeed, the event that no particle branches on the time interval $[0,t_1]$ has probability $e^{-N t_1}$, which can be arbitrarily close to $1$ if $t_1$ is small enough. Conditioned on this event, the random process $Y(s)= X \uppar N (s) -X \uppar N (0)$ is a Brownian motion in $\R^{dN}$.  In particular, $Y(s)$ is almost surely continuous on $[0,t_1]$, with $Y_0 = 0$, and the law of $Y(s)$ does not depend on $X \uppar N (0)$ or $t_1$. Then $\P_{\tilde{\mathcal X}}\big(X\uppar N(s)\in C^\epsilon\big)\ge e^{-Nt_1}\P(\|Y(s)\|<\epsilon\ \forall s\in[0,t_1])$, which can be made arbitrarily close to $1$ by taking $t_1$ sufficiently small.

When choosing $t_1$ small enough for \eqref{eq:dontmovemuch}, we furthermore require that $t_0/t_1\in\N$. It is then clear from \eqref{eq:corsspB} than $\pi\uppar N_{t_1}=\pi\uppar N_{t_0}$.
Take $n_0\in \N$ sufficiently large that for $n\ge n_0$,
\[
\psub{\mathcal X}{X\uppar N (t_1 n)\in C^\epsilon }\le \pi_{t_0}\uppar N (C^\epsilon)+\tfrac 13 \delta 
\le \pi_{t_0}\uppar N (C)+\tfrac 23 \delta .
\]
For $t\ge t_1 n_0$ we have
\begin{align*}
\psub{\mathcal X}{X\uppar N(t)\in C}
&\le \psub{\mathcal X}{X\uppar N\big(\lceil t/t_1\rceil t_1 \big)\in C^\epsilon }
+\psub{\mathcal X}{X\uppar N(t)\in C, \, X\uppar N\big(\lceil t/t_1\rceil
t_1 \big)\notin C^\epsilon }\\
&\le \pi_{t_0}\uppar N(C)+\tfrac 23 \delta +
\psub{\mathcal X}{X\uppar N(t)\in C, \, X\uppar N\big(\lceil t/t_1\rceil
t_1 \big)\notin C^\epsilon }\\
&\le \pi_{t_0}\uppar N(C)+\delta
\end{align*}
by the Markov property at time $t$ and~\eqref{eq:dontmovemuch}.
Since $\delta>0$ was arbitrary, it follows that
\begin{equation}
\limsup_{t\to \infty}\psub{\mathcal X}{X\uppar N(t)\in C}\le \pi_{t_0}\uppar N(C).
\label{closedCcvg}
\end{equation}
Comparing \eqref{closedCcvg} and \eqref{eq:corsspB}, we see that
$\pi_{t_1}\uppar N(C)\leq \pi_{t_0}\uppar N(C)$ for all closed sets
$C$ and all $t_1\in(0,t_0]$.  Hence, by the Portmanteau theorem, $\pi_{t_1}\uppar N = \pi_{t_0}\uppar N$ for all $t_1\in(0,t_0]$.  This proves \eqref{allmeaseq}, so we now write $\pi \uppar N$ 
for the unique invariant measure of the process $(X\uppar N(t),t\ge 0)$.

By the Markov property, we have that for any $\mathcal X_0\in (\R^d)^N$, $D\subseteq (\R^d)^N$ and $t>s>0$,
\begin{align*}
\P_{\mathcal X_0}\big(X\uppar N(t)\in D\big)
	&=\int_{(\R^d)^N} 
\P_{\mathcal X_0}\big(X\uppar N(s)\in \diffd\mathcal X\big)
\P_{\mathcal X}  \big(X\uppar N(t-s)\in D\big)
\\
\text{and }\qquad \qquad  \pi\uppar N(D) &= \int_{(\R^d)^N} \pi\uppar N(\diffd \mathcal X) 
\P_{\mathcal X}\big(X\uppar N(t-s)\in D \big).
\end{align*}
By taking the difference between these two equations, 
\begin{align} \label{eq:TNV}
\Big| \P_{\mathcal X_0}\big(X\uppar N(t)\in D\big)
-\pi\uppar N(D) \Big|
&\le \int_{(\R^d)^N}  \Big|\P_{\mathcal X_0}\big(X\uppar N(s)\in\diffd\mathcal X\big)-
\pi\uppar N(\diffd \mathcal X) \Big|
\P_{\mathcal X}\big(X\uppar N(t-s)\in D\big) \notag \\
&\le \int_{(\R^d)^N}  \Big|\P_{\mathcal X_0}\big(X\uppar N(s)\in\diffd\mathcal X\big)-
\pi\uppar N(\diffd \mathcal X) \Big|,
\end{align}
where the right hand side is the total variation norm of the difference between $\pi\uppar N$ and the law of $X\uppar N(s)$ under $\P_{\mathcal X_0}$.

Now choose $s=\lfloor t/t_0\rfloor t_0$ and let $t\to \infty$.
Since the law of $X\uppar N(nt_0)$ under
$\P_{\mathcal X_0}$ converges to $\pi \uppar N$
as $n\to\infty$ in total variation norm, the right hand side of
\eqref{eq:TNV} converges to zero as $t\to\infty$, and the result follows.
\end{proof}
\begin{proof}[Proof of Theorem~\ref{cor:ssp}]
Take $\epsilon>0$ and $A\subseteq \R^d$ measurable. Let 
$$D_\epsilon =\left\{\mathcal X\in (\R^d)^N:\left|\frac 1 N\sum_{i=1}^N
\indic{\mathcal X_i\in A}-\int_A U(x)\,\diffd x\right|\ge\epsilon
\text{ or }\Big|\max_{i\le N}\|\mathcal X_i\|-R_\infty \Big|\ge\epsilon\right\}.$$
Take the initial condition $\mathcal X=(0,\ldots,0) \in(\R^d)^N$.
By Theorem~\ref{thm:sspd}, for any $\delta\in(0,\epsilon)$, there exist $N_\delta$, $T_\delta<\infty $ such that for $N\geq N_\delta$ and $t\geq T_\delta $,
\[
\P_{\mathcal X}\big( X\uppar N (t)\in D_\epsilon\big)\le
\P_{\mathcal X}\big( X\uppar N (t)\in D_{\delta}\big)<2\delta.
\]
But by Theorem~\ref{thm:piNexists},
\[
\P_{\mathcal X}\big( X\uppar N (t)\in D_\epsilon \big)\to \pi\uppar N(D_\epsilon) \quad \text{as }t\to \infty.
\]
It follows that $\pi\uppar N(D_\epsilon)\le 2\delta$ for $N\ge N_\delta$, and so
$\lim_{N\to \infty}\pi\uppar N(D_\epsilon) =0$, which completes the proof.
\end{proof}

\bibliography{globalrefs} 
\bibliographystyle{alpha}

\end{document}